\documentclass[a4paper]{article}
\usepackage{amsmath,amsfonts,amssymb,mathrsfs,mathtools,mathdots}
\usepackage{ntheorem}
\usepackage{xcolor}
\usepackage{booktabs}
\usepackage{float}
\usepackage[a4paper,margin=2cm]{geometry}
\usepackage{listings}
\usepackage{MnSymbol}

\newtheorem{definition}{Definition}
\newtheorem{remark}{Remark}
\newtheorem{lemma}{Lemma}
\newtheorem{theorem}{Theorem}
\newtheorem{corollary}{Corollary}

\newtheorem{conj}{Conjecture}
\newtheorem{algo}{Algorithm}
\newtheorem{exmp}{Example}
\theoremstyle{plain}
\newtheorem*{proof}{Proof}

\newcommand{\E}{\mathrm{e}}
\definecolor{myblue}{rgb}{0.0,0.6056,0.9787}
\definecolor{myred}{rgb}{0.8889,0.4356,0.2781}
\definecolor{mypink}{rgb}{0.7644,0.4441,0.8242}

\newcommand{\qed}{\hfill \ensuremath{\Box}}

\newcommand{\f}{\mathbb}

\newcommand{\ol}{\overline}
\newcommand{\cu}{\subseteq}
\newcommand{\wt}{\widetilde}

\newcommand{\serie}[1]{\{#1_{n}\}_n}
\newcommand{\GLT}{\sim_{GLT}}

\newcommand{\ve}{\varepsilon}

\DeclareMathOperator{\diag}{diag}

\renewcommand{\epsilon}{\varepsilon}
\renewcommand{\aa}{{\boldsymbol a}}
\newcommand{\bb}{{\boldsymbol b}}

\newcommand{\ii}{{\boldsymbol i}}

\newcommand{\nn}{{\boldsymbol n}}

\newcommand{\xx}{{\boldsymbol x}}

\newcommand{\vv}{{\boldsymbol v}}

\newcommand{\bu}{{\mathbf 1}}

\newcommand{\II}{{\mathcal I}}

\title{Theoretical results for eigenvalues, singular values, and eigenvectors of (flipped) Toeplitz matrices and related computational proposals}
\date{\today}

\author{
Giovanni Barbarino\thanks{giovanni.barbarino@aalto.fi}\\ Department of Mathematics and Systems Analysis, Aalto University
\and
Sven-Erik Ekstr\"om\thanks{sven-erik.ekstrom@it.uu.se}\\ Division of Scientific Computing, Department of Information Technology, Uppsala University
\and 
Stefano Serra-Capizzano\thanks{s.serracapizzano@uninsubria.it} \\ Department of Science and High Technology,  Insubria University - Como \\ INdAM Research Unit at Department of Science and High Technology,  Insubria University - Como\\ Division of Scientific Computing, Department of Information Technology, Uppsala University
\and 
Paris Vassalos\thanks{pvassal@aueb.gr} \\ Department of Informatics, Athens University of Economics and Business}
\begin{document}

\maketitle
\begin{abstract}
In a series of recent papers the spectral behavior of the matrix sequence $\{Y_nT_n(f)\}$ is studied in the sense of the spectral distribution, where $Y_n$ is the main antidiagonal (or flip matrix) and $T_n(f)$ is the Toeplitz matrix generated by the function $f$, with $f$ being Lebesgue integrable and with real Fourier coefficients. This kind of study is also motivated by computational purposes for the solution of the related large linear systems using the (preconditioned) MINRES algorithm. Here we complement the spectral study with more results holding both asymptotically and for a fixed dimension $n$, and with regard to eigenvalues, singular values, and eigenvectors of $T_n(f), Y_nT_n(f)$ and to several relationships among them: beside fast linear solvers, a further target is the design of ad hoc procedures for the computation of the related spectra via matrix-less algorithms, with a cost being linear in the number of computed eigenvalues. We emphasize that the challenge of the case of non-monotone generating functions is considered in the current work, for which the previous matrix-less algorithms fail. Numerical experiments are reported and commented, with the aim of showing in a visual way the theoretical analysis.
\end{abstract}

\section{Introduction}
In a number of recent papers \cite{Ferrari2019,Hon2019,Mazza2018} the spectral behavior of the matrix-sequence  $\{Y_nT_n(f)\}$ is studied in the sense of the spectral distribution, where
\begin{equation*}
Y_n=\begin{bmatrix*}
&&&&1\\
&&&1\\
&&\iddots\\
&1\\
1
\end{bmatrix*}
\end{equation*}
is the main antidiagonal or flip matrix and $T_n(f)$ is the Toeplitz matrix generated by the symbol $f$, with $f$ being Lebesgue integrable and with real Fourier coefficients. Of course the singular values of $T_n(f)$ and $Y_n T_n(f)$ coincide, given the unitary character of the permutation matrix $Y_n$. This study has been complemented by the same type of analysis in the a multilevel context, where additional technical issues have been addressed \cite{ferrari2020,mazzapestana20,pestana}, taking into account the specific difficulties of the multilevel setting.

In this work we focus our attention on studying the  eigenvalues, singular values, and eigenvectors of $T_n(f)$ and of the resulting Hankel matrices $Y_nT_n(f)$, both asymptotically  and for a fixed dimension $n$.  In particular we study the spectral relationship among the Toeplitz matrix $T_n(f)$, the matrix $Y_nT_n(f)$, and the generating function $f$, and we furnish a more precise description of eigenvalues and eigenvectors of $Y_nT_n(f)$ than in the previous literature, using also quite old results on the eigenstructure of Toeplitz matrices \cite{Cantoni1976,Delsarte}.

The practical target relies in designing ad hoc procedures for the computation of the related spectra via matrix-less algorithms (see \cite{matrix-less1} and references therein), with a cost being linear in the number of computed eigenvalues. Here the novelty relies in considering non-monotone generating functions, for which the previous matrix-less procedures usually fail; see \cite{EFS,matrix-less1,EGS} and references therein. Furthermore, this type of study is also motivated by other computational purposes such as the solution of the related large linear systems, using the (preconditioned) MINRES algorithm (see \cite{Ferrari2019,ferrari2020,mazzapestana20,pestana} and references therein).

A careful selection of numerical tests is considered and the numerical experiments confirm the precise forecasts contained in the theoretical derivations.

The current work is organized as follows. In Section \ref{sec:prel} we introduce the basic notions and we set the notation. Section \ref{sec:main} contains the theoretical analysis, while related numerical experiments are discussed in Section \ref{sec:num}. Finally Section \ref{sec:end} is concerned with conclusions and open problems.

\section{Notation and Basic Notions}\label{sec:prel}

The present section is divided into four parts. In Subsection \ref{ssec:toe} we report the definition of Toeplitz matrices  and of the notion of generating function; Subsection \ref{ssec:hank} contains the analogous setting for Hankel matrices; finally Subsection   \ref{ssec:spectral} is devoted to the notions of eigenvalue and singular value distribution, while in Subsection \ref{ssec:aux} we state and prove preliminary results that will be used in the theoretical analysis.

\subsection{Toeplitz Matrices and Matrix-Sequences}\label{ssec:toe}

Let $f\in L^1(-\pi,\pi)$ and let $T_n(f)$ be the Toeplitz matrix generated by $f$, i.e.,
$\left(T_n(f)\right)_{s,t}=\hat f_{s-t}$, $s,t=1,\ldots,n$, with $f$ being the generating function of $\{T_n(f)\}$ and
with $\hat f_k$ being the $k$-th Fourier coefficient of $f$, that is,
\begin{equation}\label{fou}
\hat f_k=\frac{1}{2\pi}\int_{-\pi}^{\pi} f(\theta)\ e^{-\mathbf{i}k \theta}\, \mathrm{d}\theta,\ \ \ {\bf i}^2=-1, \ \ k\in \mathbb Z.
\end{equation}
If $f$ is real-valued then several spectral properties are known (localization, extremal behavior, collective distribution, see \cite{BS,Se1} and references therein) and $f$ is also the spectral symbol of $\{T_n(f)\}$ in the Weyl sense \cite{BS,GS,Ti1,Tyuni}.
If $f$ is complex-valued, then the same type of information is transferred to the singular values, while the eigenvalues can have a ``wild'' behavior~\cite{cluster} in some cases and a quite regular behavior in other cases \cite{Ti-complex}. More advanced material on distribution results are collected in the books on Generalized Locally Toeplitz matrix sequences \cite{GLT-book-I,GLT-book-II}.

\subsection{Hankel Matrices and Matrix-Sequences}\label{ssec:hank}

The standard definition~\cite[Section 1.4]{BS} of Hankel matrices generated by a function $f$ concerns the two matrices,
\begin{equation}
H_n^{(1)}(f)=[\hat{f}_{i+j-1}]_{i,j=1}^{n},\qquad
H_n^{(2)}(f)=[\hat{f}_{-(i+j-1)}]_{i,j=1}^{n},
\end{equation}
or
\begin{equation*}
H_n^{(1)}(f)=\begin{bmatrix*}
\hat{f}_1&\hat{f}_2&\hat{f}_3&\cdots\\
\hat{f}_2&\hat{f}_3&\ddots&\ddots\\
\hat{f}_3&\ddots&\ddots&\ddots\\
\vdots&\ddots&\ddots&\ddots
\end{bmatrix*},\qquad
H_n^{(2)}(f)=\begin{bmatrix*}
\hat{f}_{-1}&\hat{f}_{-2}&\hat{f}_{-3}&\cdots\\
\hat{f}_{-2}&\hat{f}_{-3}&\ddots&\ddots\\
\hat{f}_{-3}&\ddots&\ddots&\ddots\\
\vdots&\ddots&\ddots&\ddots
\end{bmatrix*},
\end{equation*}
with $\hat f_k$, $k\in \mathbb Z$, as in (\ref{fou}).

Here we treat a different setting and we define the Hankel matrix $H_n(f)$, generated by the function $f$, as
\begin{equation*}
H_n(f)=Y_nT_n(f)=H_n(1)T_n(f),
\end{equation*}
where $T_n(f)$ is the Toeplitz matrix generated by $f$ and
\begin{equation*}
Y_n=H_n(1)=\begin{bmatrix*}
&&&&1\\
&&&1\\
&&\iddots\\
&1\\
1
\end{bmatrix*},
\end{equation*}
is the antidiagonal or ``flip matrix'', of size $n$. The matrix $Y_n$ is a permutation matrix and hence it is unitary so that the singular values of $T_n(f)$ and $Y_nT_n(f)$ coincide, while for the eigenvalues there is a substantial (and computationally beneficial) change; see \cite{Ferrari2019,
ferrari2020,Hon2019,mazzapestana20,pestana} and references therein.

\subsection{Spectral and Singular Value Distributions}\label{ssec:spectral}

We now consider previous results concerning spectral distributions in the sense of Weyl. First we introduce some notations and definitions concerning general sequences of matrices. For any function $F$ defined on the complex field and
for any matrix $A_n$ of size $d_n$, by the symbol $\Sigma_{\lambda}(F,A_n)$, we denote the mean
\begin{equation*}
\Sigma_{\lambda}(F,A_n)=\frac{1}{d_n} \sum_{j=1}^{d_n}
F[\lambda_j(A_n)],
\end{equation*}
 while, by the symbol $\Sigma_{\sigma}(F,A_n)$, we denote the mean
\begin{equation*}
\Sigma_{\sigma}(F,A_n)=\frac{1}{d_n} \sum_{j=1}^{d_n}
F[\sigma_j(A_n)].
\end{equation*}
\begin{definition}
\label{def:eig-sv distr}
Given a sequence $\{A_n\}$ of matrices of size $d_n$ with $d_n<d_{n+1}$ and given a Lebesgue-measurable function $\psi$ defined over a measurable set $K\subset {\mathbb R}^{\nu}$, $\nu \in \mathbb N^+$, of finite and positive Lebesgue measure $\mu_{\nu}(K)$, we say that $\{A_n\}$ is distributed as $(\psi,K)$ in the sense of the eigenvalues if for any continuous $F$ with bounded support the following limit relation holds
\begin{equation}\label{distribution:eig}
\lim_{n\rightarrow \infty}\Sigma_{\lambda}(F,A_n)= \frac{1}{\mu_{\nu}(K)}\int_K F(\psi)\,\mathrm{d}\mu_{\nu}.
\end{equation}
In this case, we write in short $\{A_n\}\sim_{\lambda} (\psi,K)$. Furthermore we say that $\{A_n\}$ is distributed as $(\psi,K)$ in the sense of the singular values if for any continuous $F$ with bounded support the following limit relation holds
\begin{equation}
\label{distribution:sv}
\lim_{n\rightarrow \infty}\Sigma_{\sigma}(F,A_n)= \frac{1}{\mu_{\nu}(K)}\int_K F(|\psi|)\,\mathrm{d}\mu_{\nu}.
\end{equation}
In this case, we write in short $\{A_n\}\sim_{\sigma} (\psi,K)$, which is equivalent to $\{A_n^*A_n\}\sim_{\lambda} (|\psi|^2,K)$.

When the set $K$ is clear from the context, instead of $\{A_n\}\sim_{\lambda} (\psi,K)$, $\{A_n\}\sim_{\sigma} (\psi,K)$, we will write
$\{A_n\}\sim_{\lambda} \psi$, $\{A_n\}\sim_{\sigma} \psi$, respectively.
\end{definition}

In Remark~\ref{rem:meaning-distribution} we provide an informal meaning of the notion of eigenvalue distribution. For the singular value distribution similar statements can be written.

\begin{remark}\label{rem:meaning-distribution}
The informal meaning behind the above definition is the following. If $\psi$ is continuous, $n$ is large enough, and
\begin{equation*}
\left\{{\bf x}_j^{(d_n)},\ j=1,\ldots, d_n\right\}
\end{equation*}
is an equispaced grid on $K$, then a suitable ordering $\lambda_j(A_n)$, $j=1,\ldots,d_n$, of the eigenvalues of $A_n$ is such that the pairs $\left\{\left({\bf x}_j^{(d_n)},\lambda_j(A_n)\right),\ j=1,\ldots,d_n\right\}$ reconstruct approximately the hypersurface
\begin{equation*}
\{({\bf x},\psi({\bf x})),\ {\bf x}\in K\}.
\end{equation*}
 In other words, the spectrum of $A_n$ `behaves' like a uniform sampling of $\psi$ over $K$, up to few outliers.
For instance, if
$\nu=1$, $d_n=n$, and $K=[a,b]$, then the eigenvalues of $A_n$ are approximately equal to $\psi(a+j(b-a)/n)$, $j=1,\ldots,n$, for $n$ large enough and up to at most $o(n)$ outliers.
Analogously, if
$\nu=2$, $d_n=n^2$, and $K=[a_1,b_1]\times [a_2,b_2]$, then the eigenvalues of $A_n$ are approximately equal to $\psi(a_1+j(b_1-a_1)/n,a_2+k(b_2-a_2)/n)$, $j,k=1,\ldots,n$, for $n$ large enough and up to at most $o(n^2)$ outliers.
In general, when the symbol $\psi$ is smooth enough, the number of outliers reduce and can decrease to $O(1)$: for instance, for Hermitian Toeplitz matrix sequences having generating function real-valued a.e. and with the range being a unique interval, the number of outliers is simply zero.
\end{remark}

The asymptotic distribution of eigen and singular values of a sequence of Toeplitz matrices has been thoroughly studied in the last century (for example see \cite{BS,tyrtL1}\ and the references reported therein).
The starting point of this theory, which contains many extensions and other results, is a famous theorem of Szeg\H{o}~\cite{GS}, which we report in the
Tyrtyshnikov and Zamarashkin version~\cite{tyrtL1}.
\begin{theorem}
\label{teoszego-tyr}
If $f$ is integrable over $[-\pi,\pi]$, and if $\{T_n(f)\}$ is the sequence of
Toeplitz matrices generated by $f$, then
\begin{equation}
\label{tesityr}
\{ T_n(f)\}\sim_{\sigma} (f,[-\pi,\pi]).
\end{equation}
Moreover, if $f$ is also real-valued almost everywhere (a.e.), then each matrix $T_n(f)$
is Hermitian and
\begin{equation}
\label{tesiszego}
\{T_n(f)\}\sim_{\lambda} (f,[-\pi,\pi]).
\end{equation}
\end{theorem}
\ \\
On the other hand, if $f$ is real-valued a.e., then very precise localization results are known. In fact, in that case all the eigenvalues of $T_n(f)$ belong to the open interval $(m,M)$, where $m$ and $M$ are the essential infimum and the essential supremum of $f$, respectively, under the assumption that $f$ is not constant a.e. In the general case where $f$ is constant a.e., the result is trivial since $T_n(f)\equiv mI_n$ for every matrix order $n$, with $I_n$ being the identity matrix of size $n$ (see \cite{local2,Se1}). In any case, with regard to Remark \ref{rem:meaning-distribution}, in this setting we do not observe the presence of outliers.

First we introduce the notion of equal distribution regarding (at least) two  sequences of numerical sets of increasing cardinality. Then we state a selection of results which emphasize the relationships among equal distribution, uniform gridding, and spectral distribution of matrix-sequences (see also \cite{gridding}). Part of the related material is taken from \cite{SerraCapizzano2001} and will be used in our subsequent derivations.

\begin{definition}\label{def:equal distr}
Two sequences $\{X_n\}$ and $\{Y_n\}$ of numerical sets with $X_n=\left\{{ x}_j^{(d_n)},\ j=1,\ldots, d_n\right\}$ and $Y_n=\left\{{y}_j^{(d_n)},\ j=1,\ldots, d_n\right\}$ are equally distributed if for any continuous $F$ with bounded support the following limit relation holds
\begin{equation}\label{eq distribution}
\lim_{n\rightarrow \infty}{\frac 1 {d_n} \sum_{j=1}^{d_n}
F\left({ x}_j^{(d_n)}\right)-F\left({y}_j^{(d_n)}\right)}= 0.
\end{equation}
 In the case where the two sequences of sets $\{X_n\}$ and $\{Y_n\}$ are made up by the spectra of two sequences of matrices $\{A_n\}$ and $\{B_n\}$ we write that the two sequences of matrices are spectrally equally distributed, while the two sequences are equally distributed in the singular value sense if (\ref{eq distribution}) holds true and the two sequences of sets $\{X_n\}$ and $\{Y_n\}$ are made up by the sets of singular values of two sequences of matrices $\{A_n\}$ and $\{B_n\}$.
\end{definition}

\begin{remark}\label{rem:link-distribution}
Of course, by playing with the given definitions,  in the case where two sequences of matrices $\{A_n\}$ and $\{B_n\}$ are spectrally equally distributed, we have $\{A_n\}\sim_{\lambda} (\psi,K)$ if and only if $\{B_n\}\sim_{\lambda} (\psi,K)$. Furthermore,  in the case where two sequences of matrices $\{A_n\}$ and $\{B_n\}$ are equally distributed in the singular value sense, we have $\{A_n\}\sim_{\sigma} (\psi,K)$ if and only if $\{B_n\}\sim_{\sigma} (\psi,K)$.
\end{remark}

\begin{definition}\label{def:uniform gridding}
A grid of points $\{X_n\}$, $X_n=\left\{{ x}_j^{(d_n)},\ j=1,\ldots, d_n\right\}$, is asymptotically uniform (a.u.) in $[a,b]$ if and only if
$\{X_n\}$ and $\{U_n\}$ are equally distributed with $U_n=\left\{{ u}_j^{(d_n)}=a+(b-a)j/d_n,\ j=1,\ldots, d_n\right\}$.

More in general, a grid of points $\{X_n\}$, $X_n=\left\{{\bf x}_j^{(d_n)},\ j=1,\ldots, d_n\right\}$, is a.u. in a Peano-Jordan measurable set $K$, contained in ${\mathbb R}^d$ and of positive measure, if and only if for any $d$ dimensional rectangle $R$ contained in $K$
\begin{equation}\label{eq distribution-grid}
\lim_{n\rightarrow \infty}{\frac 1 {d_n} \sum_{j=1}^{d_n}
{\rm card}\left\{{\bf x}_j^{(d_n)}\in R\right\}}=\frac {\mu_{d}(R)} {\mu_d(K)},
\end{equation}
with $\mu_d$ being the Lebesgue measure on ${\mathbb R}^d$.
\end{definition}

If we assume that $X_n=\left\{{ x}_j^{(d_n)},\ j=1,\ldots, d_n\right\}$, is a.u. in $[a,b]$ and, as it is natural
$a\le { x}_1^{(d_n)}< { x}_2^{(d_n)}< \cdots < { x}_{d_n}^{(d_n)} \le b$, $j=1,\ldots, d_n$,  then
\begin{equation}\label{asym unif gridding}
\lim_{n\to\infty}\biggl(\max_{j=1,\ldots,d_n}\left\|{ x}_j^{(d_n)}-\Bigl(a+ j\,\frac{b-a}{d_n}\Bigr)\right\|_\infty\biggr)=0.
\end{equation}

\subsection{Auxiliary Results}\label{ssec:aux}

In the current subsection we first introduce and prove auxiliary results and then we collect known results from the literature. The presented theoretical tools are useful in the main theoretical derivations in Section \ref{sec:main}.

\begin{lemma}\label{lem:graph}
	Let $X$ be a finite set, and let $A_1,\dots,A_k$ and $B_1,\dots,B_k$ be two partitions of $X$ with $|A_i| = |B_i|$ for every $i$. Let $\mathscr G = (V,E)$ be a directed graph on $k$ nodes that has a directed edge $(i,j)\in E$ if and only if $A_i\cap B_j$ is not empty.
	If $(i,j)\in E$ then there exists a directed path from $j$ to $i$.
\end{lemma}
\begin{proof}
Suppose that $(i,j)\in E$ but that there does not exist a directed path from $j$ to $i$. As a consequence, the set of nodes
\[
N_i \coloneqq \{ \text{nodes with a direct path to }i  \},\quad
N^j \coloneqq \{ \text{nodes with a direct path from }j  \}
\]
are disjoint, where by convention we let $i\in N_i$, $j\in N^j$. Moreover, there is no edge from $N^j$ to $(N^j)^C$, so
\begin{align*}
\sum_{x\in N^j} |A_x|
&= \sum_{x\in N^j} \sum_{y\in V}|A_x\cap B_y| = \sum_{x\in N^j} \sum_{y\in N^j}|A_x\cap B_y|, \\
\sum_{y\in N^j} |B_y|
&= \sum_{y\in N^j} \sum_{x\in V}|A_x\cap B_y| \ge \sum_{x\in N^j} \sum_{y\in N^j}|A_x\cap B_y| +
|A_i\cap B_j|,
\end{align*}
that is a contradiction since $\sum_{x\in N^j} |A_x|   = \sum_{y\in N^j} |B_y| $ and $|A_i\cap B_j|>0$.

\qed
\end{proof}

\begin{lemma}\label{lem:splitting}
Let $\Lambda^{(n)} \coloneqq \{\lambda^{(n)}_1, \lambda^{(n)}_2, \dots, \lambda_{d_n}^{(n)} \}$ a sequence of $d_n$ real values for any $n\in\f N$ with $d_n\to\infty$, and $D_n\coloneqq\diag(\lambda^{(n)}_i)_{i=1,\dots,d_n}$.
Given a diagonal  $k\times k$ matrix-valued measurable function $H(x) \coloneqq \diag(f_j(x))_{j=1,\dots,k}$, where $f_j:[0,1]\to \f R$, suppose that $\{D_n\}_n\sim_\lambda H(x)$.
Then for any $n$
and for any sequence of integer numbers $L^{(n)}_j$ such that $L^{(n)}_j/d_n\to 1/k$ and $\sum_j L^{(n)}_j = d_n$,
 there exists a partition of $\Lambda^{(n)}$ into $k$ subset $\Lambda^{(n)}_1,\dots,\Lambda^{(n)}_k$ such that, for every $j=1,\dots,k$, we have
\begin{itemize}
	\item $L^{(n)}_j$ is
	the cardinality  of $\Lambda^{(n)}_j$,
	\item furthermore
	\begin{equation}\label{eq:splitting}
			\{D^{(j)}_n\}_n\coloneqq\left\{\diag\left(\lambda^{(n)}_{i}\right)_{\lambda^{(n)}_{i}\in \Lambda^{(n)}_j}\right\}_n \sim_\lambda f_j(x).
	\end{equation}
\end{itemize}	
Moreover, if $f_j$ are all Riemann integrable with connected range, and for any $n$,  $j\in \{1,\dots,k\}$ and $\lambda\in \wt \Lambda^{(n)}_j$
$$\min_{x\in [0,1]} f_j(x) - c_n \le \lambda \le \max_{x\in [0,1]} f_j(x) + c_n$$
holds for some $c_n\to 0$ and a partition of $\Lambda^{(n)}$ into $\wt \Lambda^{(n)}_j$ of cardinality $L^{(n)}_j$ satisfying $L^{(n)}_j/d_n\to 1/k$ and $\sum_j L^{(n)}_j = d_n$, then the
$\Lambda^{(n)}_j$ can be chosen so that \eqref{eq:splitting} holds and for any $n$,  $j\in \{1,\dots,k\}$ and $\lambda\in  \Lambda^{(n)}_j$
\begin{align*}
\min_{x\in [0,1]} f_j(x) - \wt c_n \le \lambda \le \max_{x\in [0,1]} f_j(x) +\wt  c_n
\end{align*}
for some $\wt c_n\to 0$.
\end{lemma}
\begin{proof}
From the hypothesis, $\{D_n\}_n\sim_\lambda H(x)$, that can be rewritten as $\{D_n\}_n\sim_\lambda h(x)$ where $h:[0,1]\to \f R$ is a concatenation of resized versions of $f_j(x)$. In particular, for any $j\in\{1,\dots,k\}$ and $x\in [0,1)$,
\[
h\left( \frac{j-1}{k} + \frac xk \right) \coloneqq f_j(x), \qquad h(1) = f_k(1).
\]
We can thus apply Theorem \ref{th:GLT_up_to_permutation}  and find that after a permutation $\tau_n$ of the diagonal elements $\wt D_n \coloneqq  P_nD_nP_n^T$, we have $\{\wt D_n\}_n \GLT h(x)$. By Lemma \ref{lem:reduced_GLT_split}, we conclude that
\begin{equation*}
\{D^{(j)}_n\}_n\coloneqq \left\{\diag\left(\lambda^{(n)}_{\tau_n\left(L^{(n)}_1 +\dots + L^{(n)}_{j-1} + i\right)} \right)_{i=
    1,\dots,L^{(n)}_j
}\right\}_n \sim_\lambda h(x)|_{[(j-1)/k,j/k]},
\end{equation*}
where $L^{(n)}_{j}$ are all integer numbers such that  $L^{(n)}_{j}/d_n \to 1/k$ for all $j=1,\dots,k$, and  $\sum_{j=1}^k L^{(n)}_{j} = d_n$.
Since $ h(x)|_{[(j-1)/k,j/k]}$ is a rearranged version of $f_j(x)$, then \eqref{eq:splitting} is proved with
\[
\Lambda^{(n)}_j \coloneqq \left\{
\lambda^{(n)}_{\tau_n\left(L^{(n)}_1 +\dots + L^{(n)}_{j-1} + i\right)}
\right\}_{i=1,\dots,L^{(n)}_j}.
\]
Suppose now that
$f_j$ are all Riemann integrable functions with connected range and  that for any $n$,  $j\in \{1,\dots,k\}$ and $\lambda\in \wt \Lambda^{(n)}_j$
$$\min_{x\in [0,1]} f_j(x) - c_n \le \lambda \le \max_{x\in [0,1]} f_j(x) + c_n$$
holds for some $c_n\to 0$ and a partition of $\Lambda^{(n)}$ into $\wt \Lambda^{(n)}_j$ of cardinality $L^{(n)}_j$.
Call $R^j$ the range of the function $f_j$, and $R^j_{\delta}$ its $\delta$ expansion. Notice that both of them are real intervals by hypothesis. We just proved that   $\{D^{(j)}_n\}_n\sim_\lambda f_j$, so we can apply Theorem \ref{th:su}
 and find a sequence of positive values $\wt c_n$ such that $c_n\le \wt c_n \to 0$ and for any $j$,
 \[
 |E^j|\coloneqq
 \left|\Lambda_j^{(n)} \cap ( R^j_{\wt c_n}  )^C \right|
= o(L^{(n)}_{j}) =  o(d_n).
\]
Fix now an element $x \in E^1$. Since $E^j\cu  ( R^j_{\wt c_n}  )^C \cu ( R^j_{ c_n}  )^C\cu ( \wt \Lambda^{(n)}_1  )^C$, then surely $x\not\in \wt \Lambda^{(n)}_1$ and $x\in \wt \Lambda^{(n)}_p$ for some $p\ne 1$. Notice that all hypotheses of Lemma \ref{lem:graph} hold for $A_j = \Lambda^{(n)}_j$ and $B_j = \wt \Lambda^{(n)}_j$, and moreover $(1,p)$ is an edge of the graph due to the element $x$. As a consequence, there must be a directed path from $p$ to $1$, meaning that there are distinct indexes $i_0 = 1, i_1 = p, i_2, i_3, \dots, i_q$ and relative elements $x_0 = x, x_1, x_2, x_3, \dots, x_q$ s.t.\
\[
x_s \in \Lambda^{(n)}_{i_s} \cap \wt \Lambda^{(n)}_{i_{s+1}}, \quad s = 0,1,\dots, q-1,
\qquad x_q \in \Lambda^{(n)}_{i_q} \cap \wt \Lambda^{(n)}_{i_{0}}
\]
As a consequence, we can produce a new partition $\ol \Lambda^{(n)}_j$ of $\Lambda^{(n)}$ with the same cardinalities $L^{(n)}_j$ by removing $x_s$ from $\Lambda^{(n)}_{i_s}$ and adding it to $\Lambda^{(n)}_{i_{s+1}}$ for each $s$, with the convention $i_{q+1} = i_0$. Notice that $x_s\in \wt \Lambda^{(n)}_{i_{s+1}} \cu R^{i_{s+1}}_{ c_n}$, so if now
$|\ol E^j|\coloneqq
\left|\ol \Lambda_j^{(n)} \cap ( R^j_{\wt c_n}  )^C \right|$ we find that $\ol E^j\cu  E^j$ for every $j$ and $|\ol E^1| = |E^1|-1$. We can thus repeat the same procedure for some other element of $\cup_{j} \ol E_j$ iteratively until they are all empty.
With an abuse of notation, let $\Lambda^{(n)}_j$ be the partition generated by the whole procedure, and notice that $\Lambda^{(n)}_j$ differs from the starting partition by at most $k\sum_j|E^j| = o(d_n)$ elements. If
\begin{equation*}
\{ \wt D^{(j)}_n\}_n\coloneqq\left\{\diag\left(\lambda^{(n)}_{i} \right)_{
\lambda^{(n)}_{i} \in \Lambda^{(n)}_j
}\right\}_n,
\end{equation*}
then the difference with the matrices $D^{(j)}_n$, up to an opportune permutation, is of rank $o(d_n)$, so by Corollary 5.2 of \cite{GLT-book-I}, one finds that $\{ \wt D^{(j)}_n\}_n\sim_\lambda f_j$ and that for any $n$,  $j\in \{1,\dots,k\}$ and $\lambda\in  \Lambda^{(n)}_j$,
$$\min_{x\in [0,1]} f_j(x) -\wt c_n \le \lambda \le \max_{x\in [0,1]} f_j(x) + \wt c_n$$
by construction.

\qed
\end{proof}

\begin{lemma}\label{lem:au_grids_up_to_o(n)}
Suppose $\mathcal G_n=\{\xi_{i,n}\}_{i=1,\dots,d_n}$ is an a.u.\ grid on $[0,1]$ with $d_n\to \infty$. If $\mathcal G'_n=\{\xi'_{i,n}\}_{i=1,\dots,d'_n}$ is still a grid on $[0,1]$ with $|\mathcal G_n \triangle \mathcal G'_n| = o(d_n)$, then $\mathcal G'_n$ is still a.u.\ on $[0,1]$. Here $\triangle$ is the symmetric difference between sets.
\end{lemma}

\begin{proof}
Recall that by definition $\mathcal G_n=\{\xi_{i,n}\}_{i=1,\dots,d_n}$ is an a.u.\ grid on $[0,1]$ when
$ m_n\coloneqq
\max_{i=1,\dots,d_n} \left| \xi_{i,n} - i/d_n  \right| \to 0$. Suppose now that $\xi_{i,n}$ and $\xi'_{i,n}$ are sorted in increasing order. Moreover, let $\xi'_{j_i,n} = \xi_{i,n}$ for every $\xi_{i,n} \in  \mathcal G_n\cap \mathcal G'_n$, where also the indices $j_i$ are sorted in increasing order. Call now $c_n \coloneqq |\mathcal G_n \triangle \mathcal G'_n|= o(d_n)$, that can be seen as the number of elements removed from $\mathcal G_n$  plus those added to it,  in order to obtain $\mathcal G'_n$. Under this optic, it is easy to see that  $|d_n-d'_n| \le c_n$, but also $|i-j_i| \le c_n$ for every $\xi_{i,n} \in  \mathcal G_n\cap \mathcal G'_n$. As a consequence,
\[
\left| \frac{d'_n}{d_n}  -1  \right| \le \frac{c_n}{d_n} \coloneqq e_n\to 0, \qquad
\left| \frac{d_n}{d'_n}  -1  \right|
=
\left| \frac{1}{\frac{d_n}{d_n-d'_n}-1}    \right|
\le
\frac{1}{\frac{1}{e_n}+1}\eqqcolon r_n\to 0,
\]
and thus
\begin{equation}\label{eq:au_grid}
    \left| \xi'_{j_i,n} - \frac{j_i}{d'_n}  \right|
    \le
    \left| \xi_{i,n} - \frac{i}{d_n}  \right| +
    \left| \frac{i}{d_n} - \frac{i}{d'_n}  \right| +
    \left| \frac{i}{d'_n} - \frac{j_i}{d'_n}  \right|
    \le
    \left| \xi_{i,n} - \frac{i}{d_n}  \right| +
    2 \frac{c_n}{d_n} \frac {d_n}{d'_n}
    \le
    m_n + 2 e_n \left( 1 +  r_n \right).
\end{equation}
If now $j\ne j_i$ for any $i$, then let $\ol i, \underline{i}$ be the indices such that $\xi_{\ol i,n} = \xi'_{j_{\ol i},n}$ and $\xi_{\underline i,n} = \xi'_{j_{\underline i},n}$ are the closest possible to $\xi'_j$ with $j_{\underline i}\le j\le j_{\ol i}$ and the convention that $\xi'_{j_{\underline i},n}= 0, j_{\underline i}=0, \underline{i} = 0$ and
$\xi'_{j_{\ol i},n}=1, j_{\ol i}=d'_n+1, \ol i= d_n+1$ if they do not exist. In this case, surely $|j -j_{\ol i}| \le c_n$ and $|j -j_{\underline i}| \le c_n$ because otherwise there would be a  $\xi'_{j_{i^*},n}$ closer to  $\xi'_j$ than $\xi'_{j_{\ol i},n}$ or $\xi'_{j_{\underline i},n}$. Moreover, we have that $|\ol i - \underline i| \le c_n+1$, so thanks to \eqref{eq:au_grid} we can  write
\begin{align*}
\left|\xi'_{j_{\underline i},n} - \xi'_{j_{\ol i},n} \right| &\le
\left|\xi_{\underline i,n} - \frac{\underline i}{d_n} \right|
+
\left|\frac{\underline i}{d_n} - \frac{\ol i}{d_n}\right|
+
\left|\xi_{\ol i,n} - \frac{\ol i}{d_n}\right|
\le
2m_n+ \frac 3{d_n}+ e_n,
\\
    \left| \xi'_{j,n} - \frac{j}{d'_n}  \right|
    &\le
    \left|\xi'_{j_{\underline i},n} - \xi'_{j_{\ol i},n} \right|+
    \left| \xi'_{j_{\ol i},n} -  \frac{j_{\ol i}}{d'_n}  \right| +
    \left|\frac{j_{\ol i}}{d'_n} - \frac{j}{d'_n}  \right|
    \le
    3m_n+ \frac 3{d_n}+ e_n
    +  3 e_n \left( 1 +  r_n \right),
\end{align*}
thus proving that concluding that $\mathcal G'_n$ is an a.u.\ grid, since
\[
\max_{j=1,\dots,d'_n}   \left| \xi'_{j,n} - \frac{j}{d'_n}  \right|
\le 3m_n+ \frac 3{d_n}+ e_n
+  3 e_n \left( 1 +  r_n \right) \to 0.
\]

\qed
\end{proof}

\begin{lemma}\label{lem:union_of_au_grids_au}
	Suppose $\mathcal G^j_n=\{\xi^j_{i,n}\}_{i=1,\dots,d^j_n}$ are two a.u.\ grids on $[0,1]$ with $d^j_n\to \infty$ for $j=1,2$. If $d^j_n/n\to 1/2$ for $j=1,2$, then $\mathcal G^1_n\cup \mathcal G^2_n$ is still a.u.\ on $[0,1]$.
\end{lemma}
\begin{proof}
For this, let $\mathcal G^j_n \coloneqq \{ \xi^j_{i,n}  \}_{i=1,\dots,d^j_n}$  and
$\mathcal G_n \coloneqq \{ \xi_{i,n}  \}_{i=1,\dots,n}$, where all elements are sorted in increasing order, and call $\xi^j_{i,n} = \xi_{a_j(i), n}$. Let
\[
c_n \coloneqq \max_{i=1,\dots,d_n^1}
\left| \xi^1_{i,n} - \frac {i }{d_n^1}\right|
+
\max_{i=1,\dots,d_n^2}
\left| \xi^2_{i,n} - \frac {i }{d_n^2}\right|
\to 0
\]
and notice that $|2d_n^j/n - 1| \le 2/n$.
Fix an index $i$ and suppose $i = a_1(p)$, meaning  $\xi_{i,n} = \xi^1_{p,n}\in \mathcal G^1_n$, and let $j<i$ be the biggest index such that $j=a_2(q)$ (or $j=0$ and $q=0$, $\xi_{0,n}=0$ if there is none). As a consequence, $i = p+q$ and moreover $a_2(q+1)>i$ (where $a_2(d_n^2+1) = n+1$, and the respective point $\xi_{a_2(d_n^2+1),n}=1$), so that $\xi_{i,n}= \xi^1_{p,n}$  stands between  $\xi^2_{q,n}$ and $\xi^2_{q+1,n}$. As a consequence,
\begin{align*}
\left|
\frac{ q}{d_n^2} - \frac{ p}{d_n^1}\right|&\le
\left|
\frac{ q}{d_n^2} - \xi^2_{q,n}\right|
+
\left|
\xi^2_{q,n} - \xi^1_{p,n}\right|
+
\left|
\xi^1_{p,n} - \frac{ p}{d_n^1}\right|\le
2c_n
+
\left|
\xi^2_{q,n} - \xi^2_{q+1,n}\right|\\
&\le
2c_n
+
\left|
\xi^2_{q,n} - \frac{ q}{d_n^2}\right|
+
\left|
\frac{ q}{d_n^2} - \frac{ q+1}{d_n^2}\right|
+
\left|
\frac{ q+1}{d_n^2} - \xi^2_{q+1,n}\right|\\
&\le 4c_n +
\frac{1 }{d_n^2},
\\
\left|
\frac{ p+q}n - \frac{ p}{d_n^1}\right| &
\le
\frac{ p}{d_n^1}
\left|
\frac{d_n^1}n - \frac 12\right| +
\frac{ q}{d_n^2}
\left|
\frac{d_n^2}n - \frac{1}{2}\right| +
\frac 12
\left|
\frac{ q}{d_n^2} - \frac{ p}{d_n^1}\right|\\
&\le
\frac {2 } n
+
2c_n+
\frac{1 }{2d_n^2}, \\
\left| \xi_{i,n} - \frac {i }n\right| &\le
\left| \xi^1_{p,n} - \frac { p}{d_n^1}\right|
+
\left| \frac { p}{d_n^1} - \frac { p+q}n\right|
\le 3c_n + \frac {2 } n
+
\frac{1 }{2d_n^2} \to 0.
\end{align*}
The same bound with $d_n^1$ instead of $d_n^2$ applies in the case $\xi_{i,n} \in \mathcal G_n^2$, so this is enough to prove that $\mathcal G_n$ is an a.u. grid on $[0,1]$.

\qed
\end{proof}

\begin{remark}
	Lemma \ref{lem:union_of_au_grids_au} holds also without the hypothesis $d^j_n/n\to 1/2$.
\end{remark}

We now collect further useful results from the quoted literature.\\

\noindent\textbf{Monotone rearrangement} (see \cite{EM22} and references therein).
Let $f:\Omega\subset\mathbb R^d\to\mathbb R$ be measurable on a set $\Omega$ with $0<\mu_d(\Omega)<\infty$. The monotone rearrangement of $f$ is the function denoted by $f^\dag$ and defined as follows:
\begin{equation}\label{crv}
f^\dag:(0,1)\to\mathbb R,\qquad f^\dag(y)=\inf\biggl\{u\in\mathbb R:\frac{\mu_d\{f\le u\}}{\mu_d(\Omega)}\ge y\biggr\}.
\end{equation}
If $f$ is continuous and bounded, then $f^\dag$ ia also defined on $\{0,1\}$ as
\[
f^\dag(0) = \inf_{x\in \Omega} f(x), \qquad f^\dag(1) = \sup_{x\in \Omega} f(x).
\]

\begin{theorem}[Cantoni-Butler \cite{Cantoni1976}]\label{th:Cantoni-Butler}
For any real $f\in L^1[-\pi,\pi]$,
$$\lambda_i(H_n(f)) = (-1)^{i+1} \lambda_i(T_n(f)),$$
where the order of the eigenvalues is not specified.
\end{theorem}

\noindent\textbf{Regular sets.}
We say that $\Omega\subset\mathbb R^d$ is a regular set if it is bounded and $\mu_d(\partial\Omega)=0$. \\

If $\aa,\bb\in\mathbb R^d$ with $\aa\le\bb$, then we denote by $(\aa,\bb]$ the $d$-dimensional rectangle $(a_1,b_1]\times\cdots\times(a_d,b_d]$. Similar meanings have the notations for the open $d$-dimensional rectangle $(\aa,\bb)$ and the closed $d$-dimensional rectangle $[\aa,\bb]$.
Let $[\aa,\bb]$ be a $d$-dimensional rectangle, let $\nn=(n_1,\ldots,n_d)\in\mathbb N^d$, and let $\mathcal G_\nn=\{\xx_{\ii,\nn}\}_{\ii=\bu,\ldots,\nn}$ be a sequence of $d_\nn=n_1n_2\cdots n_d$ grid points in $\mathbb R^d$. We say that the grid $\mathcal G_\nn$ is a.u.\ in $[\aa,\bb]$ if
\[ \lim_{\nn\to\infty}\biggl(\max_{\ii=\bu,\ldots,\nn}\left\|\xx_{\ii,\nn}-\Bigl(\aa+\ii\,\frac{\bb-\aa}\nn\Bigr)\right\|_\infty\biggr)=0, \]
where $\|\xx\|_\infty=\max(|x_1|,\ldots,|x_d|)$ for every $\xx\in\mathbb R^d$.
Notice that the former is a generalization of the relation in (\ref{asym unif gridding}) for $d=1$ and is in line with Definition \ref{def:uniform gridding}.

\begin{theorem}[Theorem 3.1, \cite{EM22}]\label{th:rearr_main}
Let $f:\Omega\subset\mathbb R^d\to\mathbb R$ be continuous a.e.\ on the regular set $\Omega$ with $\mu_d(\Omega)>0$. 
Take any $d$-dimensional rectangle $[\aa,\bb]$ containing $\Omega$ and any a.u.\ grid $\mathcal G_\nn=\{\xx_{\ii,\nn}\}_{\ii=\bu,\ldots,\nn}$ in $[\aa,\bb]$.
For each $\nn\in\mathbb N^d$, consider the samples
\begin{equation*}
f(\xx_{\ii,\nn}),\qquad\ii\in\II_\nn(\Omega)=\{\ii\in\{\bu,\ldots,\nn\}:\xx_{\ii,\nn}\in\Omega\},
\end{equation*}
sort them in non-decreasing order, and put them into a vector $(s_0,\ldots,s_{\omega(\nn)})$, where $\omega(\nn)=\#\II_\nn(\Omega)-1$.
Let $f^\dag_\nn:[0,1]\to\mathbb R$ be the linear spline function that interpolates the samples $(s_0,\ldots,s_{\omega(\nn)})$ over the equally spaced nodes $(0,\frac{1}{\omega(\nn)},\frac{2}{\omega(\nn)},\ldots,1)$ in $[0,1]$. Then,
\[ \lim_{\nn\to\infty} f^\dag_\nn(y)=f^\dag(y) \]
for every continuity point $y$ of $f^\dag$. In particular, $f^\dag_\nn\to f^\dag$ a.e.\ in $(0,1)$.
\end{theorem}

\begin{lemma}[Lemma 3.3, \cite{EM22}]\label{lem:rearr_main}
Let $\omega_n$ be a sequence of positive integers such that $\omega_n\to\infty$ and let $g_n:[0,1]\to\mathbb R$ be a sequence of non-decreasing functions such that
\[ \lim_{n\to\infty}\frac1{\omega_n}\sum_{\ell=0}^{\omega_n}F\Bigl(g_n\Bigl(\frac\ell{\omega_n}\Bigr)\Bigr)=\int_0^1F(g(y)){\rm d}y,\qquad\forall\,F\in C_c(\mathbb R), \]
where $g:(0,1)\to\mathbb R$ is non-decreasing. Then, $g_n(y)\to g(y)$ for every continuity point $y$ of $g$. 
\end{lemma}

\begin{theorem}[Theorem 2, \cite{barbarinoDM}]\label{th:GLT_up_to_permutation}
Given a matrix sequence of diagonal matrices $\{D_n\}_n\sim_\lambda f(x)$ where $f:[0,1]\to \f C$ is a measurable function, then
\[
\{P_nD_nP_n^T\}_n \GLT f(x)
\]
for some $P_n$ permutation matrices.
\end{theorem}

\begin{lemma}[Lemma 5.1, Lemma 4.9 \cite{barbarinoREDUCED}]\label{lem:reduced_GLT_split}
Let $\lambda^{(n)}_1, \lambda^{(n)}_2, \dots, \lambda^{(n)}_{n}$ a sequence of $n$ real values for any $n\in\f N$, and $D_n\coloneqq\diag(\lambda^{(n)}_i)_{i=1,\dots,n}$.
If $\{D_n\}_n\GLT h(x)$ where $h:[0,1]\to \f C$ is a measurable function, then for any $j=1,\dots,k$,
\begin{equation*}
\{D^{(j)}_n\}_n\coloneqq\left\{\diag\left(\lambda^{(n)}_{L^{(n)}_1 +\dots + L^{(n)}_{j-1} + i} \right)_{i=
	1,\dots,L^{(n)}_j
}\right\}_n \sim_\lambda h(x)|_{[(j-1)/k,j/k]},
\end{equation*}
where $L^{(n)}_{j}$ are all integer numbers such that
\begin{itemize}
	\item $L^{(n)}_{j}/n \to 1/k$ for all $j=1,\dots,k$,
	\item $\sum_{j=1}^k L^{(n)}_{j} = n$.
\end{itemize}
\end{lemma}

\begin{theorem}[Theorem 3.1, \cite{GLT-book-I}]\label{th:su}
Let $\serie A\sim_\lambda f$ for $d_n\times d_n$ matrices $A_n$ and some measurable function $f:D\to \f C$. Let $R^f$ be the range of $f$ and $R^f_\epsilon$ its $\epsilon$-expansion, that is $R^f_\epsilon = \cup_{x\in R^f}\{y\in \f C : |y-x|\le \epsilon  \}$. If $\epsilon >0$, then
\[
\left| \{ j\in\{1,\dots,n\} : \lambda_j(A_n)\not\in R^f_\epsilon  \}   \right| = o(d_n).
\]
\end{theorem}

The following is sometimes referred to as the Dini second theorem \cite[pp.~81 and 270, Problem~127]{Polya-Szego}.

\begin{lemma}\label{Dini}
If a sequence of monotone functions converges pointwise on a compact interval to a continuous function, then it converges uniformly.
\end{lemma}

\section{Eigenstructure of Flipped Toeplitz matrices}\label{sec:main}

By combining old and recent results, including those in Subsection \ref{ssec:aux}, we describe specific properties related to the eigenstructure of flipped Toeplitz matrices. We start by providing the eigenstructure of $Y_n=H_n(1)$. Then the rest of the section is divided into three subsections.
Subsection \ref{ssec: real symm} treats eigenvalues and eigenvectors of $H_n(f)$ in the case where $f$ is even and real-valued (which corresponds to real Fourier coefficients with $\hat f_{k}=\hat f_{-k}$, for any integer $k$). Subsection \ref{ssec: general} contains general results on the spectral distribution of matrix sequences, not necessarily of structured type. Finally Subsection \ref{ssec: real nonsymm} treats eigenvalues and eigenvectors of $H_n(f)$ in the case where $f$ is complex-valued and the Fourier coefficients are still real.

First we begin with an algebraic study, which relies on the Cantoni-Butler Theorem \ref{th:Cantoni-Butler}.
A vector $\vv\in\mathbb R^n$ is called symmetric if $Y_n\vv=\vv$ and skew-symmetric if $Y_n\vv=-\vv$.
An $n\times n$ matrix $A$ is called centrosymmetric if it is symmetric with respect to its center, i.e.,
\begin{equation}\label{centrosym}
A_{ij}=A_{n-i+1,n-j+1},\qquad i,j=1,\ldots,n.
\end{equation} 
Equivalently, $A$ is centrosymmetric if
\[ Y_nAY_n = A. \]
Note that any symmetric Toeplitz matrix is centrosymmetric.

One eigendecomposition of the antidiagonal $H_n(1)$, which is centrosymmetric according to the relation in (\ref{centrosym}), is described as follows and its verification is a direct check
\begin{equation*}
H_n(1)=\mathbb{S}_n \mathbb{H}_n \mathbb{S}_n,
\end{equation*}
where $\mathbb{H}_n$ is a diagonal matrix
\begin{equation*}
\mathbb{H}_n=\left[
\begin{array}{rrrrrrrrrrr}
1\\
&-1\\
&&1\\
&&&-1\\
&&&&\ddots\\
&&&&&(-1)^{n+1}
\end{array}
\right],
\end{equation*}
that is, $(\mathbb{H}_n)_{i,i}=(-1)^{i+1}$ and $\mathbb{S}_n$ is the unitary discrete sine transform
\begin{equation}\label{S_n}
\begin{split}
\mathbb{S}_n&=\sqrt{\frac{2}{n+1}}\left(\sin\left(\frac{ij\pi}{n+1}\right)\right)_{i,j=1}^{n}\\
&=[\mathbf{v}_1^{(n)}, \mathbf{v}_2^{(n)} ,\ldots, \mathbf{v}_n^{(n)}],
\end{split}
\end{equation}
where $\mathbf{v}_j^{(n)}$ is the $j$th column,
\begin{equation}\label{S_n columns}
\mathbf{v}_j^{(n)}=\sqrt{\frac{2}{n+1}}\left[
\begin{array}{c}
\sin(j\pi/(n+1))\\
\sin(2j\pi/(n+1))\\
\vdots\\
\sin(nj\pi/(n+1))
\end{array}
\right]
\end{equation}
and $H_n(1)\mathbf{v}_i^{(n)}=(-1)^{i+1}\mathbf{v}_i^{(n)}$, $i=1,\ldots,n$. Of course $H_n(1)=Y_n$, but it is interesting to show that he latter type of relation holds any real symmetric Toeplitz matrix $T_n$ and for its flipped counterpart $H_n=Y_n T_n$.

\begin{theorem}\label{cbt}
Let $T_n$ be a real symmetric Toeplitz matrix of size $n$ and let $H_n=Y_nT_n$. Then, the following properties hold.
\begin{enumerate}
\item There exists an orthonormal basis of $\mathbb R^n$ consisting of eigenvectors of $T_n$ such that $\lceil n/2\rceil$ vectors of this basis are symmetric and the other $\lfloor n/2\rfloor$ vectors are skew-symmetric.
\item Let $\{\vv_1,\ldots,\vv_n\}$ be a basis of $\mathbb R^n$ such that:
	\begin{itemize}
		\item $T_n\vv_i=\lambda_i(T_n)\vv_i$ for $i=1,\ldots,n$;
		\item $\vv_1$ is symmetric, $\vv_2$ is skew-symmetric, $\vv_3$ is symmetric, and so on until $\vv_n$, which is either symmetric or skew-symmetric depending on whether $n$ is odd or even.
	\end{itemize}
	Then, the eigenpairs of $H_n$ are given by
	\[ (\lambda_i(H_n),\vv_i),\qquad i=1,\ldots,n, \]
	with
	\[ \lambda_i(H_n)=(-1)^{i+1}\lambda_i(T_n),\qquad i=1,\ldots,n. \]
\end{enumerate}
\end{theorem}
\begin{proof} \ \\
\begin{itemize}
\item The matrix $T_n$ is symmetric centrosymmetric. Hence, the result follows from Theorem \ref{th:Cantoni-Butler}.

\item  Since $\vv_i$ is alternatively symmetric and skew-symmetric (starting with symmetric), for $i=1,\ldots,n$ we have
\[
H_n\vv_i=Y_nT_n\vv_i=Y_n\lambda_i(T_n)\vv_i=\lambda_i(T_n)Y_n\vv_i=\lambda_i(T_n)(-1)^{i+1}\vv_i. 
\]
\end{itemize}

\qed
\end{proof}

\subsection{Real Symmetric Case}\label{ssec: real symm}

The current subsection contains three theorems of increasing generality, regarding the relationships among the eigenvalues of $H_n(f)=Y_n T_n(f)$, the eigenvalues of $T_n(f)$, the evaluations of the generating function $f$ on a a.u. grid, at least when the generating function $f$ is real-valued, even, and Riemann integrable.

\begin{theorem}\label{th:Carl}
Let $f:[-\pi,\pi]\to \f R$ be a real even continuous function which is positive and strictly monotone
increasing on $[0,\pi]$. Then there exists an a.u.\ grid $\{\xi_{1,n}, \xi_{2,n}, \dots, \xi_{n,n}\}_n$ such that
the eigenvalues of $T_n(f)$ and $H_n(f)$ are given by
\begin{align*}
\lambda_i(T_n(f)) &= f(\xi_{i,n}),\\
\lambda_i(H_n(f)) &= (-1)^{i+1} \lambda_i(T_n(f)),
\end{align*}
for all $i=1,\dots,n$.
\end{theorem}

\begin{proof}
By Cantoni-Butler Theorem \ref{th:Cantoni-Butler}, there exists an ordering of the eigenvalues of $T_n(f)$ and $H_n(f)$ such that  $\lambda_i(H_n(f)) = (-1)^{1+1}\lambda_{i}(T_n(f))$. Moreover, since $f$ is strictly increasing on $[0,\pi]$ and $\Lambda(T_n(f))\cu\text{Range}(f)$, then for every $i$ there exists an unique point $\wt \xi_{i,n}$ in $[0,\pi]$ such that  $\lambda_i(T_n(f))=f(\wt \xi_{i,n})$, and we just need to prove that they form an a.u.\ grid, that is
    \begin{equation}\label{to_prove}
    \max_{i=1,\ldots,n}|\xi_{i,n}-\theta_{i,n}|\to0\,\mbox{ as }\,n\to\infty,
    \end{equation}
where $\theta_{i,n} = i\pi /n$, and the $\xi_{i,n}$ are just the $\wt \xi_{i,n}$ sorted in an increasing manner.

Suppose by contradiction that \eqref{to_prove} is not satisfied. Then, we have
\[ \max_{i=1,\ldots,n}|\xi_{i,n}-\theta_{i,n}|\ge\epsilon \]
infinitely often (i.o.)\ for some fixed $\epsilon>0$. Hence, there exists a sequence $\{\xi_{i(n),n}\}_n$ such that
\[ |\xi_{i(n),n}-\theta_{i(n),n}|\ge\epsilon\,\mbox{ i.o.} \]
There are two possible (mutually non-exclusive) cases.

    \medskip

    \noindent{\em Case 1: $\xi_{i(n),n}-\theta_{i(n),n}\ge\epsilon$ i.o.}
    Take a subsequence $\{\xi_{i(m),m}-\theta_{i(m),m}\}_{m}$ of $\{\xi_{i(n),n}-\theta_{i(n),n}\}_n$ such that
    \begin{itemize}
        \item $\xi_{i(m),m}-\theta_{i(m),m}\ge\epsilon\,\mbox{ for all }\,m$,
        \item $\xi_{i(m),m}\to\xi, \quad\xi\in[0,\pi], \qquad \xi_{i(m),m} > \xi - \epsilon/2\,\mbox{ for all }\,m$,
        \item $\theta_{i(m),m}\to\theta,\quad\theta\in[0,\pi]$.
    \end{itemize}
    
    In particular, we find that $\xi \ge\theta + \epsilon$.
    By \cite{Ferrari2019,Hon2019,Mazza2018}, we have
    \begin{equation}\label{...lambda...}
    \frac1n\sum_{i=1}^mF((-1)^{i+1}f(\xi_{i,m}))=\frac1m\sum_{i=1}^mF(\lambda_i(H_m(f)))\to\frac1{2\pi}\int_0^\pi F(f(\varphi)){\rm d}\varphi+\frac1{2\pi}\int_0^\pi F(-f(\varphi)){\rm d}\varphi
    \end{equation}
    for all bounded functions $F:\mathbb R\to\mathbb R$ with at most a finite number of discontinuities (recall that $f$ is strictly monotone increasing, which implies that the sets $\{f=a\}$ and $\{-f=a\}$ have zero measure for all $a\in\mathbb R$). By choosing $F=\chi_{(-\infty,f(\xi - \epsilon/2)]}$ in \eqref{...lambda...} and keeping in mind that $f$ is positive and strictly monotone increasing on $[0,\pi]$, we obtain
    \begin{align}\nonumber
    &\frac1m\sum_{i=1}^m\chi_{(-\infty,f(\xi- \epsilon/2)]}((-1)^{i+1}f(\xi_{i,n}))\to\frac1{2\pi}\int_0^\pi \chi_{(-\infty,f(\xi- \epsilon/2)]}(f(\varphi)){\rm d}\varphi+\frac1{2\pi}\int_0^\pi\chi_{(-\infty,f(\xi- \epsilon/2)]}(-f(\varphi)){\rm d}\varphi\\\nonumber
    &\iff\frac{\#\{i\in\{1,\ldots,m\}:i\mbox{ is even }\,\vee\,i\mbox{ is odd and }f(\xi_{i,m})\le f(\xi-\epsilon/2)\}}m\to\frac{\mu_1\{f\le f(\xi-\epsilon/2))\}}{2\pi}+\frac12\\\nonumber
    &\iff\frac{\lfloor m/2\rfloor}{m}+\frac{\#\{i\mbox{ is odd and }\xi_{i,m}\le \xi-\epsilon/2\}}{m}\to\frac{\xi-\epsilon/2}{2\pi}+\frac12\\
    &\iff\frac{\#\{i\mbox{ is odd and }\xi_{i,m}\le \xi-\epsilon/2\}}{m}\to\frac{\xi-\epsilon/2}{2\pi}.
    \label{eq:last}
    \end{align}
    But now
    \begin{align*}
        \frac{\#\{i\mbox{ is odd and }\xi_{i,m}\le \xi-\epsilon/2\}}{m} &\le \frac{\#\{i\mbox{ is odd and }\xi_{i,m}\le  \xi_{i(m),m}   \}}{m}
        \le
        \frac 1m + \frac{i(m)}{2m}
        =
        \frac 1m + \frac{\theta_{i(m),m }}{2\pi}
        \to \frac{\theta}{2\pi}
        \le \frac{\xi - \epsilon}{2\pi}
    \end{align*}
    that contradicts \eqref{eq:last}.
    \medskip

    \noindent{\em Case 2: $\theta_{i(n),n} - \xi_{i(n),n}\ge\epsilon$ i.o.} Analogously to Case 1,
Take a subsequence $\{\xi_{i(m),m}-\theta_{i(m),m}\}_{m}$ of $\{\xi_{i(n),n}-\theta_{i(n),n}\}_n$ such that
\begin{itemize}
    \item $\theta_{i(m),m} - \xi_{i(m),m}\ge\epsilon\,\mbox{ for all }\,m$,
    \item $\xi_{i(m),m}\to\xi, \quad\xi\in[0,\pi], \qquad \xi_{i(m),m} < \xi + \epsilon/2\,\mbox{ for all }\,m$,
    \item $\theta_{i(m),m}\to\theta,\quad\theta\in[0,\pi]$.
\end{itemize}
In particular, we find that $\theta \ge\xi + \epsilon$. By choosing $F=\chi_{(-\infty,f(\xi + \epsilon/2)]}$ in \eqref{...lambda...}, we obtain
\begin{align}\nonumber
&\frac1m\sum_{i=1}^m\chi_{(-\infty,f(\xi+ \epsilon/2)]}((-1)^{i+1}f(\xi_{i,n}))\to\frac1{2\pi}\int_0^\pi \chi_{(-\infty,f(\xi+ \epsilon/2)]}(f(\varphi)){\rm d}\varphi+\frac1{2\pi}\int_0^\pi\chi_{(-\infty,f(\xi+ \epsilon/2)]}(-f(\varphi)){\rm d}\varphi\\\nonumber
&\iff\frac{\#\{i\in\{1,\ldots,m\}:i\mbox{ is even }\,\vee\,i\mbox{ is odd and }f(\xi_{i,m})\le f(\xi+\epsilon/2)\}}m\to\frac{\mu_1\{f\le f(\xi+\epsilon/2))\}}{2\pi}+\frac12\\\nonumber
&\iff\frac{\lfloor m/2\rfloor}{m}+\frac{\#\{i\mbox{ is odd and }\xi_{i,m}\le \xi+\epsilon/2\}}{m}\to\frac{\xi+\epsilon/2}{2\pi}+\frac12\\
&\iff\frac{\#\{i\mbox{ is odd and }\xi_{i,m}\le \xi+\epsilon/2\}}{m}\to\frac{\xi+\epsilon/2}{2\pi}.
\label{eq:last2}
\end{align}
But now
\begin{align*}
\frac{\#\{i\mbox{ is odd and }\xi_{i,m}\le \xi+\epsilon/2\}}{m} &\ge \frac{\#\{i\mbox{ is odd and }\xi_{i,m}\le  \xi_{i(m),m}   \}}{m}
\ge  \frac{i(m)}{2m}
= \frac{\theta_{i(m),m }}{2\pi}
\to \frac{\theta}{2\pi}
\ge  \frac{\xi + \epsilon}{2\pi}
\end{align*}
that contradicts \eqref{eq:last2}.
\qed
\end{proof}

\begin{theorem}\label{th:Carl_generalized}
Let $f:[-\pi,\pi]\to \f R$ be a real even function such that
\begin{itemize}
	\item $f$ is Riemann integrable with connected range,
	\item $f$ has a finite number of local maxima and minima and discontinuity points.
\end{itemize}
 Then Theorem \ref{th:Carl} holds.
\end{theorem}
\begin{proof}
By Cantoni-Butler Theorem \ref{th:Cantoni-Butler}, there exists an ordering of the eigenvalues of $T_n(f)$ and $H_n(f)$ such that  $\lambda_i(H_n(f)) = (-1)^{1+1}\lambda_{i}(T_n(f))$. Moreover, we know by \cite{Ferrari2019,Hon2019,Mazza2018} that $\{H_n(f)\}_n \sim_\lambda H(x)$ where $H(x) = \diag(f(x),-f(x))$, so we can rewrite it as
$\{H_n(f)\}_n \sim_\lambda g(x)$, where $g(x) = f(2\pi x)$ for $x\in (0,1/2]$ and $g(1/2+x) = -f(2\pi x)$ for $x\in [0,1/2]$. In the case $\inf_{x\in [0,\pi]} f(x) = 0$, but it is never attained as a minimum, we impose $g(1/2) = 0$.
 Notice that $g$ is still a Riemann integrable function with a finite number of maxima, minima and discontinuity points. Here we distinguished two cases.

    \medskip

\noindent{\em Case 1: $\inf_{x\in [0,\pi]} f(x) > \delta >0$.}
In this case, $\text{Range}(f)$ and $\text{Range}(-f)$ are disjoint intervals with distance at least $2\delta$.
The hypotheses of Lemma \ref{lem:splitting} are thus satisfied with $d_n = n$, $\lambda_i^{(n)} = \lambda_i(H_n(f)))$, $f_1=f$, $f_2=-f$, $k=2$.
$\wt \Lambda_1^{(n)} = \{ \lambda_{2i}^{(n)} \}_{i = 1,\dots,\lfloor n/2\rfloor}$, $\wt \Lambda_2^{(n)} = \{ \lambda_{2i-1}^{(n)} \}_{i = 1,\dots,\lceil n/2\rceil}$. Notice in particular that
$\wt \Lambda_1^{(n)} \cu \text{Range}(f)$ and
$\wt \Lambda_2^{(n)} \cu \text{Range}(-f)$.
 Lemma \ref{lem:splitting} tells us that
 for any $n$ there exists a partition of $\Lambda^{(n)}$ into $2$ subset $\Lambda^{(n)}_1,\Lambda^{(n)}_2$ such that
 \begin{itemize}
    \item   $\Lambda^{(n)}_1$ has cardinality $L^{(n)}_1\coloneqq \lfloor n/2\rfloor$ and $\Lambda^{(n)}_2$ has cardinality $L^{(n)}_2\coloneqq \lceil n/2\rceil$,
    \item for every $j=1,2$
    \begin{equation}
    \{D^{(j)}_n\}_n\coloneqq\left\{\diag\left(\lambda^{(n)}_{i}\right)_{\lambda^{(n)}_{i}\in \Lambda^{(n)}_j}\right\}_n \sim_\lambda (-1)^{j+1} f(x),
    \end{equation}
    \item for any $n$,  $j\in \{1,\dots,k\}$ and $\lambda\in  \Lambda^{(n)}_j$
    \begin{align*}
    \min_{x\in [0,1]} f_j(x) - c_n \le \lambda \le \max_{x\in [0,1]} f_j(x) +  c_n
    \end{align*}
    for some $ c_n\to 0$.
 \end{itemize}
 But now, for any $n$ big enough $c_n<\delta$, so  \[
 \Lambda^{(n)}_1 \cu
 [\min_{x\in [0,1]} f_j(x) - c_n, \max_{x\in [0,1]} f_j(x) +  c_n]
 \cap \left(  \text{Range}(f) \cup \text{Range}(-f)  \right)
 =
  \text{Range}(f),
 \]
 and similarly $\Lambda^{(n)}_2\cu \text{Range}(-f)$, so that $\Lambda_1^{(n)} = \{ \lambda_{2i}^{(n)} \}_{i = 1,\dots,\lfloor n/2\rfloor}$, $\Lambda_2^{(n)} = \{ \lambda_{2i-1}^{(n)} \}_{i = 1,\dots,\lceil n/2\rceil}$. We can now apply Theorem \ref{th:main2} to both $\Lambda^{(n)}_1$ and $\Lambda^{(n)}_2$ to find that there exist two a.u.\ grids $\mathcal G^j_n$ on $[0,\pi]$ of size $L^{(n)}_j$, such that the elements of $\Lambda^{(n)}_j$ are the evaluation of $(-1)^{j+1}f(x)$ on the points of the grid $\mathcal G_n\coloneqq \mathcal G^j_n$ for $j=1,2$. All that is left to prove is that $\cup_{j=1,2} \mathcal G^j_n$ is still an a.u.\ grid on $[0,\pi]$, that is given by Lemma \ref{lem:union_of_au_grids_au}.

    \medskip

\noindent{\em Case 2: $\inf_{x\in [0,\pi]} f(x) \le 0$.}
In this case, the function $g(x)$ has connected range, so
the hypotheses of Theorem \ref{th:main2} are satisfied with $d_n = n$, $\lambda_i^{(n)} = \lambda_i(H_n(f)))$, and the function $g(x)$. As a consequence, there exists an a.u.\ grid $\wt {\mathcal G}_n \coloneqq \{\wt \xi_{i,n}  \}_{i=1,\dots,n}$
on $[0,1]$ such that $g(\wt \xi_{i,n}) = \lambda_{\tau_n(i)}^{(n)}$ for some permutation $\tau_n$. Notice that if $f>0$, then $1/2 \not \in \wt{ \mathcal G}_n$, since the value $0$ can never be attained by any $\lambda_i^{(n)}$.
We can thus define two grids $\wt {\mathcal G}^j_n$ such that
\[
\wt {\mathcal G}^1_n\coloneqq \{ 2\pi \wt\xi_{i,n}
:
\wt\xi_{i,n} < 1/2
\}, \qquad
\wt {\mathcal G}^2_n\coloneqq \{ 2\pi (\wt\xi_{i,n} - 1/2)
:
\wt\xi_{i,n} \ge  1/2
\}
\]
and the relative partition of $\Lambda_n$
\[
\wt {\Lambda}^1_n\coloneqq \{ f(\wt \xi^1_{i,n})
:
\wt \xi^1_{i,n} \in \wt {\mathcal G}^1_n
\}, \qquad
\wt {\Lambda}^2_n\coloneqq \{ -f(\wt \xi^2_{i,n})
:
\wt \xi^2_{i,n} \in \wt {\mathcal G}^2_n
\}
\]
The two grids are now a.u.\ on $[0,\pi]$, with cardinalities $\wt L^{(n)}_j =|\wt {\mathcal G}^j_n|$, but recall that at the start we had a different partition  of $\Lambda_n$ into $\{ \lambda_{2i}^{(n)} \}_{i = 1,\dots,\lfloor n/2\rfloor}\cu \text{Range}(f)$ and $\{ \lambda_{2i-1}^{(n)} \}_{i = 1,\dots,\lceil n/2\rceil}\cu \text{Range}(-f)$ of cardinality respectively $L^{(n)}_1\coloneqq \lfloor n/2\rfloor$ and  $L^{(n)}_2\coloneqq \lceil n/2\rceil$.
Since $\mathcal G_n$ is an a.u.\ grid, one can find that
\[
e_n\coloneqq \left|  |\wt L^{(n)}_1| - |L^{(n)}_1|\right|
=
\left|  |\wt L^{(n)}_2| - |L^{(n)}_2|\right|
= o(n)
\]
where without loss of generality, we can assume $|\wt L^{(n)}_1| \ge  |L^{(n)}_1|$. This means that we can find $e_n=o(n)$ elements in $\wt {\Lambda}^1_n\cap \{ \lambda_{2i-1}^{(n)} \}_{i = 1,\dots,\lceil n/2\rceil}$ that we can move from $\wt {\Lambda}^1_n$ to $\wt {\Lambda}^2_n$ by moving the corresponding points $\wt  \xi^1_{i,n}$ from $\wt {\mathcal G}^1_n$ to $\wt {\mathcal G}^2_n$. We thus generate two new grids ${\mathcal G}^1_n$ to ${\mathcal G}^2_n$ that are still a.u.\ on $[0,\pi]$ due to Lemma \ref{lem:au_grids_up_to_o(n)} and such that the generated partitions ${\Lambda}^1_n$ and ${\Lambda}^2_n$ satisfy
$
|{\Lambda}^j_n| =  L^{(n)}_j
$
and
${\Lambda}^j_n \cu  \text{Range}((-1)^{j+1}f)$.
The union $\cup_{j=1,2} \mathcal G^j_n$ is still a.u.\ on $[0,\pi]$ due to Lemma \ref{lem:union_of_au_grids_au}, thus concluding the proof.

\qed
\end{proof}

\begin{theorem}\label{th:Stefano_version}
Let $f:[-\pi,\pi]\to \f R$ be a real even Riemann integrable function with connected range. Then, for every $n\in \f N$
and for every $\{\xi_{1,n}, \xi_{2,n}, \dots, \xi_{n,n}\}_n$ a.u. grid on $[0,\pi]$,
there exist real values $\psi_{1,n}, \psi_{2,n},\dots,\psi_{n,n}$ with the following properties.
\begin{enumerate}
\item The eigenvalues of $T_n(f)$ and $H_n(f)$ are given by
\begin{align*}
\lambda_i(T_n(f)) &= f(\xi_{i,n}) + \psi_{i,n},\\
\lambda_i(H_n(f)) &= (-1)^{i+1} \lambda_i(T_n(f)),
\end{align*}
for all $i=1,\dots,n$.
\item $\max_{i=1,\dots, n}|\psi_{i,n}|\to 0$ as $n\to \infty$.
\end{enumerate}
\end{theorem}

\begin{proof}

By Cantoni-Butler Theorem \ref{th:Cantoni-Butler}, there exists an ordering of the eigenvalues of $T_n(f)$ and $H_n(f)$ such that  $\lambda_i(H_n(f)) = (-1)^{1+1}\lambda_{i}(T_n(f))$. Moreover, we know by \cite{Ferrari2019,Hon2019,Mazza2018} that $\{H_n(f)\}_n \sim_\lambda H(x)$ where $H(x) = \diag(f(x),-f(x))$. We can then directly apply Theorem \ref{th:main} with $d_n = n$, $\lambda_i^{(n)} = \lambda_i(H_n(f)))$, $f_1=f$, $f_2=-f$, $k=2$,
$\Lambda_1^{(n)} = \{ \lambda_{2i}^{(n)} \}_{i = 1,\dots,\lfloor n/2\rfloor}$, $\Lambda_2^{(n)} = \{ \lambda_{2i-1}^{(n)} \}_{i = 1,\dots,\lceil n/2\rceil}$.
The theorem tells us that there exists a partition of $\Lambda^{(n)}$ into two subsets $\wt \Lambda_j^{(n)}$ with the same cardinality of $\Lambda_j^{(n)}$ and such that for every couple of a.u.\ grid on $[0,\pi]$
$\mathcal G^j_n=\{\xi^j_{i,n}\}_{i=1,\dots,L^{(n)}_j}$ in $[0,1]$ with cardinality $|\Lambda_j^{(n)}| = L^{(n)}_j$, there exists an ordering of the elements of $\wt \Lambda_j^{(n)}$ such that
\[
\max_{\lambda^{(n)}_{i,n} \in \wt \Lambda_j^{(n)}} \left| f_j(\xi_{i,n}^j) - \lambda^{(n)}_{i,n}  \right|
\to 0.
\]
Given now a fixed a.u.\ grid $\mathcal G_n=\{\xi_{i,n}\}_{i=1,\dots,n}$ on $[0,\pi]$, the two subgrids  $\mathcal G^1_n = \{ \xi_{2i,n} \}_{i = 1,\dots,\lfloor n/2\rfloor}$, $\mathcal G^2_n = \{ \xi_{2i-1,n} \}_{i = 1,\dots,\lceil n/2\rceil}$ have exactly cardinality $L^{(n)}_j$ and they are both a.u.\ grids on $[0,\pi]$, so the result follows.

\qed
\end{proof}

\subsection{More General Results}\label{ssec: general}
The following is a generalization of \cite[Theorem 1.5]{Al} and  \cite[Theorem 1.3]{Al2}, but the proof is almost identical. Note that in the latest article the hypothesis ``$f$ has connected and bounded essential range" must be replaced with ``$f$ has connected and bounded range" otherwise the result is false.

\begin{theorem}\label{th:spectral_bounded_means_uniform}
Let $\lambda^{(n)}_1, \lambda^{(n)}_2, \dots, \lambda^{(n)}_{d_n}$ a sequence of $d_n\to\infty$ real values for any $n\in\f N$, and $D_n\coloneqq\diag(\lambda^{(n)}_i)_{i=1,\dots,d_n}$. Given a Riemann integrable function $f:[0,1]\to \f R$ with connected range,  suppose that $\{D_n\}_n\sim_\lambda f(x)$ and that for any $n$ and any $i\in \{1,\dots,d_n\}$
\[
\min_{x\in [0,1]} f(x) - c_n \le \lambda^{(n)}_i \le \max_{x\in [0,1]} f(x) + c_n,
\]
where $c_n\to 0$. In this case,
for any a.u.\ grid $\mathcal G_n=\{\xi_{i,n}\}_{i=1,\ldots,d_n}$ in $[0,1]$
there exists a permutation $\tau_n\in S^{d_n}$ such that
\[
\max_{i=1,\dots,d_n} \left| f(\xi_{i,n}) - \lambda^{(n)}_{\tau_n(i)}  \right|
\to 0.
\]
\end{theorem}
\begin{proof}
Fix an a.u.\ grid $\mathcal G_n=\{\xi_{i,n}\}_{i=1,\ldots,d_n}$ in $[0,1]$. If we add the point $0$ to $\mathcal G_n$ for every $n$, it still is an a.u.\ grid, and since $f^\dagger$ is continuous on $[0,1]$, by Theorem \ref{th:rearr_main} $f_n^\dagger\to f^\dagger$ on $(0,1)$, but it is also possible to prove the convergence at the extrema, since $\text{Range}(f_n^\dagger)\cu \text{Range}(f^\dagger)$, so for any $\ve>0$,
\begin{equation}\label{eq:extrema_conv}
f_n^\dagger(\epsilon) \ge f_n^\dagger(0) = \lambda^{(n)}_1\ge f^\dagger(0)-c_n \implies
f^\dagger(\epsilon)\ge \limsup_{n\to \infty} f_n^\dagger(0) \ge \liminf_{n\to \infty} f_n^\dagger(0) \ge f^\dagger(0).
\end{equation}
Since $f^\dagger$ is continuous, then $f_n^\dagger(0)\to f^\dagger(0)$. The same reasoning proves that $f_n^\dagger(1)\to f^\dagger(1)$, thus concluding that $f_n^\dagger\to f^\dagger$ on $[0,1]$. By Lemma \ref{Dini}, we conclude that $f_n^\dagger \to f^\dagger$ uniformly in $n$ and in particular
there exists a permutation $\tau_n^{-1}\in S^{d_n}$ such that
\begin{equation}\label{eq:1}
\max_{i=1,\dots,d_n} \left| f(\xi_{\tau_n^{-1}(i),n}) - f^\dagger(i/d_n)  \right| \to 0.
\end{equation}
Since $f^\dagger$ is a rearranged version of $f$, by hypothesis $\{D_n\}_n\sim_\lambda f^\dagger$. Suppose now without loss of generality that $\lambda^{(n)}_1\le \lambda^{(n)}_2\le \dots\le  \lambda^{(n)}_{d_n}$ and let $g_n:[0,1]\to \f R$ be the linear spline function that interpolates $(\lambda^{(n)}_1,\lambda^{(n)}_1, \lambda^{(n)}_2, \dots, \lambda^{(n)}_{d_n})$ (notice that only the first value is repeated two times) over the equally spaced nodes $(0,\frac{1}{d_n},\frac{2}{d_n},\ldots,1)$ in $[0,1]$. The hypothesis of Lemma \ref{lem:rearr_main} is now satisfied with $g=f^\dagger$ due to the ergodic formula associated to $\{D_n\}_n\sim_\lambda f^\dagger$, so $g_n(x)\to f^\dagger(x)$ for every $x\in (0,1)$ because $f^\dagger$ is continuous. Repeating the same reasoning as in \eqref{eq:extrema_conv} but with $g_n$ instead of $f_n^\dagger$, we conclude that $g_n \to f^\dagger$ uniformly in $n$ and in particular
\begin{equation}\label{eq:2}
\max_{i=1,\dots,d_n} \left| \lambda^{(n)}_i - f^\dagger(i/d_n)  \right| \to 0.
\end{equation}
The equation \eqref{eq:1} and \eqref{eq:2} let us conclude that
\begin{equation*}
\max_{i=1,\dots,d_n} \left| f(\xi_{\tau_n^{-1}(i),n}) - \lambda^{(n)}_i  \right| \to 0
\end{equation*}
that proves the theorem for the a.u.\ grid $\mathcal G_n=\{\xi_{i,n}\}_{i=1,\ldots,d_n}$.

\qed
\end{proof}
\begin{theorem}\label{th:main}
Let $\Lambda^{(n)} \coloneqq \{\lambda^{(n)}_1, \lambda^{(n)}_2, \dots, \lambda_{d_n}^{(n)} \}$ be a sequence of $d_n$ real values for any $n\in\f N$ with $d_n\to\infty$, and $D_n\coloneqq\diag(\lambda^{(n)}_i)_{i=1,\dots,d_n}$.
Given a diagonal  $k\times k$ matrix-valued function $H(x) \coloneqq \diag(f_j(x))_{j=1,\dots,k}$, where $f_j:[0,1]\to \f R$ are Riemann integrable with connected range, suppose that $\{D_n\}_n\sim_\lambda H(x)$ and that for any $n$ there exists a partition of $\Lambda^{(n)}$ into $k$ subset $\Lambda^{(n)}_1,\dots,\Lambda^{(n)}_k$ such that for every $j=1,\dots,k$
\begin{itemize}
\item $L^{(n)}_j/d_n\to 1/k$, where $L^{(n)}_j$ is
the cardinality  of $\Lambda^{(n)}_j$ ,
\item $\min_{x\in [0,1]} f_j(x) - c_n \le \lambda \le \max_{x\in [0,1]} f_j(x) + c_n$ for each $\lambda\in \Lambda^{(n)}_j$, where $c_n\to 0$.
\end{itemize}
Then for every $n$ there exists a partition of $\Lambda^{(n)}$ into $k$ subset $\wt \Lambda^{(n)}_1,\dots,\wt \Lambda^{(n)}_k$ such that for every $j=1,\dots,k$
\begin{itemize}
\item $\wt \Lambda^{(n)}_j$ has cardinality $L^{(n)}_j$,
\item there exists an ordering of the elements $\lambda^{(n,j)}_{i}$ of $\wt \Lambda^{(n)}_j$ such that for any a.u.\ grid $\mathcal G^j_n=\{\xi^j_{i,n}\}_{i=1,\dots,L^{(n)}_j}$ in $[0,1]$ with cardinality $|\mathcal G^j_n| = L^{(n)}_j$, we have
\[
\max_{i=1,\dots,L^{(n)}_j} \left| f_j(\xi_{i,n}^j) - \lambda^{(n,j)}_{i}  \right|
\to 0.
\]
\end{itemize}

\end{theorem}
\begin{proof}
By hypothesis, we can apply directly Lemma \ref{lem:splitting} and deduce that for any $n$ there exists a partition of $\Lambda^{(n)}$ into $k$ subset $\wt \Lambda^{(n)}_1,\dots,\wt \Lambda^{(n)}_k$ such that for every $j=1,\dots,k$
\begin{itemize}
    \item  $\wt \Lambda^{(n)}_j$ has cardinality $L^{(n)}_j$,
    \item$
    \{D^{(j)}_n\}_n\coloneqq\left\{\diag\left(\lambda^{(n)}_{i}\right)_{\lambda^{(n)}_{i}\in \wt \Lambda^{(n)}_j}\right\}_n \sim_\lambda f_j(x).
    $
    \item $\min_{x\in [0,1]} f_j(x) -  c_n \le \lambda \le \max_{x\in [0,1]} f_j(x) +  c_n$
    for all  $\lambda\in  \Lambda^{(n)}_j$ and some $c_n\to 0$.
\end{itemize}
Applying now Theorem \ref{th:spectral_bounded_means_uniform} to all sequences $\{D^{(j)}_n\}_n$ we find that up to a permutation of the elements $\lambda_i^{(n,j)}$ inside each $\wt \Lambda^{(n)}_j$, for any a.u.\ grid $\mathcal G^j_n=\{\xi_{i,n}\}_{i=1,\ldots,L^{(n)}_j}$ in $[0,1]$
and all $j=1,\dots,k$,
\[
\max_{i=1,\dots,L^{(n)}_j} \left| f_j(\xi_{i,n}) - \lambda_i^{(n,j)}  \right|
\to 0.
\]
thus proving the result.

\qed
\end{proof}

\begin{theorem}\label{th:main2}
Let $\Lambda^{(n)} \coloneqq \{\lambda^{(n)}_1, \lambda^{(n)}_2, \dots, \lambda_{d_n}^{(n)} \}$ be a sequence of $d_n$ real values for any $n\in\f N$ with $d_n\to\infty$, and $D_n\coloneqq\diag(\lambda^{(n)}_i)_{i=1,\dots,d_n}$. Let $f:[0,1]\to \f R$ be a Riemann integrable function such that
\begin{itemize}
\item $f$ has a finite number of local maxima and minima and discontinuity points,
\item $f$ has connected range,
\item $\{D_n\}_n\sim_\lambda f(x)$,
\item $\Lambda^{(n)}\cu \text{Range}(f)$.
\end{itemize}
 Then there exists an a.u.\ grid $\{\xi_{1,n}, \xi_{2,n}, \dots, \xi_{d_n,n}\}_n$ and a permutation $\tau_n$ such that  for every $i=1,\dots,d_n$
$$\lambda^{(n)}_{\tau_n(i)}= f(\xi_{i,n}).$$
\end{theorem}
\begin{proof}
By hypothesis, we can apply directly Theorem \ref{th:spectral_bounded_means_uniform} and find that for the regular grid $\theta_{i,n} = i/d_n$ and for a specific ordering of the values $\lambda_i^{(n)}$,
    \[
\max_{i=1,\dots,d_n} \left| f(\theta_{i,n}) - \lambda^{(n)}_{i}  \right| = c_n
\to 0.
\]
Moreover, by hypothesis, $\lambda_i^{(n)}\in \text{Range}(f)$, so the sets $f^{-1}(\lambda_i^{(n)})$ are never empty. As a consequence, we can generate the grid $\mathcal G_n=\{\xi_{i,n}\}_{i=1,\dots,d_n}$
such that for any $i,n$, $\xi_{i,n}$ is the closest value to $\theta_{i,n}$ in $[0,1]$ such that $f(\xi_{i,n}) = \lambda_i^{(n)}$. All that is left to prove is that $\mathcal G_n$ is an a.u.\ grid.

For fixed $\delta,\ve>0$, define the set
\[
E_{\delta,\epsilon}\coloneqq \left\{ x\in[\delta,1-\delta] :
[f(x)-\epsilon, f(x) +\epsilon ] \not\cu f( [x-\delta,x+\delta] )  \right\}
\cup
[0,\delta]\cup [1-\delta,1]
\]
and notice that $\theta_{i,n} \in (E_{\delta,c_n})^C \implies  |\xi_{i,n}  - \theta_{i,n}| \le \delta$. Call now
$
\mathcal E_{\delta,n}\coloneqq \{ i : \theta_{i,n} \in E_{\delta,c_n}  \}\cu \{1,2,\dots,d_n\}.
$
Call $x_0< x_1<x_2<\dots<x_{k}<x_{k+1}$ the local minima and maxima and the discontinuity points of $f$, where $x_0=0$, $x_{k+1}=1$. Since $f$
 is continuous on the intervals $\mathcal I_i\coloneqq (x_i,x_{i+1})$, then it must be strictly monotonous on $\mathcal I_i$ for every $i=0,\dots k$. Let us fix the interval $\mathcal I = (a,b) = \mathcal I_i$, and suppose without loss of generality that $f$ is strictly increasing on $I$. Since $f$ is continuous on $\mathcal I$, the function $f^+_\delta(x) = f(x+\delta)-f(x)>0$ is also continuous and strictly positive on $[a+\delta,b-2\delta]$, that is a nonempty closed interval for $\delta$ small enough. As a consequence, $f^+_\delta(x)$ has a strictly positive minimum value $\ve^+>0$, and analogously, the function  $f^-_\delta(x) = f(x) - f(x-\delta)>0$ has a strictly positive minimum value $\ve^->0$ on the nonempty closed interval $[a+2\delta,b-\delta]$. If now $\ve = \min\{\ve^+,\ve^-\}>0$, we find that
 \[
 [f(x)-\epsilon, f(x) +\epsilon ] \cu f( [x-\delta,x+\delta] ),     \qquad \forall x\in [a+2\delta,b-2\delta],
 \]
and in particular, if $c_n<\epsilon$ and $\theta_{i,n}\in [a+2\delta,b-2\delta]$, then   $i\not \in \mathcal E_{\delta,n}$. Repeating the reasoning for all intervals, it is clear that for any $n$ big enough,
$\mathcal E_{\delta,n} \cu E_\delta \coloneqq \cup_{i=0,\dots,k+1} [x_i-2\delta,x_i+2\delta]$, that has Lebesgue measure $4(k+2)\delta$. Due to the grid of $\theta_{i,n}$ being regular on $[0,1]$ one finds that for any $n$ big enough, $|\mathcal E_{\delta,n}|\le 5(k+2)\delta d_n$, so it is possible to choose a sequence $\delta_n\to 0$ such that $|\mathcal E_{\delta_n,n}|= o(d_n)$.

To conclude the proof, let now  $\mathcal G'_n=\{\xi_{i,n} :\theta_{i,n} \in (E_{\delta_n,c_n})^C \} \cup \{\theta_{i,n} :\theta_{i,n} \in E_{\delta_n,c_n}  \}$ be a grid of the same cardinality of $\mathcal G_n$, and notice that it is an a.u.\ grid, since the distance from the regular grid of $\theta_{i,n}$ is uniformly bounded by $\delta_n\to 0$. Since it differs from the original grid $\mathcal G_n$ at most by $2|\mathcal E_{\delta_n,n}|= o(d_n)$ element, it is enough to apply Lemma \ref{lem:au_grids_up_to_o(n)} to conclude that also $\mathcal G_n$ is an a.u.\ grid.

\qed
\end{proof}

\subsection{Real Non-Symmetric Case}\label{ssec: real nonsymm}

\subsubsection{Eigenvectors}

The eigenvectors of $H_n(f)$ and $T_n(f)$ are the same for a real-valued even generating function $f$ (that is $\hat f_{k}=\hat f_{-k}$ any integer $k$, all being real). However, this is true in the more general case considered in the present subsection, where $\hat f_{k}$ is real for any integer $k$, but $f$ is complex-valued. Indeed, in the case where $f$ is complex-valued, but the Fourier coefficients are all real, the left-eigenvectors of $T_n(f)$ coincide with the eigenvectors of the real symmetric matrix $H_n(f)$.

\begin{theorem}
\label{th:nonsym-eigval}
Under the assumption that the Fourier coefficients are real, the eigenvectors of $H_n(f)$ are the same as the corresponding either left or right singular vectors of  $T_n(f)$ (and eigenvectors either of $(T_n(f))^{\textsc{t}}T_n(f)$ or of $T_n(f)(T_n(f))^{\textsc{t}}$).
\end{theorem}
\begin{proof}
Using the singular value decomposition we know that $T_n(f)=U\Sigma V$, where $U, V$ are unitary and $\Sigma$ is diagonal with the singular values ordered non-increasingly. Now $H_n(f)=Y_n T_n(f)$ is real symmetric and hence it admits the Schur decomposition in the form $QDQ^{\textsc{t}}$ where $Q$ is real orthogonal. However, $Y_n T_n(f)=Y_nU \Sigma V$ is automatically a singular value decomposition of $Y_n T_n(f)$, since $Y_n$ is unitary and hence
$Y_nU=W$ is the unitary matrix containing the left singular vectors of $Y_n T_n(f)$. In addition, since  $Y_n T_n(f)$ is real symmetric its eigenvalues are real and due to its normality the singular values are the moduli of the eigenvalues. In other terms we have
\[
D_{j,j}=\Sigma_{j,j} (-1)^{\alpha_j}, \ \ \ \ \alpha_j\in \{0,1\},
\]
that is $D=\Sigma S$ with $S$ phase matrix such that $D_{j,j}=(-1)^{\alpha_j}$.
From this we deduce
\[
H_n(f)=Q D Q^{\textsc{t}}=QS \Sigma Q^{\textsc{t}}= Q \Sigma SQ^{\textsc{t}}
\]
where the latter two, up to reordering, represent the singular value decomposition of $H_n(f)$ and the proof is over.

\qed
\end{proof}

\subsubsection{Eigenvalues}

We start this part by giving localization results for the eigenvalues of $H_n(f)$ in the case of a real-valued even generating function $f\in L^1(-\pi,\pi)$, i.e. $\hat f_{k}=\hat f_{-k}$ for any integer $k$, all being real. Then we provide an analogous result when the assumption that $f$ is real-valued is dropped. The two results are given as corollaries, the first consequence of Theorem \ref{cbt}, the second consequence of Theorem \ref{th:nonsym-eigval}.

\begin{corollary}\label{cor:loc1}
Let $f\in L^1(-\pi,\pi)$, $\hat f_{k}=\hat f_{-k}$ for any integer $k$, all the Fourier coefficients being real. Let $m$ be the essential infimum of $f$ and $M$ be the essential supremum of $f$. Then all the eigenvalues of $H_n(f)$ belong to $(-M,-m)\cup (m,M)$ for any size $n$ if $m<M$. If $m=M$ then $H_n(f)=mH_n(1)$ and hence the eigenvalues are $\pm 1$, for any size $n$.
\end{corollary}
\begin{proof}
Under the given assumptions on the Fourier coefficients of the generating function $f$, we know that it is real-valued a.e. so that it makes sense to consider consider the essential infimum and the essential supremum of $f$, that is $m$ and $M$ are well defined.
Now by the localization results in \cite{Se1}, the eigenvalues of $T_n(f)$ are such that
\begin{itemize}
\item $\lambda_j(T_n(f))\in (m,M)$, $\forall j=1,\ldots,n$, $\forall n\in \mathbb{N}^+$, if $m<M$;
\item  $\lambda_j(T_n(f))=m$, $\forall j=1,\ldots,n$, $\forall n\in \mathbb{N}^+$, if $m=M$ since $T_n(f)=mI$.
\end{itemize}
Therefore, by invoking Theorem \ref{cbt}, the claimed thesis follows.
\qed
\end{proof}

\begin{corollary}\label{cor:loc2}
Let $f\in L^1(-\pi,\pi)$, $\hat f_{k}\in \mathbb{R}$ for any integer $k$. Let $d$ be the distance of the essential range of $f$ from the complex zero, $m$ be the essential infimum of $|f|$, and $M$ be the essential supremum of $|f|$. Then all the eigenvalues of $H_n(f)$ belong to $[-M,-d]\cup [d,M]$ for any size $n$ and, if $m<M$, then the number of those belonging $[-M,-m]\cup [m,M]$ are $n-c_n$ with $c_n=o(n)$ ($O(1)$ if $f$ is a trigonometric polynomial).
\end{corollary}

\begin{proof}
The singular values of $T_n(f)$ are localized in $[d,M]$, where the right parenthesis ``$]$" can by replaced by ``$)$" if $m<M$ and the left parenthesis ``$[$" can be replaced by ``$($", if there exist $\omega$ of modulus $1$ such that the essential range of $\omega f-d$ is weakly sectorial \cite{local2}.
Furthermore, if $m<M$ by the relation $\{ T_n(f)\}\sim_{\sigma} (f,[-\pi,\pi])$ in Theorem \ref{teoszego-tyr}, the number of the singular values of $T_n(f)$ not belonging to $[m,M]$ is $n-c_n$, $c_n=o(n)$ and $c_n=O(1)$ if $f$ is a trigonometric polynomial.

Now the claimed thesis follows by using Theorem \ref{th:nonsym-eigval}.
\qed
\end{proof}
Here, under the assumption that the Fourier coefficients are real, we refine the distribution results given in Theorem \ref{teoszego-tyr} for the Toeplitz matrix-sequence generated by $f$ and in \cite{Ferrari2019, ferrari2020,Hon2019,mazzapestana20,pestana} for the related flipped Toeplitz matrix-sequence.

Before doing it, we introduce the following notation, which was established in \cite{Ferrari2019} for certain symmetrized Toeplitz matrix-sequences. Given $D \subset \mathbb{R}$ with Lebesgue measure $0<\mu_1(D)<\infty$, we define $\widetilde D$ as $D\bigcup D_p$, where $p\in  \mathbb{R}$ and $D_p=p+D$, with the constraint that $D$ and $D_p$ have non-intersecting interior part, i.e. $D^\circ \bigcap  D_p^\circ  =\emptyset$. In this way, we have $\mu_1(\widetilde D)=2\mu_1(D)$. Given any $g$ measurable and defined over $D$, we define $\psi_g$ over $\widetilde D$ in the following fashion
	\begin{equation}\label{def-psi}
	\psi_g(x)=\left\{
	\begin{array}{cc}
	g(x), & x\in D, \\
	-g(x-p), & x\in D_p, \ x \notin D.
	\end{array}
	\right.\,
	\end{equation}
	
\begin{theorem}\label{th:distribution [0,pi]}
Let $f\in L^1(-\pi,\pi)$, $\hat f_{k}\in \mathbb{R}$ for any integer $k$. Let $p=\pi$, $D=[0,\pi]$, $D_p$, $\psi_f$, $\psi_{|f|}$ as (\ref{def-psi}).
Then the singular value distribution can be restricted to $[0,\pi]$ i.e.
\[
\{ T_n(f)\}\sim_{\sigma} (f,[0,\pi]),\ \ \{ H_n(f)\}\sim_{\sigma} (f,[0,\pi])
\]
and furthermore
\[
\{ H_n(f)\}\sim_{\lambda} (\psi_{|f|},D_p).
\]
If we add the even character of the Fourier coefficients, i.e., $\hat f_{k}=\hat f_{-k}\in \mathbb{R}$ for any integer $k$, then
\[
\{ T_n(f)\}\sim_{\lambda} (f,[0,\pi]),\ \ \{ H_n(f)\}\sim_{\lambda} (\psi_{f},D_p),\ \ \{ H_n(f)\}\sim_{\lambda} (\psi_{|f|},D_p).
\]
\end{theorem}
\begin{proof}
By the assumption we have ${\rm Re}(f(\theta))={\rm Re}(f(-\theta))$ and ${\rm Im}(f(\theta))=-{\rm Im}(f(-\theta))$ a.e., which implies that for any subinterval $I\subset [-\pi,0]$ the essential range of $f$ restricted to $I$ is a reflection along the real axis of the the essential range of $f$
restricted to $-I\subset [0,\pi]$. Hence
\begin{equation}\label{reflection}
|f(\theta)|=|f(-\theta)|
\end{equation}
a.e. and as consequence Theorem \ref{teoszego-tyr}, which claims $\{ T_n(f)\}\sim_{\sigma} (f,[-\pi,\pi])$, reduces to $\{ T_n(f)\}\sim_{\sigma} (f,[0,\pi])$ which is the same as $\{ H_n(f)\}\sim_{\sigma} (f,[0,\pi])$ since $H_n(f)$ and $T_n(f)$ share the same eigenvalues.
Now by using the main result in \cite{ferrari2020} the statement $\{ H_n(f)\}\sim_{\lambda} (\psi_{|f|},D_p)$ follows.

If we add to the assumptions the even character of the Fourier coefficients, we have  ${\rm Im}(f(\theta))=0$ a.e. so that $f$ is real-valued and again the combination of Theorem \ref{teoszego-tyr}, of the main result in \cite{ferrari2020}, and of the fact that $f(\theta)=f(-\theta)$ a.e. allows to deduce the desired result.
\qed
\end{proof}

Now, with reference to the proof of Theorem \ref{th:nonsym-eigval}, for analyzing the distribution of the signs $(-1)^{\alpha_j}$, it is enough to recall that the sequence $\{T_n(f)\}$ is distributed in the singular value sense as $f$ while, $\{H_n(f)\}$ is distributed as $\pm |f|$ (see \cite{Ferrari2019,ferrari2020,Hon2019,mazzapestana20,pestana} and Theorem \ref{th:distribution [0,pi]}): from this we deduce that
\[
\lim_{n\rightarrow \infty}\sum_{j=1}^n \frac{(-1)^{\alpha_j}}{n}=0
\]
and therefore, up to $o(n)$ outliers, there is around $n/2$ positive signs and $n/2$ negative signs.
In the case where the minimum of $|f|$ is zero, we do not have outliers. Furthermore in the case where $f$ is smooth, up to a negligible number of outliers the eigenvalues of $H_n(f)$ can be seen as a sampling of $\pm |f|$: the number of these outliers all with modulus less than $\min |f|$ can be bounded by a constant independent of $n$, when $f$ is also a trigonometric polynomial.

On the basis of the latter discussion, we end the current theoretical section with a conjecture.

\begin{conj}
\label{conj:nonsym-eigval}
In the case where $f$ is non-symmetric, real-valued, with the notations of Corollary \ref{cor:loc2}, and with $d=m$, the eigenvalues of $H_n(f)$ are related to the singular values of $T_n(f)$ as follows
\begin{equation*}
\lambda_j(H_n(f))=(-1)^{j+1}\sigma_j(T_n(f)),
\end{equation*}
where $\sigma_j(T_n(f))$ are ordered as the samplings of the symbol $f(\xi_{j,n})f(-\xi_{j,n})$, where $\sigma_j(T_n(f))=\sqrt{f(\xi_{j,n})f(-\xi_{j,n})}=|f(\xi_{j,n})|$ and
\begin{equation}
0\leq\xi_{1,n}\leq \xi_{2,n}\leq\ldots\leq \xi_{n,n}\leq \pi,
\label{eq:ordering2}
\end{equation}
with $\{\xi_{j,n}\}$ being a.u. on $[0,\pi]$.
\end{conj}

\section{Numerical Experiments}\label{sec:num}
In this section we start by giving numerical numerical evidence of the results in Theorem \ref{th:Carl_generalized} and Theorem \ref{th:Stefano_version}. Then we draw some relevant conclusion.

\begin{exmp}
\label{exmp:num-a}
We first define the following generating function, defined in the reference interval $[-\pi,\pi]$, that is,
\begin{align}
f(\theta)=
\begin{cases}
f(-\theta),&-\pi\leq\theta<0,\\
1,&0\leq \theta<\pi/2,\\
\theta+1-\pi/2, &\pi/2\leq \theta\leq \pi,
\end{cases}
\end{align}
Then, we generate the matrix $T_n(f)$ for $n=20$ and compute the eigenvalues, listed in Table~\ref{tbl:num-a} (computed with high precision, and then truncated to 20 digits).
As it can be observed, the generating function satisfies the requirements of Theorem \ref{th:Stefano_version}, but not those of Theorem \ref{th:Carl_generalized}, since there exist uncountable many minima/maxima for $\theta\in [0,\pi/2)$ and one local minimum at $\theta=\pi/2$.
\begin{table}[H]
\centering
\caption{The eigenvalues of $\lambda_j(T_n(f))$ and $\lambda_{j}(H_n(f))$ for $n=20$, and a grid $\xi_{j,n}$ such that $\lambda_j(T_n(f))=f(\xi_{j,n})$.}
\label{tbl:num-a}
\begin{tabular}{rrrr}
\toprule
$j$&$\lambda_j(T_n(f))$&$\lambda_j(H_n(f))$&$\xi_{j,n}$\\
\midrule
1 & 1.00000000000000353822 & 1.00000000000000353822 & 1.57079632679490009622 \\
2 & 1.00000000000071333310 & -1.00000000000071333310 & 1.57079632679560989110 \\
3 & 1.00000000006613777056 & 1.00000000006613777056 & 1.57079632686103432856 \\
4 & 1.00000000369131598005 & -1.00000000369131598005 & 1.57079633048621253805 \\
5 & 1.00000013757194985002 & 1.00000013757194985002 & 1.57079646436684640802 \\
6 & 1.00000359995028312744 & -1.00000359995028312744 & 1.57079992674517968544 \\
7 & 1.00006712436027692073 & 1.00006712436027692073 & 1.57086345115517347873 \\
8 & 1.00088357014017741679 & -1.00088357014017741679 & 1.57167989693507397479 \\
9 & 1.00779697209247498221 & 1.00779697209247498221 & 1.57859329888737154021 \\
10 & 1.04221339677660198160 & -1.04221339677660198160 & 1.61300972357149853960 \\
11 & 1.13188837384995795521 & 1.13188837384995795521 & 1.70268470064485451321 \\
12 & 1.26170015132705443149 & -1.26170015132705443149 & 1.83249647812195098949 \\
13 & 1.40266551094883595348 & 1.40266551094883595348 & 1.97346183774373251148 \\
14 & 1.54926998899631018209 & -1.54926998899631018209 & 2.12006631579120674009 \\
15 & 1.69790994462440194399 & 1.69790994462440194399 & 2.26870627141929850199 \\
16 & 1.84835063290469307888 & -1.84835063290469307888 & 2.41914695969958963688 \\
17 & 1.99917620458656980870 & 1.99917620458656980870 & 2.56997253138146636670 \\
18 & 2.15136121246339999889 & -2.15136121246339999889 & 2.72215753925829655689 \\
19 & 2.30273860591566235042 & 2.30273860591566235042 & 2.87353493271055890842 \\
20 & 2.45795620370766419040 & -2.45795620370766419040 & 3.02875253050256074839 \\
\bottomrule
\end{tabular}
\end{table}
If Figure~\ref{fig:num-a} we show the eigenvalues and the eigenvalue symbols for the Toeplitz and for the corresponding Hankel matrix-sequences. The perfect grid is in this case $\xi_{j,n}=\lambda_{j,n}+\pi/2-1$ and no eigenvalue is equal to one, as expected from the theory (see \cite{absolute-extreme-S,absolute-extreme-T}). This perfect grid is not a.u. in $[0,\pi]$ and it is not a.u. in $[\pi/2,\pi]$ since, as expected from the theoretical findings, $n/2$ eigenvalues of $T_n(f)$ cluster to one, with the smallest tending exponentially to one as the matrix-size $n$ tends to infinity \cite{absolute-extreme-S,absolute-extreme-T}.
\begin{figure}[H]
    \centering
    \includegraphics[width=0.48\textwidth]{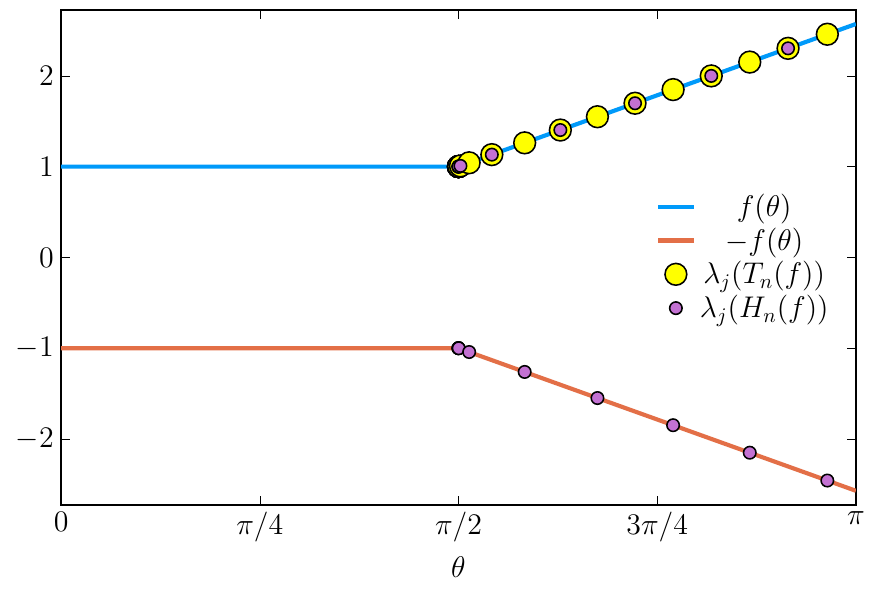}
    \caption{[Example~\ref{exmp:num-a}] Symbol which is constant on $\theta=[-\pi/2,\pi/2]$. }
    \label{fig:num-a}
    \end{figure}
    
\end{exmp}
\begin{exmp}
\label{exmp:num-b}
In this example we take into consideration a non-monotone generating function,
\begin{align}
f(\theta)=16-2\cos(\theta)-2\cos(2\theta)+\cos(3\theta),
\end{align}
and we compute the eigenvalues of $T_n(f)$ and $H_n(f)$ for $n=20$, as displayed in Figure~\ref{fig:num-b}. Here the perfect grid $\xi_{j,n}$ was computed numerically with a root finder, with $\theta_{j,n}=j\pi/(n+1)$ as an initial guess. Looking at the signs of the eigenvalues $\lambda_j(H_n(f))$, we conclude that in this specific example the perfect grid, taken from a a.u. sequence of grids in $[0,\pi]$ exists, in accordance with Theorem \ref{th:Carl_generalized}.
\begin{figure}[H]
    \centering
    \includegraphics[width=0.48\textwidth]{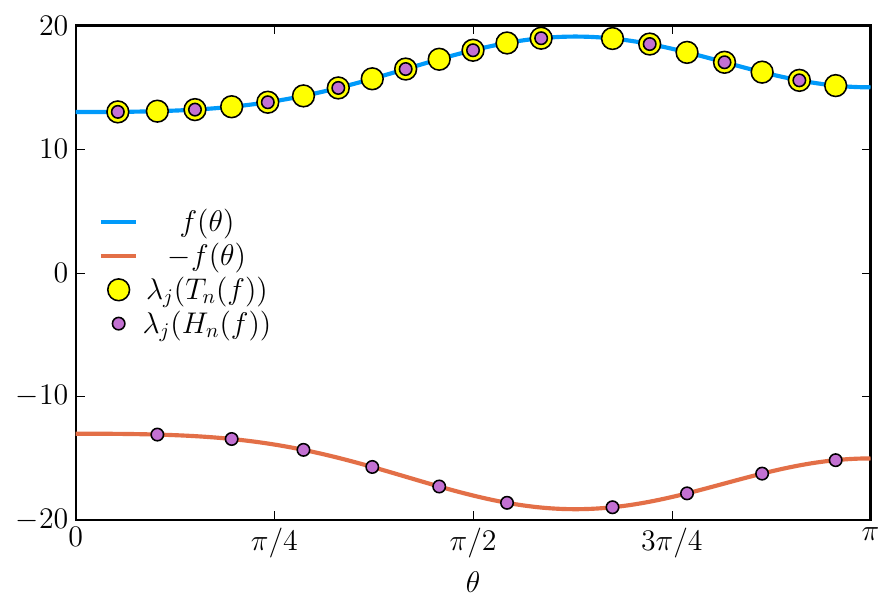}
    \caption{[Example \ref{exmp:num-b}] Non-monotone symbol $f(\theta)=16-2\cos(\theta)-2\cos(2\theta)+\cos(3\theta)$.}
    \label{fig:num-b}
    \end{figure}
\end{exmp}
\begin{exmp}
\label{exmp:num-c}
In the present example we observe a large discontinuity in the generating function at $\pi/2$. Indeed, the formal expression of $f$ is the following
\begin{align}
f(\theta)=
\begin{cases}
f(-\theta),&-\pi\leq\theta<0,\\
\cos(2\theta)+\cos(3\theta),&0\leq \theta<\pi/2,\\
\theta, &\pi/2\leq \theta\leq \pi.
\end{cases}
\end{align}
Again we compute the eigenvalues of $T_n(f)$ and $H_n(f)$ for $n=20$ and we present a possible distribution of eigenvalues, according to a computed grid $\xi_{j,n}$, in Figure~\ref{fig:num-c}.
Again the latter represents a numerical evidence of Theorem \ref{th:Carl_generalized}, since the presence of a finite number of local minima/maxima and discontinuity points does not spoil the result and a perfect grid taken from a a.u grid sequence in $[0,\pi]$ exists.

\begin{figure}[H]
    \centering
    \includegraphics[width=0.48\textwidth]{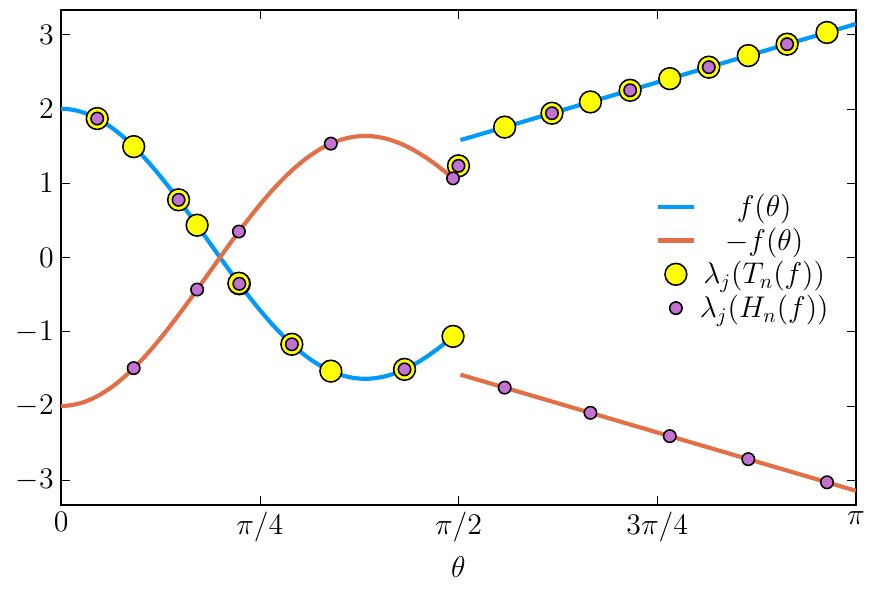}
    \caption{[Example \ref{exmp:num-c}] Symbol with finite number of local minima/maxima and discontinuity points.}
    \label{fig:num-c}
    \end{figure}
\end{exmp}
As Figure \ref{fig:num-a} shows, we stress again that the thesis of Theorem \ref{th:Stefano_version} is actually sharp and implies that the perfect grid does not exist. More precisely it is not true that $\lambda_i(T_n(f)) = f(\xi_{i,n})$ with $\{\xi_{1,n}, \xi_{2,n}, \dots, \xi_{n,n}\}_n$ a.u. grid in $[0,\pi]$ and actually the error term is necessary. In fact, Figure \ref{fig:num-a} shows that no grid points can be present in the whole and large subinterval $[0,\pi/2]$.
By the way this error term is of the form $e^{-cn}$ for some positive $c$ independent of $n$ for the minimal eigenvalue (see the combination of \cite{absolute-extreme-S,absolute-extreme-T}).
On the other hand, as reported in Figure \ref{fig:num-b} and Figure \ref{fig:num-c}, by strengthening a bit the assumptions, Theorem \ref{th:Carl_generalized} implies that the perfect grid sequence exists and it is a.u. in $[0,\pi]$, that is it is equally distributed with $\left\{\frac{i\pi}{n+1}\right\}_n$.

Therefore, if we admit an infinitesimal error, then we always have
\begin{equation}\label{error}
\lambda_i(T_n(f)) = f\left(\frac{i\pi}{n+1} \right) + \phi_{i,n}
\end{equation}
with $\phi_{i,n}$ uniformly converging to zero as $n$ tends to infinity, $i=1,\ldots,n$. In the case treated in Theorem \ref{th:Carl} that is when the generating function $f$ is monotone on $[0,\pi]$ and smooth, the error given in equation (\ref{error}) can be expanded asymptotically as
\begin{equation}\label{error-monotone}
 \phi_{i,n} = c_1\left(\frac{i\pi}{n+1} \right) h + c_2\left(\frac{i\pi}{n+1} \right) h^2 \cdots + c_k\left(\frac{i\pi}{n+1} \right) h^k + o(h^k)
\end{equation}
for given fixed functions $c_1,\ldots,c_k$ and with $h=\frac{1}{n+1}$. The asymptotic expansion in (\ref{error-monotone}) allows the use of linear in time matrix-less extrapolation-interpolation procedures for the computation of all the spectrum of $T_n(f)$, with large matrix-order $n$, given the eigenvalues of $T_{n_j}(f)$, $j=0,\ldots,q$, $q$ small and with $n_0<n_1<\cdots < n_q$, $n_q$ being moderate compared with $n$. See \cite{EFS,matrix-less1,EGS} and references therein, where also several numerical experiments are carried out, also concerning the statement in  Theorem \ref{th:Carl}.

In the non-monotone setting, if a given $\lambda_j(T_n(f))$ belongs to the interior part of $[a,b]$ such that $I=f^{-1}|_{[a,b]}$, that is $f$ restricted on $I$ is still monotonic, then (\ref{error-monotone}) is still valid and the same type of procedures can be applied. Otherwise, when the equation $f(\theta)=\lambda_j(T_n(f))$ has more than one solution, then the considered algorithms fail.

A way for recovering at least partially the good performance of the matrix-less procedures relies in employing the monotone rearranged function of $f$: for the notion of rearrangement see Subsection \ref{ssec:aux} and for its use in the context of matrix-less algorithms refer to \cite{EM22,exp-rearranged}. 

A discrete version of the rearrangement is represented by the use of permutations and this is precisely what we try in the subsequent lines. Indeed for approximating  $\lambda_i(T_n(f))$ and $\lambda_i(H_n(f))$ we will use 
$f\left(\frac{\Pi_n(i)\pi}{n+1}\right)$ instead of the already reasonable approximation $f\left(\frac{i\pi}{n+1} \right)$.
\begin{remark}
\label{rmrk:nonmon}
In the case of a non-monotone $f$ in Theorem \ref{th:Carl_generalized} and Theorem~\ref{th:Stefano_version}, the best ordering of the eigenvalues and hence the ``perfect grid'' $\xi_{j,n}$  is not obvious if we want to rely on an asymptotic expansion as in the monotone setting. Hence, in these cases the eigenvalues of $H_n(f)$ might give insights regarding the correct ordering of the eigenvalues of $T_n(f)$. A substantial benefit is that a matrix-less method could then be employed also for some non-monotone symbols (see \cite{matrix-less1} and references there reported).

Take, for example, the symbol $f(\theta)=\cos(\theta)+\cos(2\theta)$. In the left panel of Figure~\ref{fig:ordering} we see the symbol $f(\theta)$ (green line), equispaced samplings of the symbol $f(\theta_{j,n})$ (green circles), and the eigenvalues of $T_n(f)$ (yellow circles), for $n=10$. The eigenvalues are ordered with a permutation $\Pi_n^{-1}(j)$, such that, $f(\theta_{\Pi_n(1),n})\leq\ldots\leq f(\theta_{\Pi_n(n),n})$ (since $f$ is non-monotone different grid samplings might coincide, in that case we order the samplings $f(\theta_{\Pi_n(j),n})$ according to the index $j$).
The ``perfect grid'' $\xi_{j,n}$, such that $f(\xi_{j,n})=\lambda_j(T_n(f))$, is computed numerically for the ordering given the permutation $\Pi_n^{-1}(j)$. However, since the symbol $f$ is non-monotone, we expect that the ordering $\Pi_n(j)^{-1}$ might be not the most precise (we define it on a grid $\theta_{j,n}$ instead of the true unknown, and non-unique, $\xi_{j,n}$). Hence, the approximated $\xi_{j,n}$ might be not optimal.

\begin{figure}[H]
\centering
\includegraphics[width=0.48\textwidth]{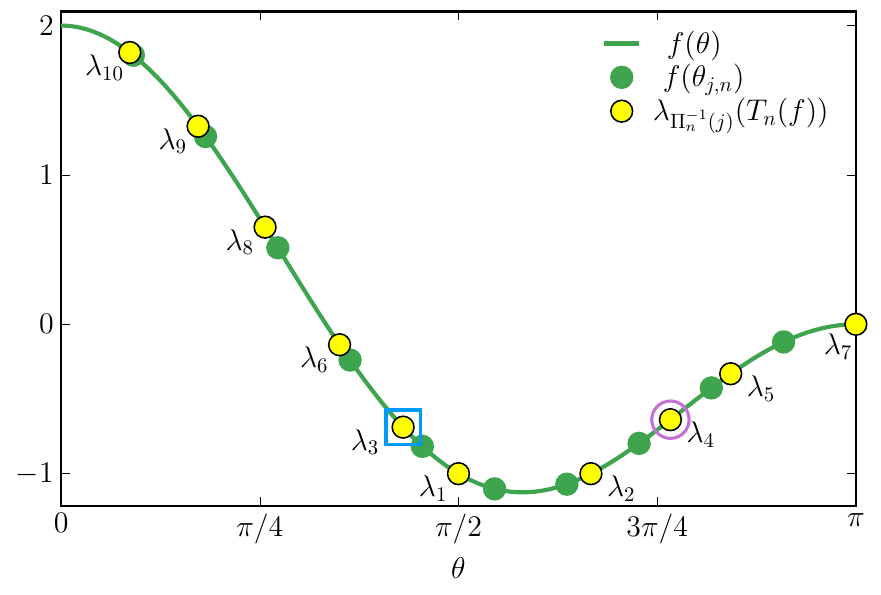}
\includegraphics[width=0.48\textwidth]{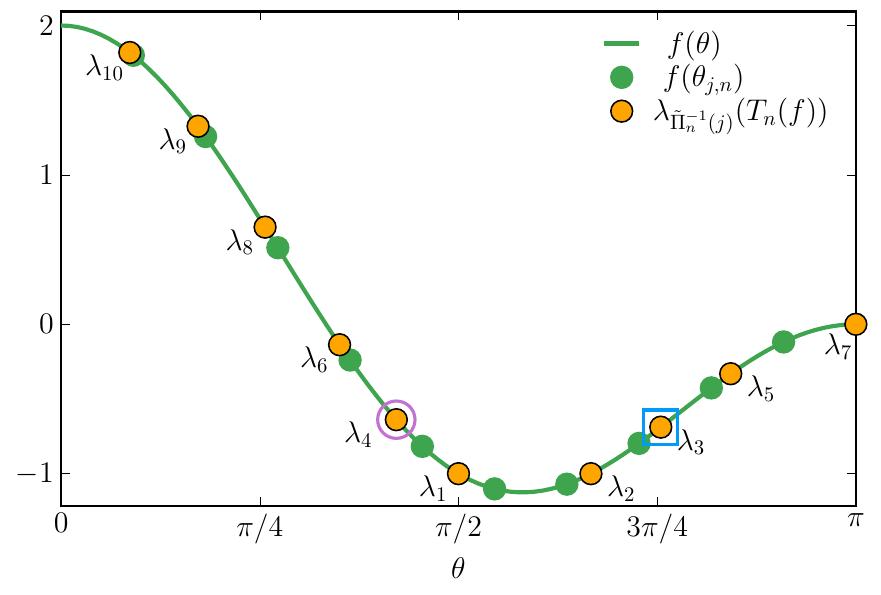}
\caption{[Ordering of eigenvalues (non-monotone symmetric symbol, $f(\theta)=\cos(\theta)+\cos(2\theta)$))] Left: Ordering $\Pi_n^{-1}(j)$ (match with symbol samplings $f(\theta_{j,n})$).   Right: Ordering $\tilde{\Pi}_n^{-1}(j)$, where two eigenvalues ($\lambda_3$, blue square, and $\lambda_4$, pink cicle) switch places, due to a mismatch with the signs of the eigenvalues of $H_n(f)$.}
\label{fig:ordering}
\end{figure}

In  Table~\ref{tbl:ordering1} we first present the permutations $\Pi_n(j)$ and $\Pi_n^{-1}(j)$, and the samplings of the symbol $f(\theta)$ with the grid $\theta_{j,n}=j\pi/(n+1)$ for $j=1,\ldots, n$, and $n=10$.

Then, the eigenvalues $\lambda_{\Pi_n^{-1}(j)}(T_n(f))$ are listed. For a correct ordering of the eigenvalues we assume, from Remark \ref{rmrk:nonmon}, that $(-1)^{j+1}\lambda_{\Pi_n^{-1}(j)}(T_n(f))$ should coincide with $\lambda_{\rho(j)}(H_n(f))$ (where $\rho(j)$ is a permutation such that $|\lambda_{\Pi_n^{-1}(j)}(T_n(f))|=|\lambda_{\rho(j)}(H_n(f))|$; of course in the rare case of multiplicity larger than one then the permutation is adapted). Highlighted in red, in the last two columns, we see a mismatch in signs for the fifth and eighth eigenvalues (of the permuted ordering $\Pi_n^{-1}(j)$, that is the original eigenvalues $\lambda_3$ (blue square) and $\lambda_4$ (pink circle)). These two eigenvalues are highlighted in the left panel of Figure~\ref{fig:ordering}; blue square and pink circle respectively.

\begin{table}[H]
\centering
\caption{[Ordering of eigenvalues (non-monotone symmetric symbol, $f(\theta)=\cos(\theta)+\cos(2\theta)$))] Ordering of the eigenvalues using permutation $\Pi_n^{-1}(j)$ gives a mismatch of two values (signs), highlighted in red, of $(-1)^{j+1}\lambda_{\Pi_n^{-1}(j)}(T_n(f))$ and $\lambda_{\rho(j)}(H_n(f))$.}
\label{tbl:ordering1}
\begin{tabular}{rrr|r|rrr}
\toprule
$\scriptstyle j$&$\scriptstyle \Pi_n(j)$&$\scriptstyle \Pi_n^{-1}(j)$&$\scriptstyle f(\theta_{j,n})$&$\scriptstyle\lambda_{\Pi_n^{-1}(j)}(T_n(f))$&$\scriptstyle(-1)^{j+1}\lambda_{\Pi_n^{-1}(j)}(T_n(f))$&$\scriptstyle\lambda_{\rho(j)}(H_n(f))$\\
\midrule
 1&6&10& 1.8007&1.8193$\hphantom{^\smallsquare}$&1.8193&1.8193\\
 2&7&9& 1.2567&1.3255$\hphantom{^\smallsquare}$&-1.3255&-1.3255\\
 3&5&8& 0.5125&0.6502$\hphantom{^\smallsquare}$&0.6502&0.6502\\
 4&8&6&-0.2394&-0.1369$\hphantom{^\smallsquare}$&0.1369&0.1369$\hphantom{^\smallsquare}$\\
 5&9&3&-0.8172&{\color{myblue}-0.6886$^\smallsquare$}&{\color{myred}-0.6886}&{\color{myred}0.6886}\\
 6&4&1&-1.1018&-1.0000$\hphantom{^\smallsquare}$&1.0000&1.0000\\
 7&10&2&-1.0703&-1.0000$\hphantom{^\smallsquare}$&-1.0000&-1.0000\\
 8&3&4&-0.7972&{\color{mypink}-0.6386$^\circ$}&{\color{myred}0.6386}&{\color{myred}-0.6386}\\
 9&2&5&-0.4258&-0.3309$\hphantom{^\smallsquare}$&-0.3309&-0.3309\\
 10&1&7&-0.1182&0.0000$\hphantom{^\smallsquare}$&0.0000&0.0000\\
\bottomrule
\end{tabular}
\end{table}

\noindent In  Table~\ref{tbl:ordering2} we show the alternative permutations $\tilde{\Pi}_n(j)$ and $\tilde{\Pi}_n^{-1}(j)$; defined such that $\lambda_3$ and $\lambda_4$ are switched in $\tilde{\Pi}_n(j)$. As seen, the values $(-1)^{j+1}\lambda_{\tilde{\Pi}_n^{-1}(j)}(T_n(f))$ and $\lambda_{\tilde{\rho}(j)}(H_n(f))$ (where $\tilde{\rho}(j)$ is the equivalent ordering as $\rho(j)$ but for $\tilde{\Pi}_n^{-1}(j)$ instead of $\Pi_n^{-1}(j)$) now match.
In the right panel of Figure~\ref{fig:ordering} we report the eigenvalues with the new ordering $\tilde{\Pi}_n^{-1}(j)$.

\begin{table}[H]
\centering
\caption{[Ordering of eigenvalues (non-monotone symmetric symbol, $f(\theta)=\cos(\theta)+\cos(2\theta)$))] Ordering of the eigenvalues using permutation $\tilde{\Pi}_n^{-1}(j)$, where two eigenvalues have switched places, does not give a mismatch of $(-1)^{j+1}\lambda_{\tilde{\Pi}_n^{-1}(j)}(T_n(f))$ and $\lambda_{\tilde{\rho}(j)}(H_n(f))$ (as in Table~\ref{tbl:ordering1}).}
\label{tbl:ordering2}
\begin{tabular}{rrr|r|rrr}
\toprule
 $\scriptstyle j$&$\scriptstyle \tilde{\Pi}_n(j)$&$\scriptstyle \tilde{\Pi}_n^{-1}(j)$&$\scriptstyle f(\theta_{j,n})$&$\scriptstyle\lambda_{\tilde{\Pi}_n^{-1}(j)}(T_n(f))$&$\scriptstyle(-1)^{j+1}\lambda_{\tilde{\Pi}_n^{-1}(j)}(T_n(f))$&$\scriptstyle\lambda_{\tilde{\rho}(j)}(H_n(f))$\\
 \midrule
 1&6&10& 1.8007&1.8193$\hphantom{^\smallsquare}$&1.8193&1.8193\\
 2&7&9& 1.2567&1.3255$\hphantom{^\smallsquare}$&-1.3255&-1.3255\\
 3&8&8& 0.5125&0.6502$\hphantom{^\smallsquare}$&0.6502&0.6502\\
 4&5&6&-0.2394&-0.1369$\hphantom{^\smallsquare}$&0.1369&0.1369\\
 5&9&4&-0.8172&{\color{mypink}-0.6386$^\circ$}&-0.6386&-0.6386\\
 6&4&1&-1.1018&-1.0000$\hphantom{^\smallsquare}$&1.0000&1.0000\\
 7&10&2&-1.0703&-1.0000$\hphantom{^\smallsquare}$&-1.0000&-1.0000\\
 8&3&3&-0.7972&{\color{myblue}-0.6886$^\smallsquare$}&0.6886&0.6886\\
 9&2&5&-0.4258&-0.3309$\hphantom{^\smallsquare}$&-0.3309&-0.3309\\
 10&1&7&-0.1182&0.0000$\hphantom{^\smallsquare}$&0.0000&0.0000\\
\bottomrule
\end{tabular}
\end{table}
\end{remark}

\subsection{Numerical verification 1 of Remark~\ref{rmrk:nonmon}}
\label{sec:numver1}

A closely related matrix, to the Toeplitz matrix $T_n(f)$ in Remark~\ref{rmrk:nonmon}, is the Toeplitz-like matrix $T_{n,0,0}(f)=T_n(f)+R_{n,0,0}(f)$ where $R_{n,0,0}(f)$, for $f(\theta)=\cos(\theta)+\cos(2\theta)$, is a low-rank matrix with $-1/2$ in the top left and bottom right corners. More generally, for a generic real-valued, even, continuous function $f$,  by $T_{n,\varepsilon,\varphi}(f)$ we indicate the matrix belonging to the $\tau_{\varepsilon,\varphi}$-algebra (see e.g.  \cite{BC-tau,Bolten2022,taubozzo,taueconomy}), generated by the function $f$. 
In this setting the $\tau_{\varepsilon,\varphi}$-algebra is generated by the tridiagonal matrix
\begin{align*}
T_{n,\varepsilon,\varphi}= T_{n,\varepsilon,\varphi}(2\cos(\theta))&=T_n(2\cos(\theta))+R_{n,\varepsilon,\varphi}\\
&=\left[
\begin{array}{rrrrrrrrrrr}
\varepsilon &1\\
1 &0 &1\\
&\ddots&\ddots&\ddots\\
&&1&0&1\\
&&&1 &\varphi
\end{array}
\right]
\end{align*}
where $\varepsilon,\varphi\in\mathbb{R}$. For $\varepsilon,\varphi\in\{-1,0,1\}$ the eigendecomposition $T_{n,\varepsilon,\varphi}(f)=\mathbb{Q}_{n,\varepsilon,\varphi}\mathrm{diag}(f(\theta_{j,n,\varepsilon,\varphi}))\mathbb{Q}_{n,\varepsilon,\varphi}^T$ is known explicitly, where for every combination $\varepsilon,\varphi$ we have an orthogonal eigenvector matrix $\mathbb{Q}_{n,\varepsilon,\varphi}$ and an equispaced grid $\theta_{j,n,\varepsilon,\varphi}$. Notice that for $\varepsilon =\varphi =0$, we obtain that the generator $T_{n,0,0}= T_{n,0,0}(2\cos(\theta))=T_n(2\cos(\theta))$ and $R_{n,0,0}=0$ so that the algebra $\tau_{0,0}$ is the standard $\tau$ algebra \cite{BC-tau}, containing the standard one-dimensional Laplacian with generating function $2-2\cos(\theta)$.

Hence, for all matrices $T_{n,\varepsilon,\varphi}(f)$, where $\varepsilon,\varphi\in\{-1,0,1\}$ we know the full eigendecomposition and ``perfect grids'' $\xi_{j,n}=\theta_{j,n,\varepsilon,\varphi}$; e.g., $\lambda_j(T_{j,0,0}(f))=f(\theta_{j,n,0,0})$, $\theta_{j,n,0,0}=j\pi/(n+1)$ and $\mathbb{Q}_{n,0,0}$ is the discrete sine transform, DST.
 Therefore, the eigenvectors of $T_n(f)$ are closely related to the DST (asymptotically they coincide), and we can assume that eigenvector $\mathbf{v}_{\Pi_n^{-1}(j)}$ of $T_n(f)$ should behave approximately as $\sin(j\theta_{i,n})$.

In Figure~\ref{fig:evflip} we show the fifth (left panels) and eighth (right panels) DST (non-normalized) eigenvectors ($\sin(5\theta)$ and $\sin(8\theta)$).

In the top two panels of Figure~\ref{fig:evflip} we show the fifth and eighth eigenvectors using the permutation $\Pi_n^{-1}(j)$ (the original $\lambda_3(T_n(f))$ and $\lambda_4(T_n(f))$). A clear mismatch is present and perturbing the grid $\theta_{j,n}$ will not yield a ``perfect grid'' $\xi_{j,n}$ to match the eigenvector elements with the DST.
In the bottom two  panels of Figure~\ref{fig:evflip} we show the fifth and eighth eigenvectors using the permutation $\tilde{\Pi}_n^{-1}(j)$ (the original $\lambda_4(T_n(f))$ and $\lambda_3(T_n(f))$). A much better match between the eigenvector elements and the DST is present.

\begin{figure}[H]
\centering
\includegraphics[width=0.48\textwidth]{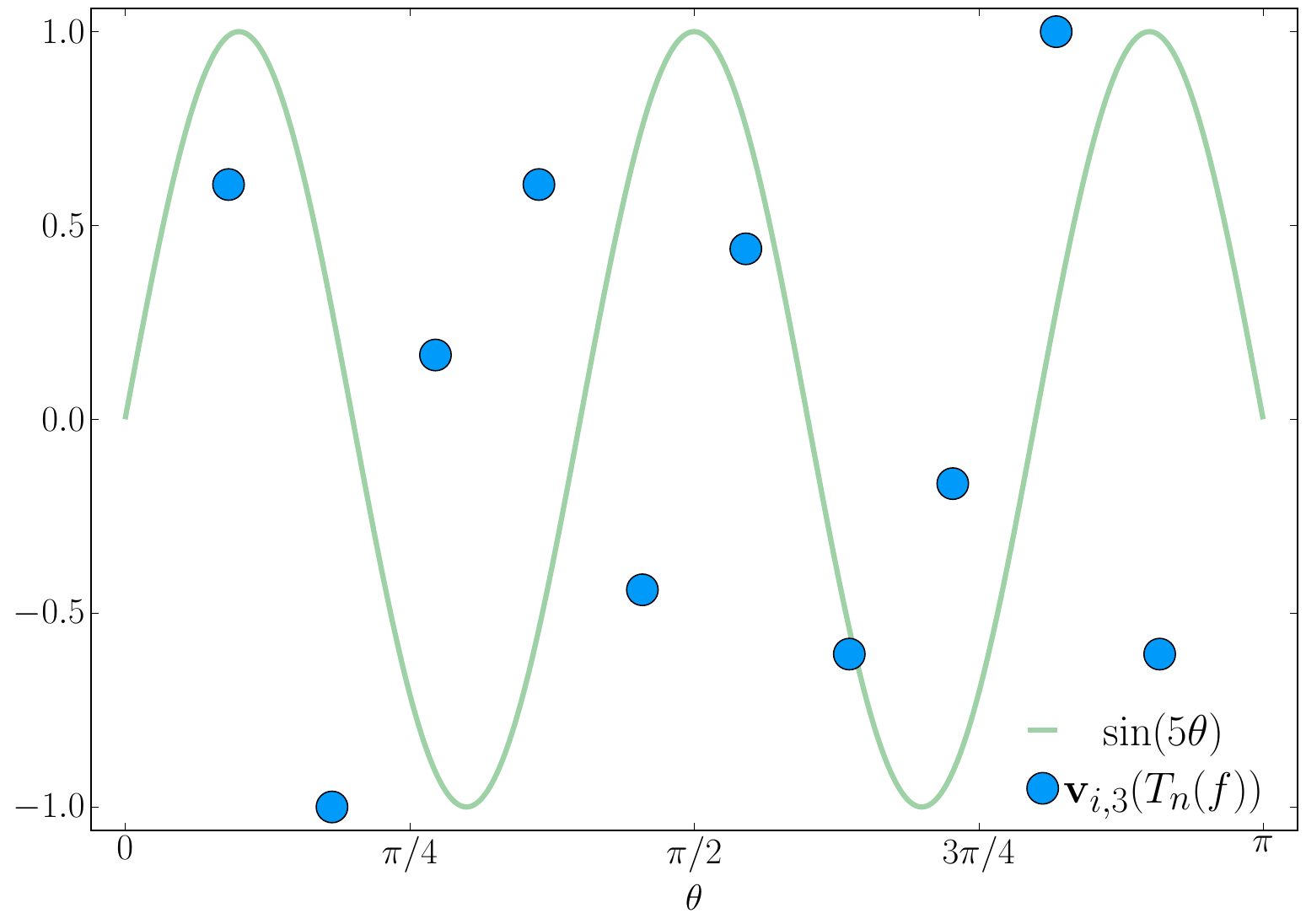}
\includegraphics[width=0.48\textwidth]{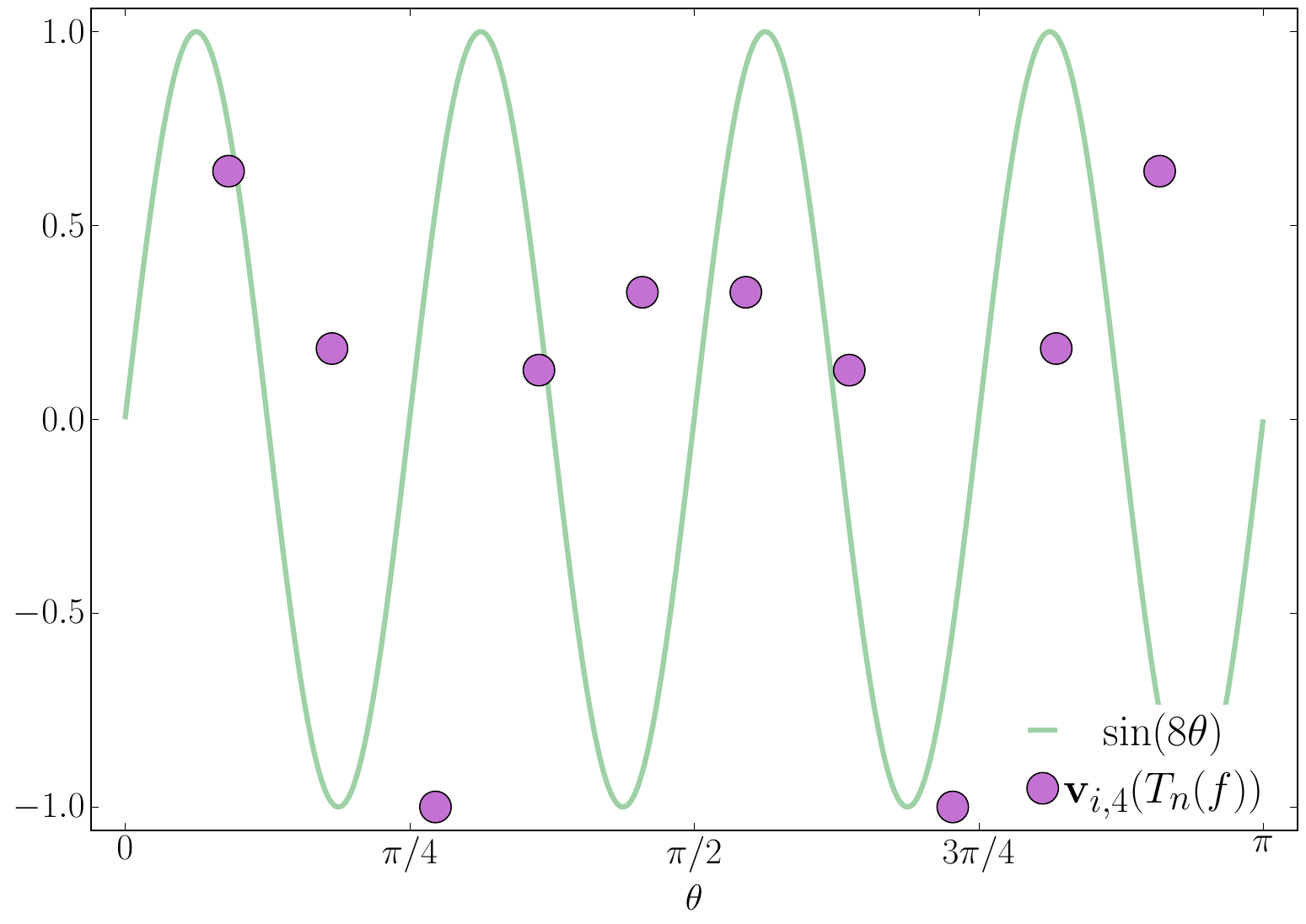}

\includegraphics[width=0.48\textwidth]{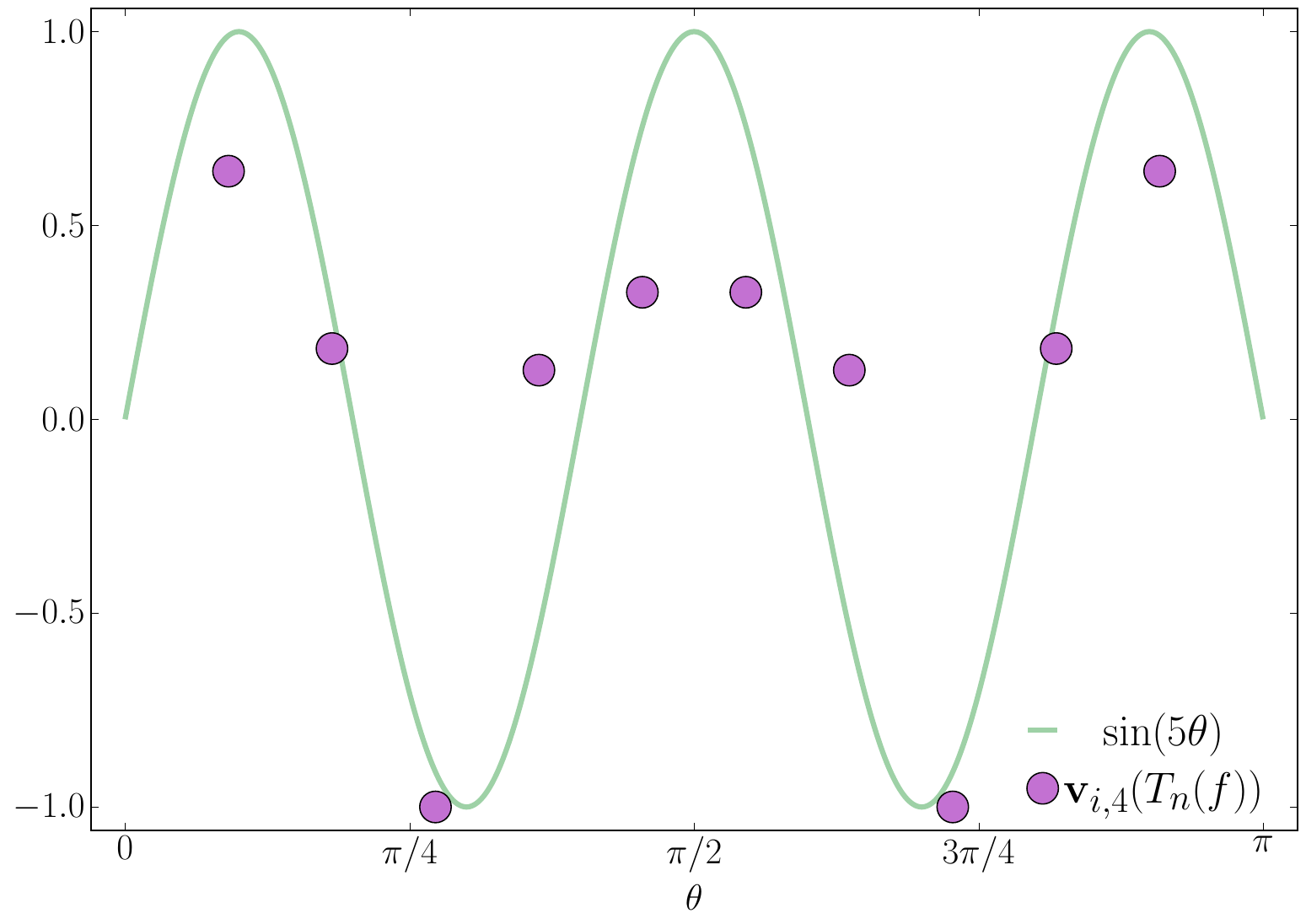}
\includegraphics[width=0.48\textwidth]{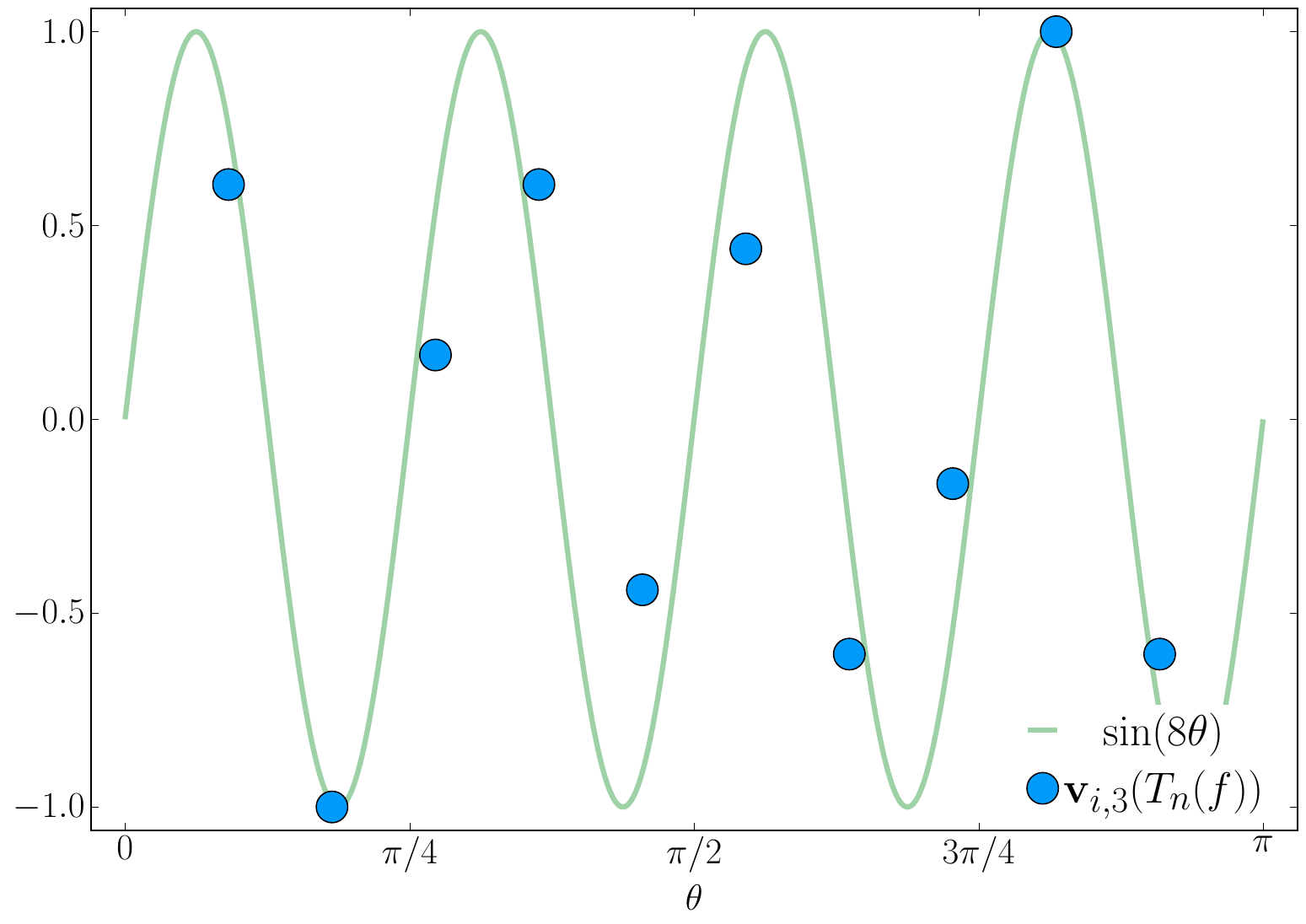}
\caption{[Ordering of eigenvalues (non-monotone symmetric symbol, $f(\theta)=\cos(\theta)+\cos(2\theta)$)] Left panels: Eigenvector five, with permutation $\Pi_n^{-1}(j)$ (top) and $\tilde{\Pi}_n^{-1}(j)$ (bottom). Right panels: Same as left panels, but for eigenvector eight.}
\label{fig:evflip}
\end{figure}

In Figure~\ref{fig:evflip2} we present a numerically computed (non-unique) grid $\xi_{j,n}$ such that the eigenvector elements matches the DST. Again, note that the DST is simply an approximation of the eigenvectors of $T_n(f)$, but Figures~\ref{fig:evflip} and \ref{fig:evflip2} are good indications that the modified permutation $\tilde{\Pi}_n^{-1}(j)$ is a better ordering of the eigenvalues, than $\Pi_n^{-1}(j)$.
\begin{figure}[H]
\centering
\includegraphics[width=0.48\textwidth]{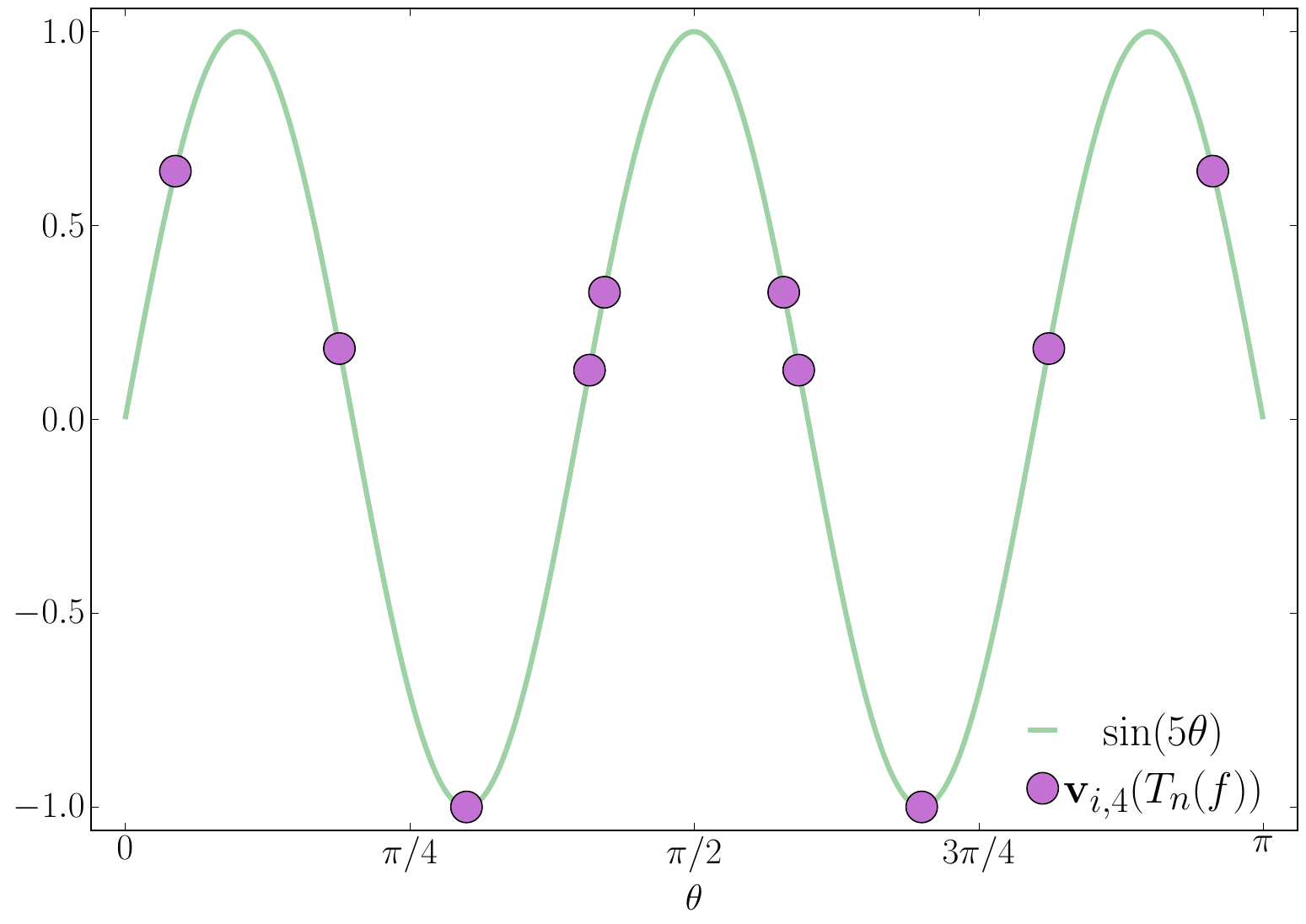}
\includegraphics[width=0.48\textwidth]{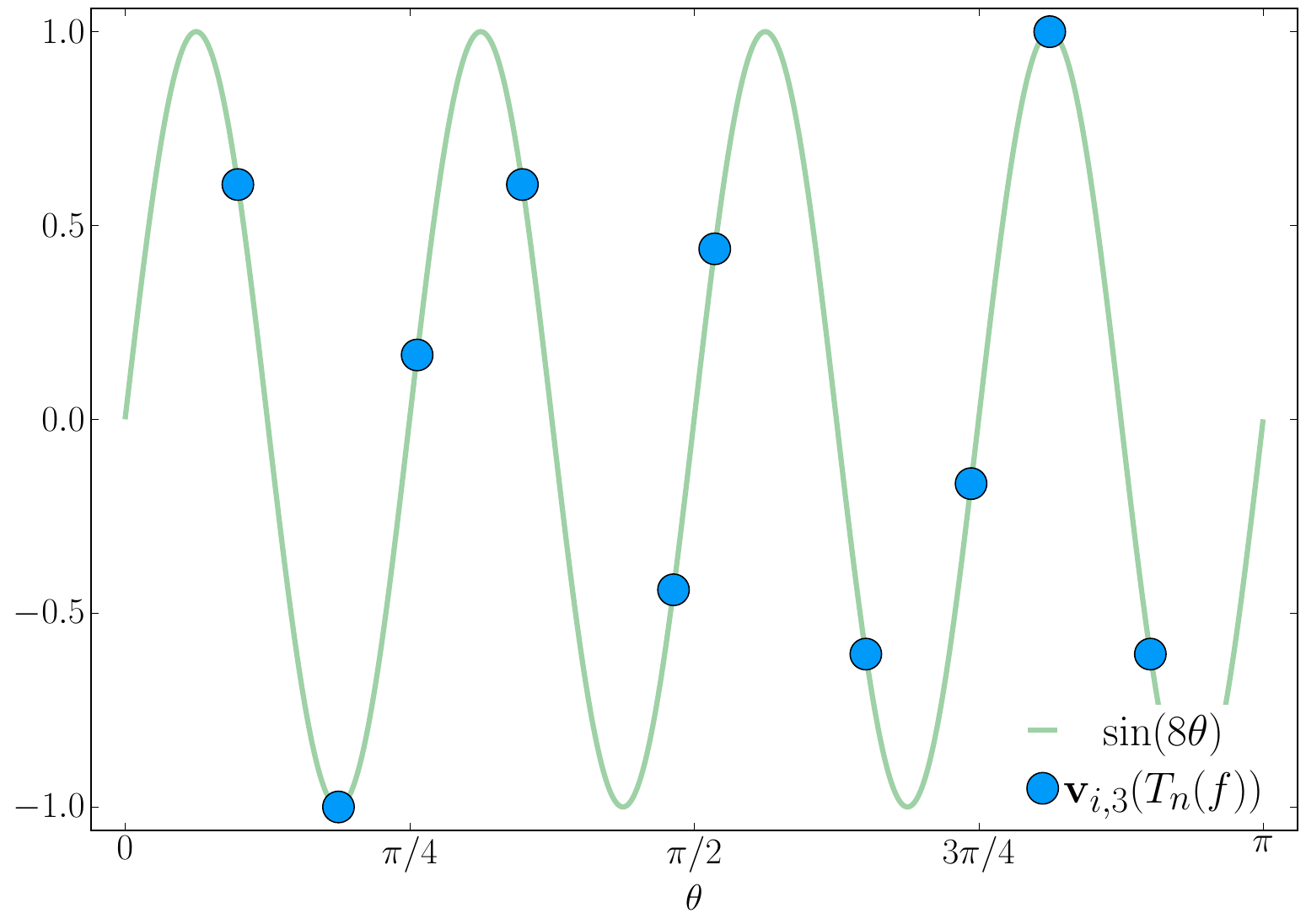}
\caption{[Ordering of eigenvalues (non-monotone symmetric symbol, $f(\theta)=\cos(\theta)+\cos(2\theta)$)] Left panel: Eigenvector five,  $\tilde{\Pi}_n^{-1}(j)$ using a numerically computed ``perfect grid''. Right panel: Same as left panel, but for eigenvector eight.}
\label{fig:evflip2}
\end{figure}

\subsection{Numerical verification 2 of Remark~\ref{rmrk:nonmon}}
\label{sec:numver2}
In Section~\ref{sec:numver1} we assumed that we knew that eigenvectors five and eight (after reordering according to the symbol) should be switched. If this type of ``true ordering'' of eigenvalues of Toeplitz matrices $T_n(f)$, generated by non-monotone symbols, is to be used in practical applications (e.g., matrix-less methods) we suggest a more automatic approach. We here propose an outline of an algorithm to find the ``true ordering'' of the eigenvalues, for a given $n$.
The algorithm can be summarized as follows:
We are interested in the ordering of the eigenvalues of a matrix $T_n(f)$, generated by the symbol $f(\theta)$.
Construct a matrix $\hat{T}_n(f)$, which has the same symbol $f$, where the full eigendecomposition is known. This can for example be the matrices generated by the $\tau_{\varepsilon,\varphi}$-algebras~\cite{BC-tau,Bolten2022,taubozzo,taueconomy}. 
Now define $R_n=\hat{T}_n(f)-T_n(f)$ and $B_n^{(\gamma)}=\hat{T}_n(f)-\gamma R_n$, and $\gamma\in[0,1]$. We now assume that the elements of same index eigenvectors for matrices $B_n^{(\gamma)}$ vary continuously as $\gamma$ is varied from zero to one. Hence, we here study the matrix sequence $\{B_n^{(\gamma)}\}_\gamma$, for $\gamma\in[0,1]$.

\begin{algo}[Automatic ordering $\tilde{\Pi}_n^{-1}(j)$ of eigenvalues]
\label{algo:ordering}
\

\begin{enumerate}
\item Define:
\begin{itemize}
\item $T_n(f)$: Matrix of interest (we note that this approach should work for more general Toeplitz-like matrices);
\item $f(\theta)$: Symbol of the matrix $T_n(f)$ (here assumed to be univariate, scalar valued and real-symmetric for simplicity);
\item $\hat{T}_{n}(f)$: Matrix, with symbol $f$, for which full eigendecomposition is known (e.g., $T_{n,\varepsilon,\varphi}(f)$, $\varepsilon,\varphi\in\{-1,0,1\}$);
\item $R_n=\hat{T}_{n}(f)-T_n(f)$: Low-rank matrix such that $B_n^{(\gamma)}=\hat{T}_{n}(f)-\gamma R_n$, $\gamma\in[0,1]$. $B_n^{(0)}=\hat{T}_{n}(f)$ and $B_n^{(1)}=T_n(f)$;
\item $N_{steps}$: Number of matrices to generate in the algorithm, from $\hat{T}_{n}(f)$ and $T_n(f)$;
\item $\gamma_k=(k-1)/(N_{steps}-1),\quad k=1,\ldots,N_{steps}$
\end{itemize}
\item We have, $\lambda_j(B_n^{(0)})=f(\xi_{j,n})$ and $\mathbf{v}_j(B_n^{(0)})$ (e.g., if $\hat{T}_n(f)=T_{n,0,0}(f)$, we have $\xi_{j,n}=j\pi/(n+1)$ and $\mathbf{v}_j(B_n^{(0)})=\sqrt{2/(n+1)}\sin(j\xi_{i,n})$).
\item Iterate $k=2,\ldots, N_{steps}$
\begin{enumerate}
\item Numerically compute the eigenvalues and eigenvectors of $B_n^{(\gamma_k)}$, called $\lambda_j(B_n^{(\gamma_k)})$ and $\mathbf{v}_j(B_n^{(\gamma_k)})$. Ordering is given by the numerical solver (often in non-decreasing order);
\item Iterate $j=1,\ldots, n$, and minimize $\epsilon_j=\|\mathbf{v}_r(B_n^{(\gamma_{k-1})})-\mathbf{v}_j(B_n^{(\gamma_k)})\|_2$ for $r\in {1,\ldots, n}$. Call each of these indices $r_j$; this corresponds to the ordering $\tilde{\rho}(j)$ for $\lambda_j(B_n^{(\gamma_k)})$. Take into account,
\begin{itemize}
    \item comparison with both $\mathbf{v}_j(B_n^{(\gamma_k)})$ and $-\mathbf{v}_j(B_n^{(\gamma_k)})$;
    \item for less computational effort, only consider the eigenvectors $\mathbf{v}_r(B_n^{(\gamma_{k-1})})$ with $r$ corresponding to eigenvalues $\lambda_r(B_n^{(\gamma_{k-1})})$ close to the eigenvalue $\lambda_j(B_n^{(\gamma_k)})$;
    \item two different eigenvectors $\mathbf{v}_j(B_n^{(\gamma_k)})$ may minimize the norm $\epsilon_j$ to the same eigenvector $\mathbf{v}_r(B_n^{(\gamma_{k-1})})$.
    \end{itemize}
    \item Reorder the eigenvalues and eigenvectors for step $k$ according to the ordering $r_j$, $j=1,\ldots,n$;
\end{enumerate}
\item The final ordering ($\tilde{\Pi}_n^{-1}(j)$) where $\gamma=1$, should be the ``true ordering'' of $T_n(f)$.
\end{enumerate}
\end{algo}

\noindent For the example in Remark~\ref{rmrk:nonmon}, we have $n=10$, $f(\theta)=\cos(\theta)+\cos(2\theta)$, and the choice of $\hat{T}_n(f)=T_{n,0,0}(f)$,
\begin{align*}
T_n(f)=\frac{1}{2}\left[\begin{array}{ccccccccccc}
0&1&1\\
1&0&1&1\\
1&1&0&1&1\\
&\ddots&\ddots&\ddots&\ddots&\ddots\\
&&1&1&0&1&1\\
&&&1&1&0&1\\
&&&&1&1&0
\end{array}\right],\qquad
\hat{T}_n(f)=T_{n,0,0}(f)=\frac{1}{2}\left[\begin{array}{ccccccccccc}
-1&1&1\\
1&0&1&1\\
1&1&0&1&1\\
&\ddots&\ddots&\ddots&\ddots&\ddots\\
&&1&1&0&1&1\\
&&&1&1&0&1\\
&&&&1&1&-1
\end{array}\right],
\end{align*}
and
\begin{align*}
R_n=\hat{T}_n(f)-T_n(f)=-\frac{1}{2}\left[\begin{array}{ccccccccccc}
1\\
&0\\
&&\ddots\\
&&&0\\
&&&&1
\end{array}\right],
\end{align*}
and $B_n^{(\gamma)}=T_{n,0,0}(f)-\gamma R_n$.

In Figure~\ref{fig:sequence_eig} we show the ten eigenvalues, ordered as $\Pi_n^{-1}(j)$ (left) and, automatically by Algorithm~\ref{algo:ordering}, $\tilde{\Pi}_n^{-1}(j)$ (right), for the matrices $B_n^{(\gamma)}$ as $\gamma$ is varied from 0 to 1 over $N_{steps}=100$ steps. The dashed black line indicates a $\gamma_b=(103-13\sqrt{41})/80\approx 0.247$ where eigenvalues five and eight of $\Pi_n^{-1}(j)$ switch places (visible as the eigenvalue curves cross, and $\lambda_5(B_n^{(\gamma_b)})=\lambda_8(B_n^{(\gamma_b)})=(9+\sqrt{41})/20\approx 0.770$).
\begin{figure}[H]
\centering
\includegraphics[width=0.48\textwidth]{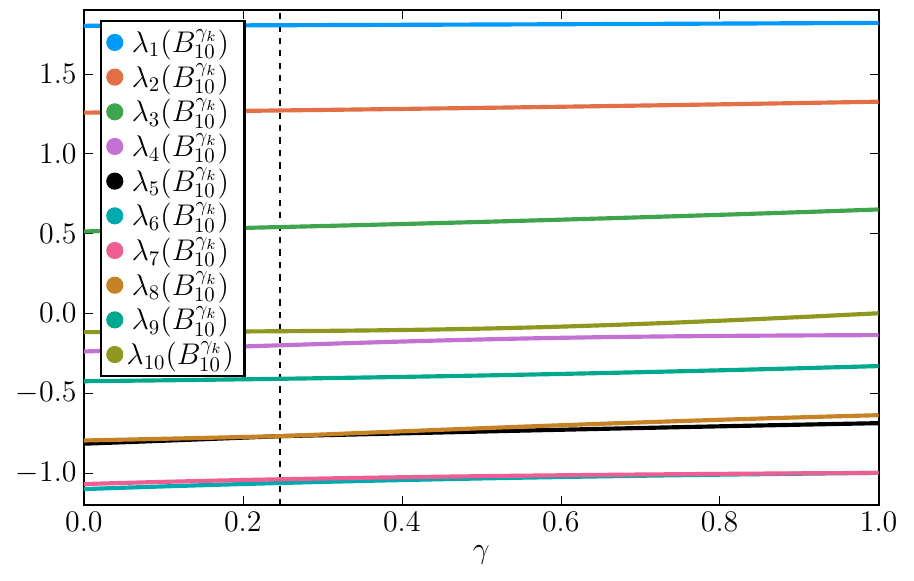}
\includegraphics[width=0.48\textwidth]{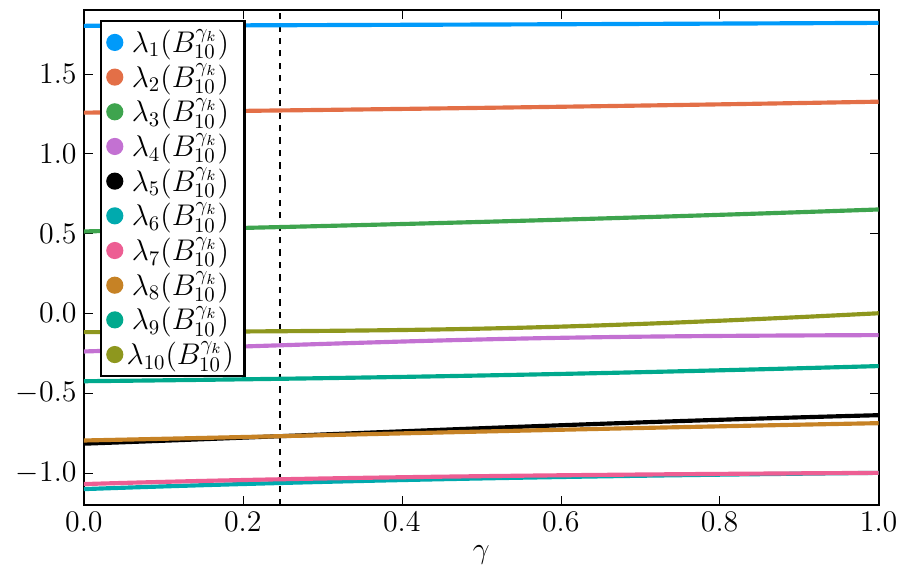}
\caption{[Ordering of eigenvalues (non-monotone symmetric symbol, $f(\theta)=\cos(\theta)+\cos(2\theta)$)] Eigenvalues for $B_{10}^{(\gamma_k)}$, $k=1,\ldots,N_{steps}$ with ordering  $\Pi_n^{-1}(j)$ (left) and $\tilde{\Pi}_n^{-1}(j)$ (right). In the left panel, the switch of $\lambda_5$ and $\lambda_8$ does not occur at $\gamma_b$ (dashed black line), whereas it does so in the right panel.}
\label{fig:sequence_eig}
\end{figure}
In Figure~\ref{fig:sequence} we show the eigenvector elements for eigenvector five (left) and eight (right) for the matrices $B_n^{(\gamma_k)}$, $k=1,\ldots,N_{steps}=100$. On top we sort the eigenvalues (and corresponding eigenvectors) solely comparing with the samplings $f(\theta_{j,n})$, i.e. $\Pi_n^{-1}(j)$. In the bottom we switch (automatically, using Algorithm~\ref{algo:ordering}) the eigenvectors five and eight for all $\gamma\geq\gamma_b$, and consequently we obtain the resulting ordering $\tilde{\Pi}_n^{-1}(j)$. Note that here $\lambda_5$ and $\lambda_8$ are the third and fourth eigenvalues in the original non-decreasing monotone ordering by the numerical solver.

Yellow circles denote the eigenvector elements for $\gamma=0$, that is $T_{n,0,0}(f)$. Blue circles correspond to the eigenvector elements for intermediate matrices; as $\gamma$ increases, so does the size of the circles. Red circles corresponds to the eigenvector elements of $T_n(f)$. All vectors in the figure are normalized (and with choice of sign to correspond to the signs of the vectors forming the DST matrix).

\begin{figure}[H]
\centering

\includegraphics[width=0.48\textwidth]{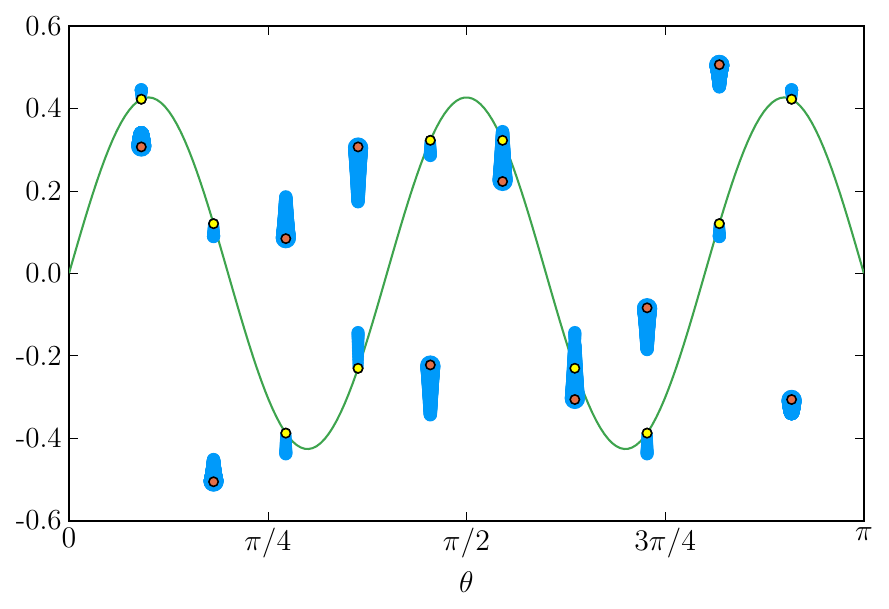}
\includegraphics[width=0.48\textwidth]{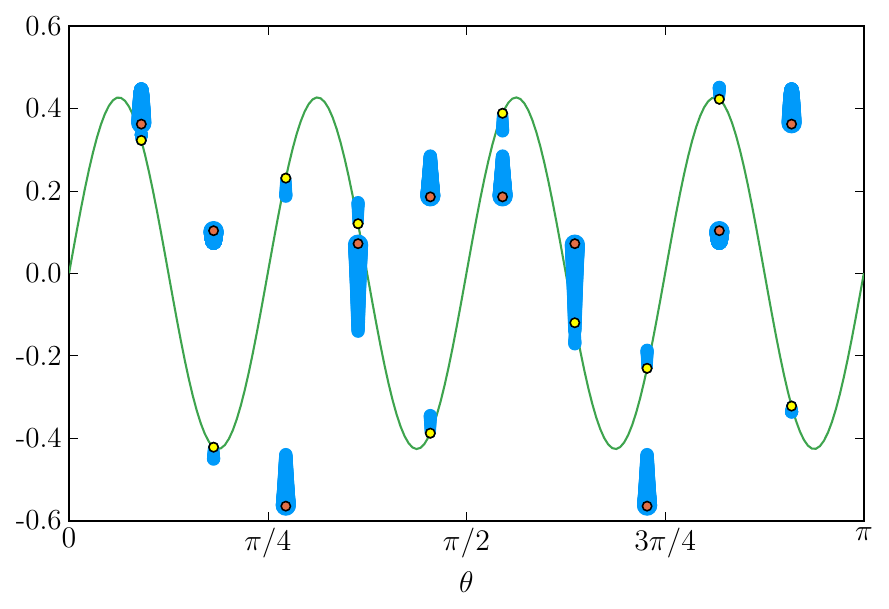}

\includegraphics[width=0.48\textwidth]{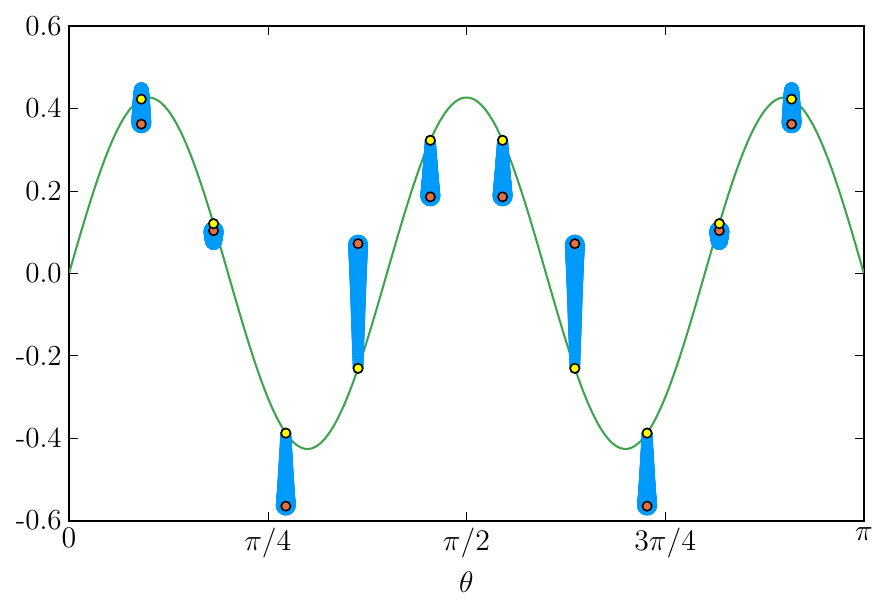}
\includegraphics[width=0.48\textwidth]{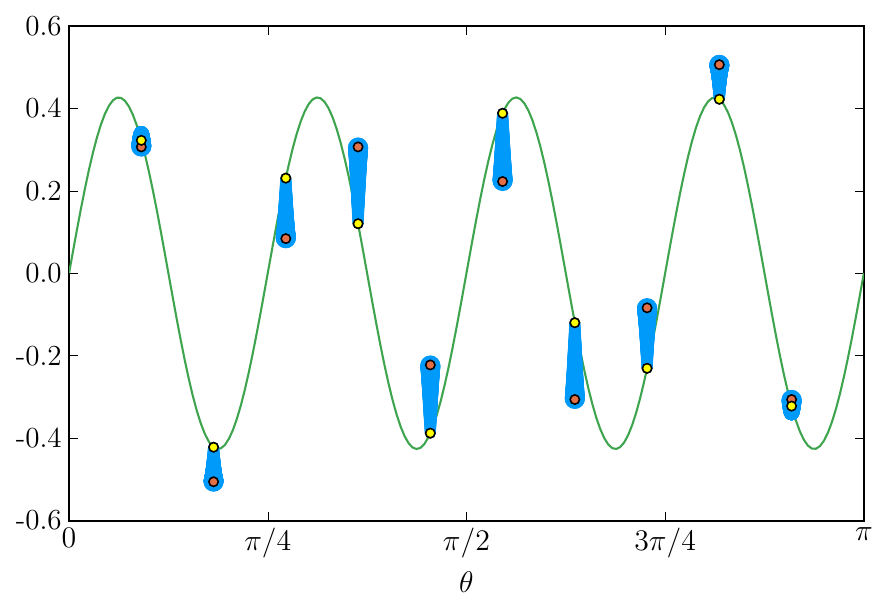}

\caption{[Ordering of eigenvalues (non-monotone symmetric symbol, $f(\theta)=\cos(\theta)+\cos(2\theta)$)] Eigenvectors five (left) and eight (right) for $B_{10}^{(\gamma_k)}$, $k=1,\ldots,N_{steps}$ with ordering $\Pi_n^{-1}(j)$ (top) and $\tilde{\Pi}_n^{-1}(j)$ (bottom). Clear erratic behavior is visible in the top panels where the switch does not occur for the ordering of eigenvalues five and eight.}
\label{fig:sequence}
\end{figure}

We here also notice that the degenerate eigenvalues $\lambda_6(T_n(f))=\lambda_7(T_n(f))=1$, of $T_n(f)$ for $n=10$, will typically yield ``erratic'' eigenvectors with respect to the sequence of eigenvector elements, as shown in Figure~\ref{fig:sequence_v67}.
\begin{figure}[H]
\centering
\includegraphics[width=0.48\textwidth]{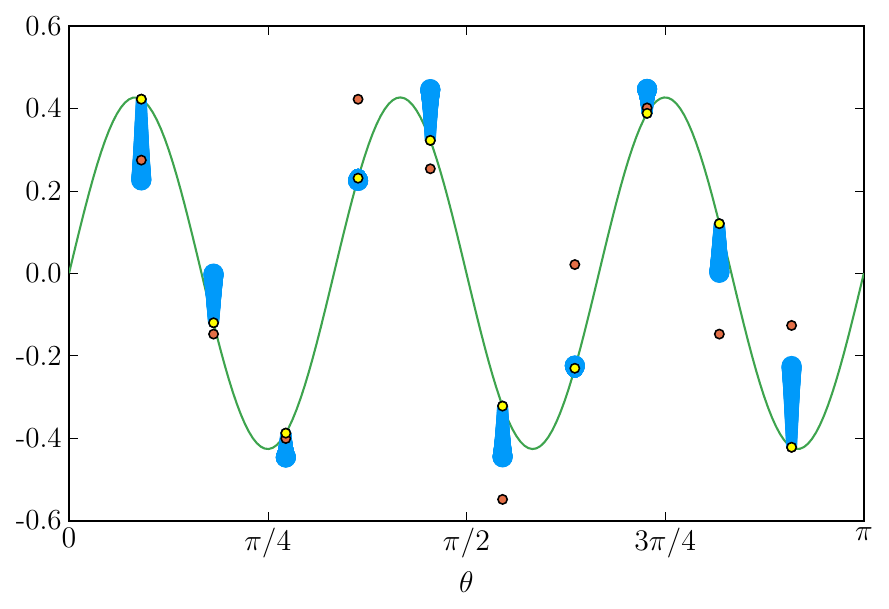}
\includegraphics[width=0.48\textwidth]{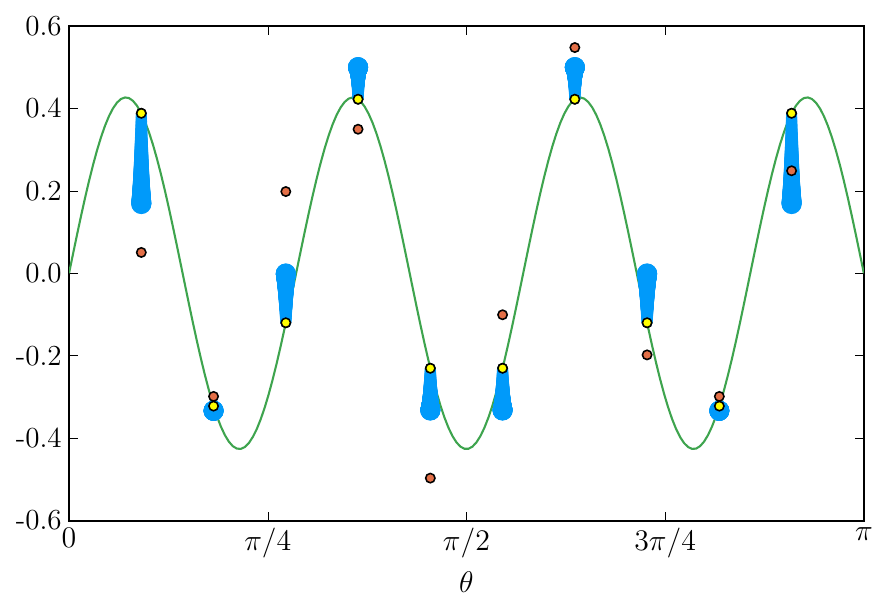}
\caption{[Ordering of eigenvalues (non-monotone symmetric symbol, $f(\theta)=\cos(\theta)+\cos(2\theta)$)] Eigenvectors six and seven, for $\gamma=1$ (red circles), correspond to degenerative eigenvalues (equal to one), and the numerically computed eigenvectors behave ``erratic'' with respect to the sequence of elements.}
\label{fig:sequence_v67}
\end{figure}

Finally, we mention that as shown in Figure~\ref{fig:sequence_v67}, also $\hat{T}_n(f)$ (or any of the matrices $B_n^{(\gamma_k)}$) can have degenerate eigenvalues. In Figure~\ref{fig:sequence_n9} we see that $T_{n,0,0}(f)$ (left) has two eigenvalues that coincide, whereas $T_{n,1,1}(f)$ (right) does not. However, $T_{n,1,1}(f)$ does have two eigenvalues that switch order (four and nine), but that is handled automatically by Algorithm~\ref{algo:ordering}. Also, the presumed signs of the eigenvalues of $H_n(f)$ match this computed ordering.
\begin{figure}[H]
\centering
\includegraphics[width=0.48\textwidth]{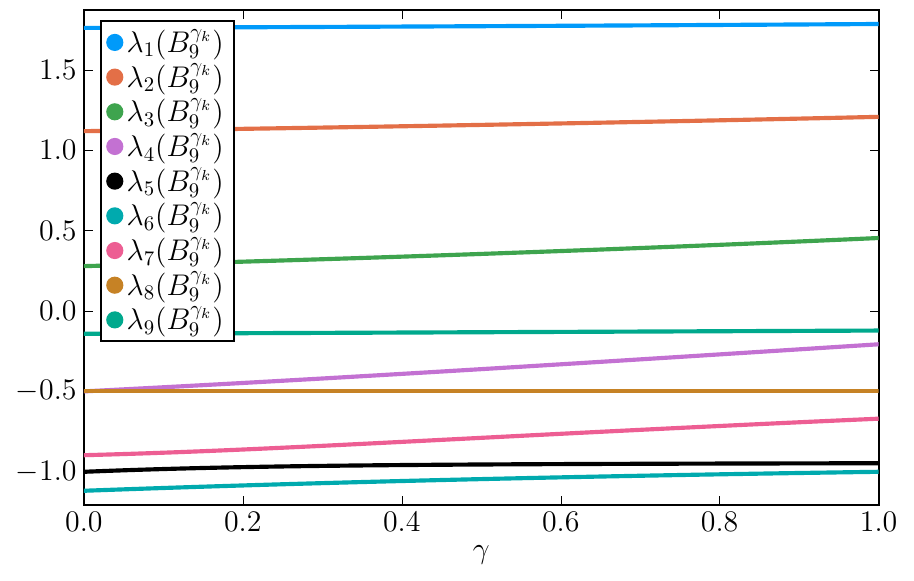}
\includegraphics[width=0.48\textwidth]{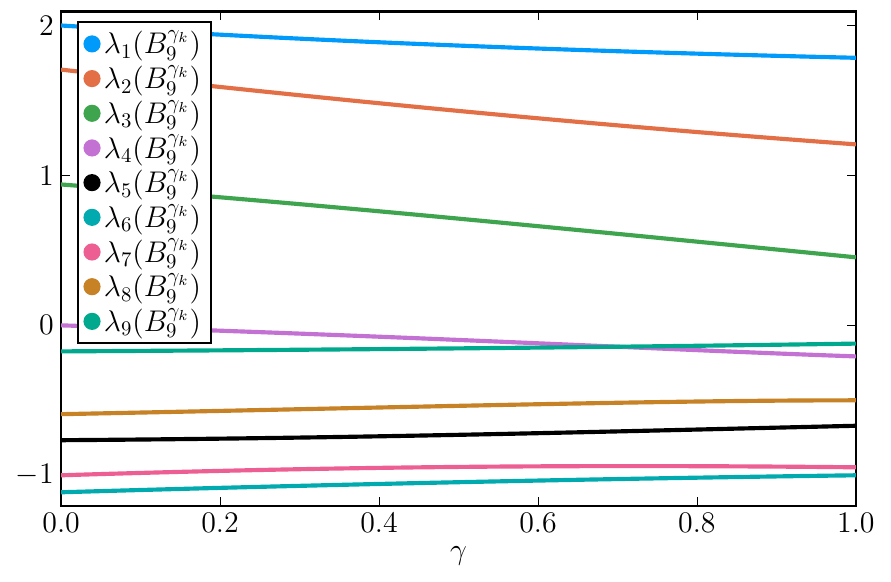}
\caption{[Ordering of eigenvalues (non-monotone symmetric symbol, $f(\theta)=\cos(\theta)+\cos(2\theta)$)] Eigenvalues for $B_9^{\gamma_k}$ with $\hat{T}_n(f)=T_{n,0,0}(f)$ (left) and $\hat{T}_n(f)=T_{n,1,1}(f)$ (right).}
\label{fig:sequence_n9}
\end{figure}

\subsection{Numerical verification 1 of Conjecture~\ref{conj:nonsym-eigval}}
We study the two non-symmetric real matrices $T_n(f_1)$ and $T_n(f_2)$ generated by $f_1(\theta)=1+\E^{\mathbf{i}\theta}$ and  $f_1(\theta)=1-\E^{\mathbf{i}\theta}$,
\begin{equation*}
T_n(f_1)=\left[
\begin{array}{rrrrrrrrr}
1\\
1&1\\
&\ddots&\ddots\\
&&1&1
\end{array}
\right],\qquad
T_n(f_2)=\left[
\begin{array}{rrrrrrrrr}
1\\
-1&1\\
&\ddots&\ddots\\
&&-1&1
\end{array}
\right].
\end{equation*}
Since the entries of $T_n(f)$ are real, we have $\sigma_j(T_n(f))=\sqrt{\lambda_j(T_n(f)^\textsc{t}T_n(f))}$, where here the corresponding positive definite matrices are
\begin{equation*}
T_n(f_1)^\textsc{t}T_n(f_1)=\left[
\begin{array}{rrrrrrrrr}
2&1\\
1&2&1\\
&\ddots&\ddots&\ddots\\
&&1&2&1\\
&&&1&1
\end{array}
\right],\qquad
T_n(f_2)^\textsc{t}T_n(f_2)=\left[
\begin{array}{rrrrrrrrr}
2&-1\\
-1&2&-1\\
&\ddots&\ddots&\ddots\\
&&-1&2&-1\\
&&&-1&1
\end{array}
\right],
\end{equation*}
with associated symbols $g_1(\theta)=|f_1(\theta)|^2=f_1(-\theta)f_1(\theta)=2+2\cos(\theta)$ and $g_2(\theta)=|f_2(\theta)|^2=f_2(-\theta)f_2(\theta)=2-2\cos(\theta)$. The matrix $T_n(f_1)^\textsc{t}T_n(f_1)$ is the matrix $T_{n,0,-1}(g_1)$, belonging to the $\tau_{0,-1}$-algebra, where the eigenvalues are given exactly by $g_1(\theta_{j,n}^{(0,-1)})$ where $\theta_{j,n}^{(0,-1)}=j\pi/(n+1/2)$. Hence, $\sigma_{\Pi_n^{-1}(j)}(T_n(f_1))=\sqrt{g_1(\theta_{j,n}^{(0,-1)})}$. Similarly, the singular values $\sigma_{\Pi_n^{-1}(j)}(T_n(f_2))=\sqrt{g_2(\theta_{j,n}^{(0,1)})}$, since $T_n(f_2)^\textsc{t}T_n(f_2)$ belongs to the $\tau_{0,1}$-algebra and $\theta_{j,n}^{(0,1)}=(j-1/2)\pi/(n+1/2)$.

Indeed, $\lambda_{\rho(j)}(H_{n}(f_1))$ and $(-1)^{j+1}\sqrt{g_1(\theta_{j,n}^{(0,-1)})}$ (and $\lambda_{\rho(j)}(H_{n}(f_2))$ and $(-1)^{j+1}\sqrt{g_2(\theta_{j,n}^{(0,1)})}$) match. Furthermore, we observe that $\lambda_{\rho(j)}(H_{n}(f_1))=-\lambda_{\rho(j)}(H_{n}(f_2))$.

\begin{proof}
We first study the symbol $f_1(\theta)=1+\E^{\mathbf{i}\theta}$, where
\begin{align}
H_n(f_1)&=
\left[
\begin{array}{rrrrrrrrr}
&&1&1\\
&\iddots&\iddots\\
1&1\\
1\\
\end{array}.
\right]
\end{align}
By a permutation matrix $P$, we have
\begin{align}
P^{-1}H_n(f_1)P&=
\left[
\begin{array}{rrrrrrrrr}
1&1\\
1&0&1\\
&\ddots&\ddots&\ddots\\
&&1&0&1\\
&&&1&0
\end{array}
\right].
\end{align}
This permuted matrix is the generated matrix $T_{n,1,0}(2\cos(\theta))$, by the $\tau_{1,0}$-algebra, and the eigenvalues are given exactly by,

\begin{align}
\lambda_j(P^{-1}H_n(f_1)P)&=2\cos(\theta_{j,n}^{1,0}),\qquad \theta_{j,n}^{1,0}=\frac{(j-1/2)\pi}{n+1/2}.
\end{align}
Note that the ordering, with this sampling grid, of these eigenvalues does not correspond to the ordering of the eigenvalues $\lambda_j(H_n(f_1))$, assuming Conjecture~\ref{conj:nonsym-eigval} is correct and the true ordering is given by
\begin{align}
\lambda_j(H_n(f_1))=(-1)^{j+1}\sqrt{2+2\cos\left(\frac{j\pi}{n+1/2}\right)}=(-1)^{j+1}2\cos\left(\frac{j\pi}{2n+1}\right), \quad j=1,\ldots,n.
\label{eq:nonsymnormaleigs}
\end{align}
Hence, we show that the set of samplings of \eqref{eq:nonsymnormaleigs} and
\begin{align}
2\cos\left(\frac{(j-1/2)\pi}{n+1/2}\right)=2\cos\left(\frac{(2j-1)\pi}{2n+1}\right), \quad j=1,\ldots,n,
\label{eq:nonsymnormaleigs2}
\end{align}
coincide (not the same order).

We note that for all odd $j$, the quantity in \eqref{eq:nonsymnormaleigs} is exactly
\begin{align}
2\cos\left(\frac{j\pi}{2n+1}\right),
\end{align}
which is equivalent to $j=1,\ldots,\lceil n/2\rceil$ of \eqref{eq:nonsymnormaleigs2}.

For even $j$, the quantity in \eqref{eq:nonsymnormaleigs} is the same as
\begin{align}
-2\cos\left(\frac{j\pi}{2n+1}\right),
\end{align}
which is equivalent to $j=n,n-1,\ldots,\lceil n/2\rceil+1$ of \eqref{eq:nonsymnormaleigs2}.

Now we study the symbol $f_2(\theta)=1-\E^{\mathbf{i}\theta}$, where
\begin{align}
H_n(f_2)&=
\left[
\begin{array}{rrrrrrrrr}
&&-1&1\\
&\iddots&\iddots\\
-1&1\\
1\\
\end{array}
\right],
\end{align}
and by a permutation matrix $P$,
\begin{align}
P_1^{-1}H_n(f_2)P_1&=
\left[
\begin{array}{ccccccc}
\left[\begin{array}{rr}
-1&1\\
1&0\\
\end{array}\right]&
\left[\begin{array}{rr}
0&0\\
-1&0\\
\end{array}\right]\\
\left[\begin{array}{rr}
0&-1\\
0&0\\
\end{array}\right]&
\left[\begin{array}{rr}
0&1\\
\phantom{-}1&0\\
\end{array}\right]&
\left[\begin{array}{rr}
0&0\\
-1&0\\
\end{array}\right]\\
&\ddots&\ddots&\ddots
\end{array}
\right].
\end{align}
The matrix-valued symbol of this matrix is
\begin{align}
f_p(\theta)&=\left[\begin{array}{rrrr}
0&1-\E^{\mathbf{i}\theta}\\
1-\E^{-\mathbf{i}\theta}&0
\end{array}
\right],
\end{align}
which can be split into the two eigenvalue functions
\begin{align}
f_p^{(1)}&=\sqrt{2-2\cos(\theta)}=2\sin(\theta/2)\\
f_p^{(2)}&=-\sqrt{2-2\cos(\theta)}=-2\sin(\theta/2)
\end{align}
By direct inspection we find
\begin{align}
\lambda_j(P^{-1}H_n(f_1)P)&=\begin{cases}
f_p^{(1)}(\theta_{\hat{\jmath},\lceil n/2\rceil}^{(1)}), &j \text{ odd},\\
f_p^{(2)}(\theta_{\hat{\jmath},\lfloor n/2\rfloor}^{(2)}), &j \text{ even},\\
\end{cases}\qquad \hat{\jmath}=\lceil j/2\rceil,\\
\theta_{\hat{\jmath},\lceil n/2\rceil}^{(1)}&=\frac{(2\hat{\jmath}-3/2)\pi}{n+1/2},\quad \hat{\jmath}=1,\ldots,\lceil n/2\rceil,\\
\theta_{\hat{\jmath},\lfloor n/2\rfloor}^{(2)}&=\frac{(2\hat{\jmath}-1/2)\pi}{n+1/2}, \quad \hat{\jmath}=1,\ldots,\lfloor n/2\rfloor,
\end{align}
which is equivalent to
\begin{align}
=(-1)^{j+1}\sqrt{2-2\cos\left(\frac{(j-1/2)\pi}{n+1/2}\right)}=(-1)^{j+1}2\sin\left(\frac{(j-1/2)\pi}{2n+1}\right).
\end{align}

\qed
\end{proof}

\subsection{Numerical verification 2 of Conjecture~\ref{conj:nonsym-eigval}}

The generating symbol for the Grcar matrix~\cite{trefethen1991pseudospectra} is $f(\theta)=-\E^{\mathbf{i}\theta}+1+\E^{-\mathbf{i}\theta}+\E^{-2\mathbf{i}\theta}+\E^{-3\mathbf{i}\theta}$.
Since we are interested in the singular values $\sigma_j(T_n(f))$ we will now work with the modulus of the symbol $|f(\theta)|$ (left panel in Figure~\ref{fig:nonsym_grcar_symbols}), or more precisely taking the square root of the eigenvalues of the normal matrix $(T_n(f))^\textsc{t}T_n(f)$ which has the symbol $g(\theta)=f(-\theta)f(\theta)=5+4\cos(\theta)+2\cos(2\theta)-2\cos(4\theta)$ (right panel in Figure~\ref{fig:nonsym_grcar_symbols}). The reason for this is that we can construct the matrix $\hat{T}_n(g)$ needed in Algorithm~\ref{algo:ordering}.

\begin{figure}[H]
\centering
\includegraphics[width=0.48\textwidth]{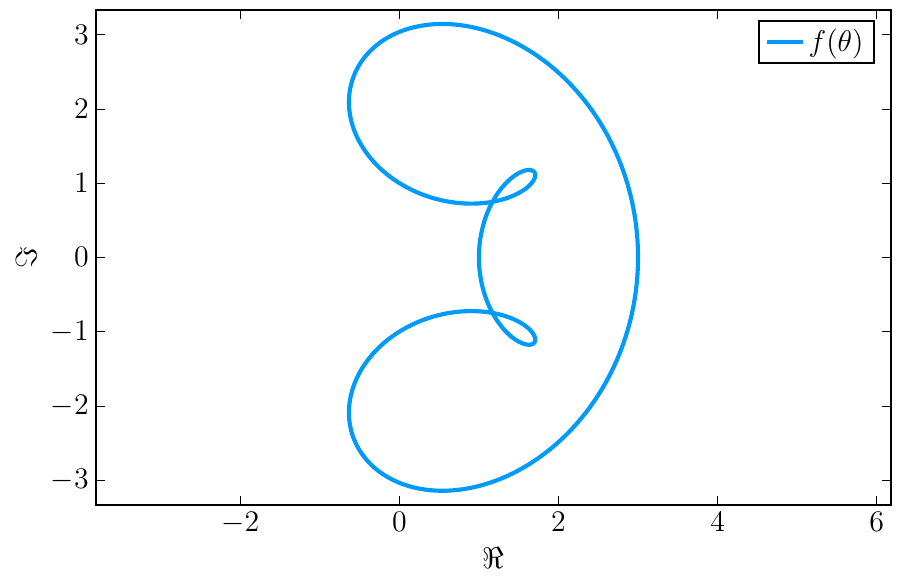}
\includegraphics[width=0.48\textwidth]{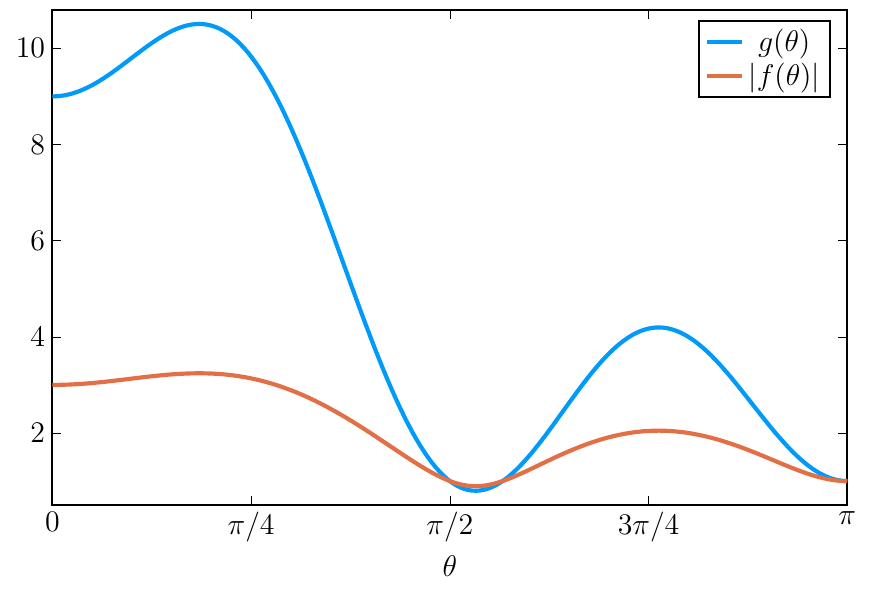}
\caption{[Ordering of singular values (non-monotone non-symmetric symbol, $f(\theta)=-\E^{\mathbf{i}\theta}+1+\E^{-\mathbf{i}\theta}+\E^{-2\mathbf{i}\theta}+\E^{-3\mathbf{i}\theta}$)] Left: Complex valued symbol $f(\theta)$. Right: $g(\theta)=f(-\theta)f(\theta)$ and $|f(\theta)|=\sqrt{g(\theta)}$ on $\theta\in[0,\pi]$.}
\label{fig:nonsym_grcar_symbols}
\end{figure}

Below we show the matrices needed for Algorithm~\ref{algo:ordering} where $\hat{T}_n(g)=T_{n,0,0}(g)$ and the target matrix is $(T_n(f))^\textsc{t}T_n(f)$. Hence, $\hat{T}_n(g)=(T_n(f))^\textsc{t}T_n(f)+R_n$,
\begin{align*}
(T_n(f))^\textsc{t}T_n(f)&=\left[
\begin{array}{rrrrrrrrrrrrrr}
2 & 0 & 0 & 0 & -1  \\
0 & 3 & 1 & 1 & 0 & -1  \\
0 & 1 & 4 & 2 & 1 & 0 & -1  \\
0 & 1 & 2 & 5 & 2 & 1 & 0 & -1  \\
-1 & 0 & 1 & 2 & 5 & 2 & 1 & 0 & -1  \\
& \ddots & \ddots & \ddots & \ddots & \ddots & \ddots & \ddots & \ddots & \ddots  \\
&& -1 & 0 & 1 & 2 & 5 & 2 & 1 & 0 & -1 \\
&&& -1 & 0 & 1 & 2 & 5 & 2 & 1 & 0 \\
&&&& -1 & 0 & 1 & 2 & 5 & 2 & 1 \\
&&&&& -1 & 0 & 1 & 2 & 5 & 2 \\
&&&&&& -1 & 0 & 1 & 2 & 4 \\
\end{array}
\right]\\
\hat{T}_n(g)=T_{n,0,0}(g)&=\left[\begin{array}{rrrrrrrrrrrrrrr}
4 & 2 & 2 & 0 & -1  \\
2 & 6 & 2 & 1 & 0 & -1  \\
2 & 2 & 5 & 2 & 1 & 0 & -1  \\
0 & 1 & 2 & 5 & 2 & 1 & 0 & -1  \\
-1 & 0 & 1 & 2 & 5 & 2 & 1 & 0 & -1  \\
& \ddots & \ddots & \ddots & \ddots & \ddots & \ddots & \ddots & \ddots & \ddots  \\
&& -1 & 0 & 1 & 2 & 5 & 2 & 1 & 0 & -1 \\
&&& -1 & 0 & 1 & 2 & 5 & 2 & 1 & 0 \\
&&&& -1 & 0 & 1 & 2 & 5 & 2 & 2 \\
&&&&& -1 & 0 & 1 & 2 & 6 & 2 \\
&&&&&& -1 & 0 & 2 & 2 & 4 \\
\end{array}
\right]\\
R_n&=\left[
\begin{array}{rrrrrrrrrrrrrrr}
2 & 2 & 2  \\
2 & 3 & 1  \\
2 & 1 & 1  \\
\\
&&& & & & 1 \\
&&& & & 1 &\\
&&& & 1 & &
\end{array}.
\right]
\end{align*}

In Figure~\ref{fig:eig_nonsym1} we show the square root of the eigenvalues (such that, the sequence yield the singular values of the true target matrix $T_n(f)$), for $n=10$, for all $B_n^{(\gamma_k)}$, with $N_{steps}=100$. The numbering in the figure is after the permutation $\tilde{\Pi}_n^{-1}(j)$.
The yellow boxes indicate where it is visible that Algorithm~\ref{algo:ordering} fails to swap singular values $\sigma_5$ and $\sigma_{10}$, two times. However, since the swap fails twice the resulting ordering is correct. The blue circle indicates an erroneous ordering, when comparing with the signs of $\lambda_j(H_n(f))$, as indicated in Table~\ref{tab:nonsym_n10}
\begin{figure}[H]
\centering
\includegraphics[width=0.48\textwidth]{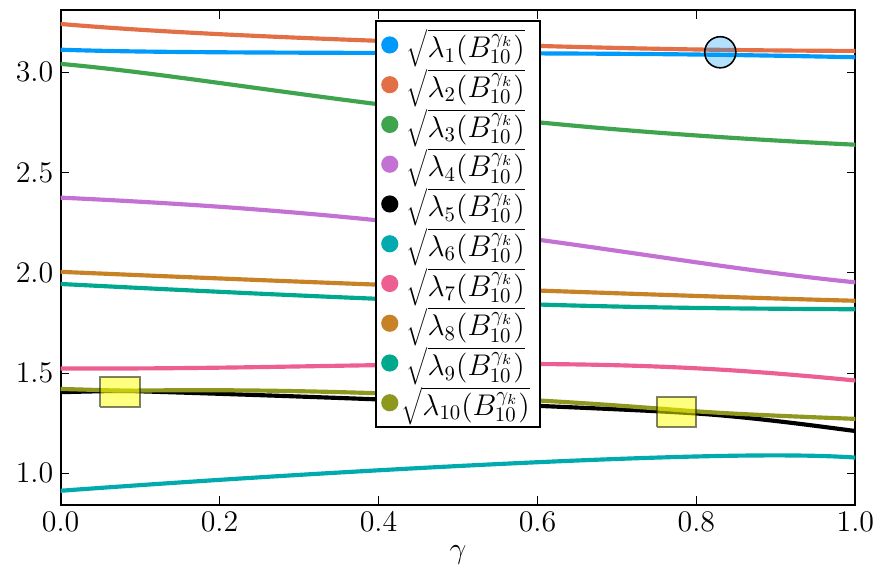}
\caption{[Ordering of singular values (non-monotone non-symmetric symbol, $f(\theta)=-\E^{\mathbf{i}\theta}+1+\E^{-\mathbf{i}\theta}+\E^{-2\mathbf{i}\theta}+\E^{-3\mathbf{i}\theta}$)] $g(\theta)=f(\theta)f(-\theta)$, and $\hat{T}_{n}(g)=T_{n,0,0}(g)$.}
\label{fig:eig_nonsym1}
\end{figure}

In Table~\ref{tab:nonsym_n10} we see that indeed $\sigma_5$  and $\sigma_{10}$ can be assumed to be correctly ordered. However, $\sigma_1$ and $\sigma_2$ are wrongly ordered, if Conjecture~\ref{conj:nonsym-eigval} is correct. Increasing $N_{steps}$ to a higher number does not seems to remedy this discrepency.

\begin{table}[H]
\centering
\caption{[Ordering of eigenvalues (non-monotone non-symmetric symbol, $f(\theta)=-\E^{\mathbf{i}\theta}+1+\E^{-\mathbf{i}\theta}+\E^{-2\mathbf{i}\theta}+\E^{-3\mathbf{i}\theta}$)] $\hat{T}_{n}(f)=T_{n,0,0}(f)$.}
\label{tab:nonsym_n10}

\begin{tabular}{rrrrrrrrrrrr}
\toprule
$j$&$\tilde{\Pi}_n(j)$&$\tilde{\Pi}_n^{-1}(j)$&$|f(\theta_{j,n})|$&$\sigma_{\tilde{\Pi}_n^{-1}(j)}(T_n(f))$ & $(-1)^{j+1}\sigma_{\tilde{\Pi}_n^{-1}(j)}(T_n(f))$& $\lambda_{\tilde{\rho}(j)}(H_n(f))$\\
\midrule
1 & 6 & 9 & 3.1128 & 3.0752 & {\color{myblue}3.0752} & -3.0752 \\
2 & 5 & 10 & 3.2412 & 3.1066 & {\color{myblue}-3.1066} & 3.1066 \\
3 & 10 & 8 & 3.0420 & 2.6384 & 2.6384 & 2.6384 \\
4 & 7 & 7 & 2.3741 & 1.9512 & -1.9512 & -1.9512 \\
5 & 9 & 2 & 1.4028 & 1.2089 & 1.2089 & 1.2089 \\
6 & 8 & 1 & 0.9106 & 1.0765 & -1.0765 & -1.0765 \\
7 & 4 & 4 & 1.5209 & 1.4612 & 1.4612 & 1.4612 \\
8 & 3 & 6 & 2.0037 & 1.8592 & -1.8592 & -1.8592 \\
9 & 1 & 5 & 1.9431 & 1.8166 & 1.8166 & 1.8166 \\
10 & 2 & 3 & 1.4191 & 1.2696 & -1.2696 & -1.2696 \\
\bottomrule
\end{tabular}
\end{table}

In Figure~\ref{fig:eigv_nonsym2} we report the ten eigenvector element sequences shown, given by Algorithm~\ref{algo:ordering}. One can clearly see the erratic behavior in eigenvectors five and ten, as previously indicated in Figure~\ref{fig:eig_nonsym1}. However, eigenvectors one and two are visually correct.

\begin{figure}[H]
\centering
\includegraphics[width=0.4\textwidth]{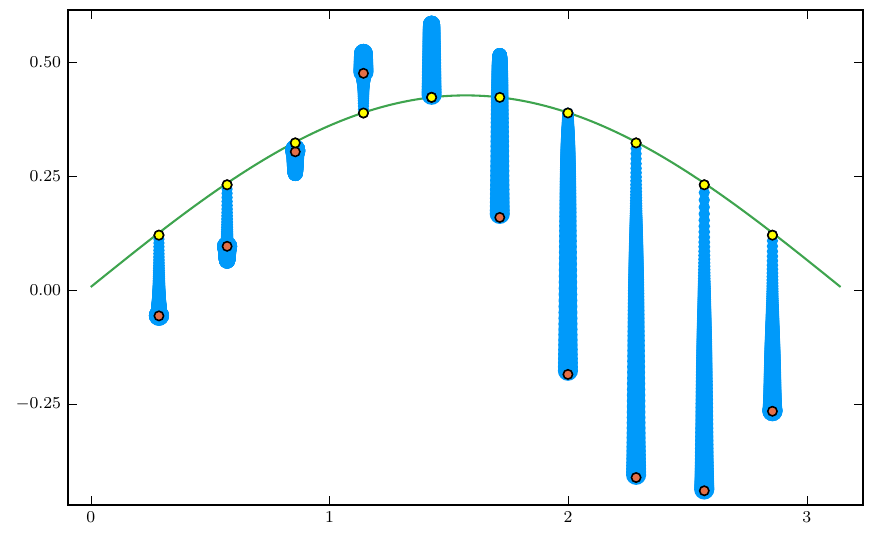}
\includegraphics[width=0.4\textwidth]{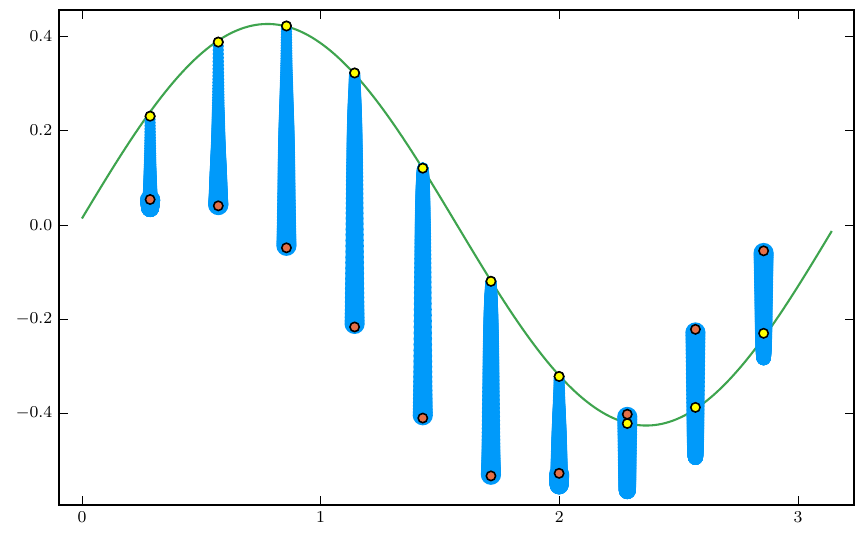}

\includegraphics[width=0.4\textwidth]{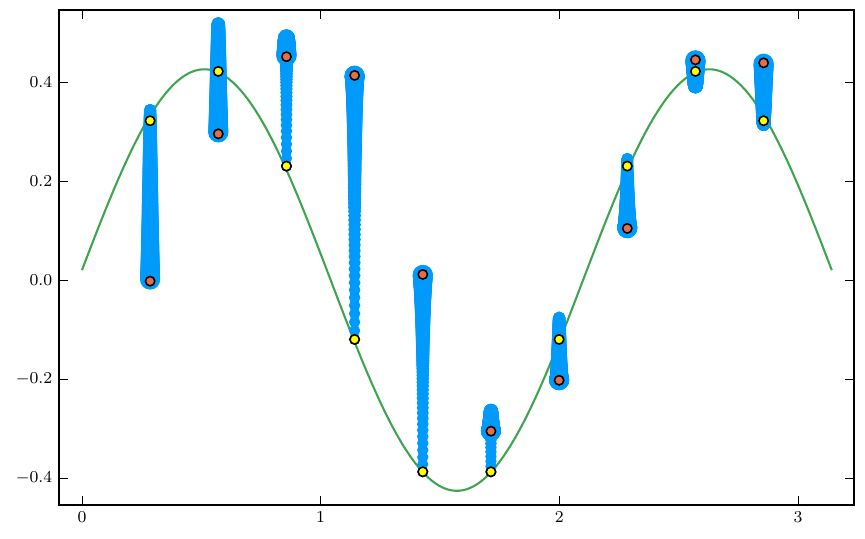}
\includegraphics[width=0.4\textwidth]{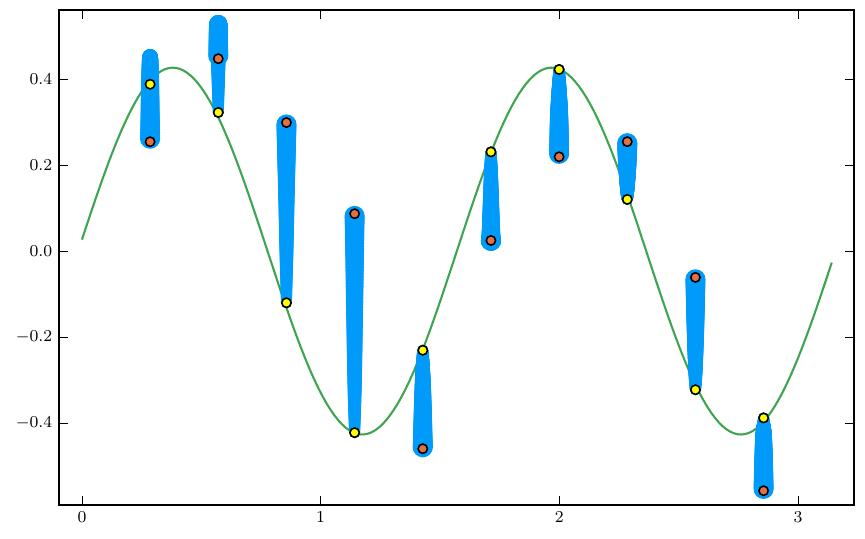}

\includegraphics[width=0.4\textwidth]{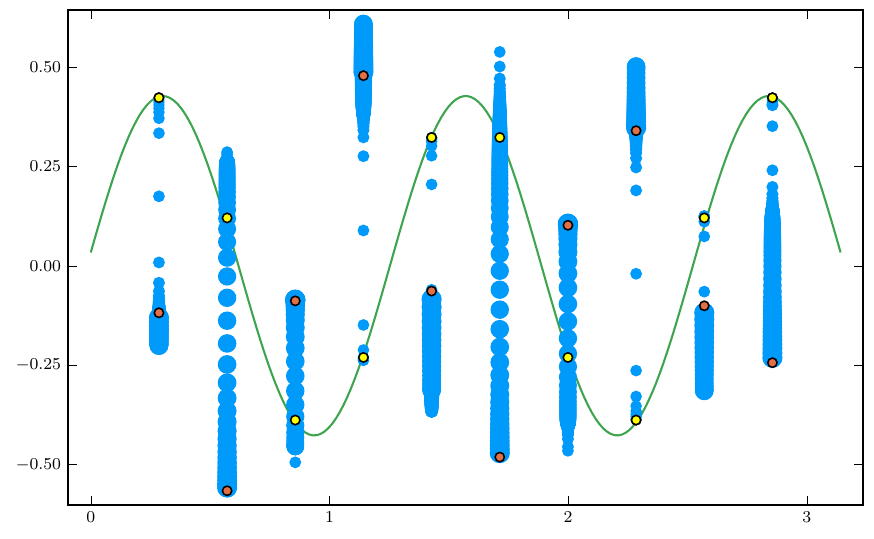}
\includegraphics[width=0.4\textwidth]{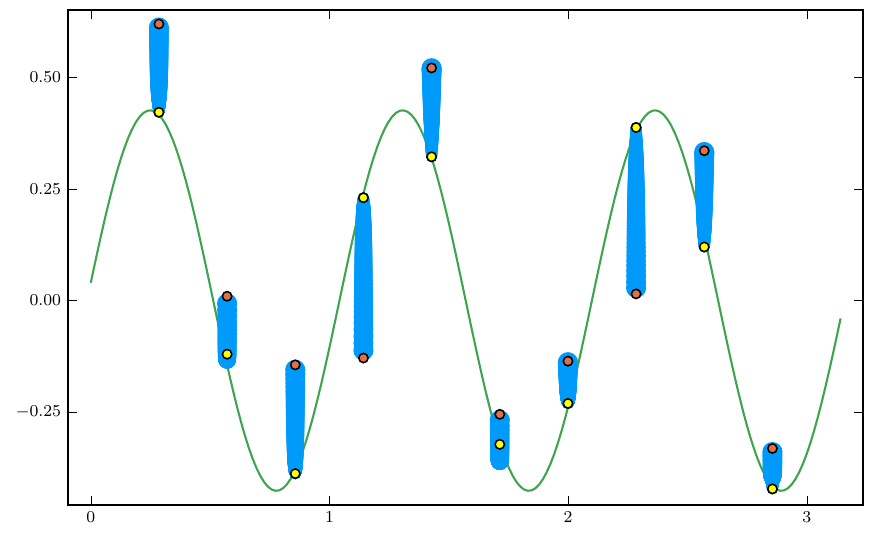}

\includegraphics[width=0.4\textwidth]{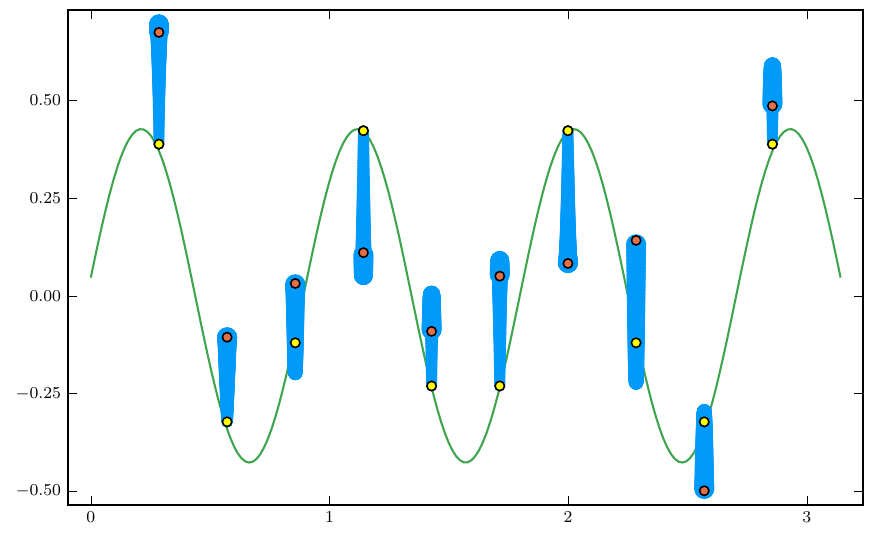}
\includegraphics[width=0.4\textwidth]{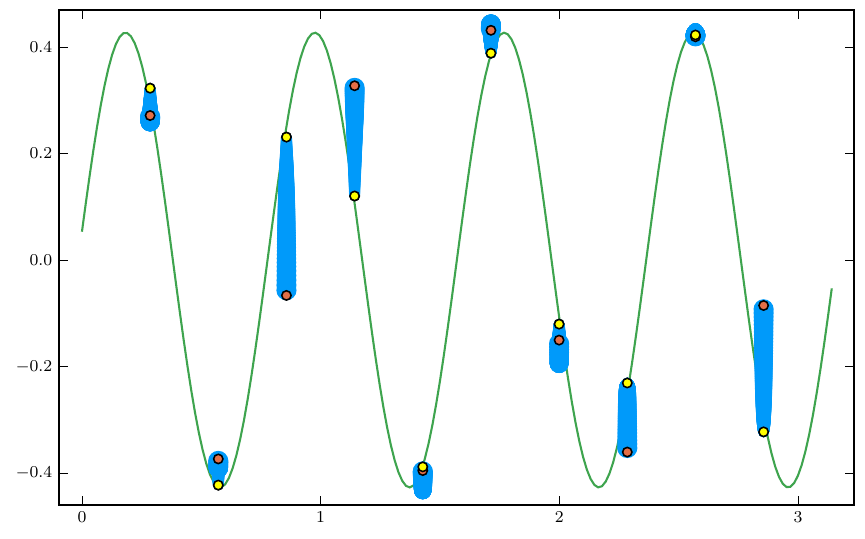}

\includegraphics[width=0.4\textwidth]{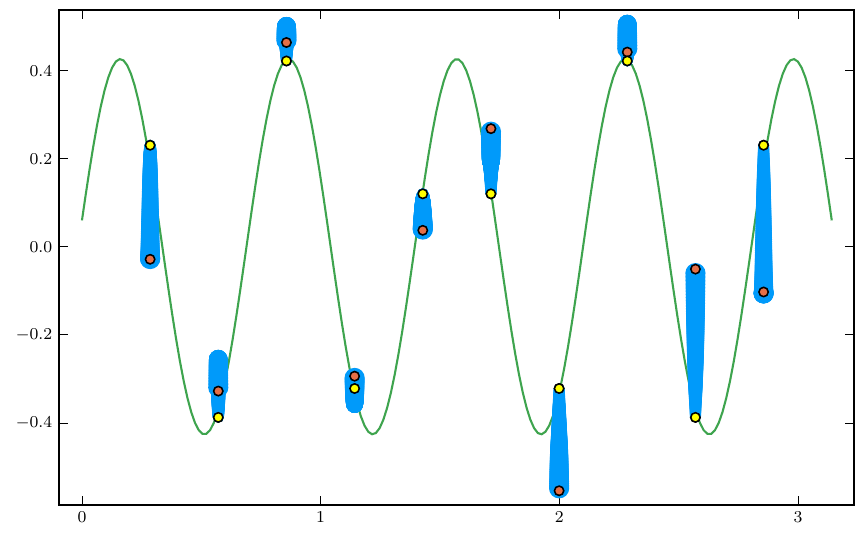}
\includegraphics[width=0.4\textwidth]{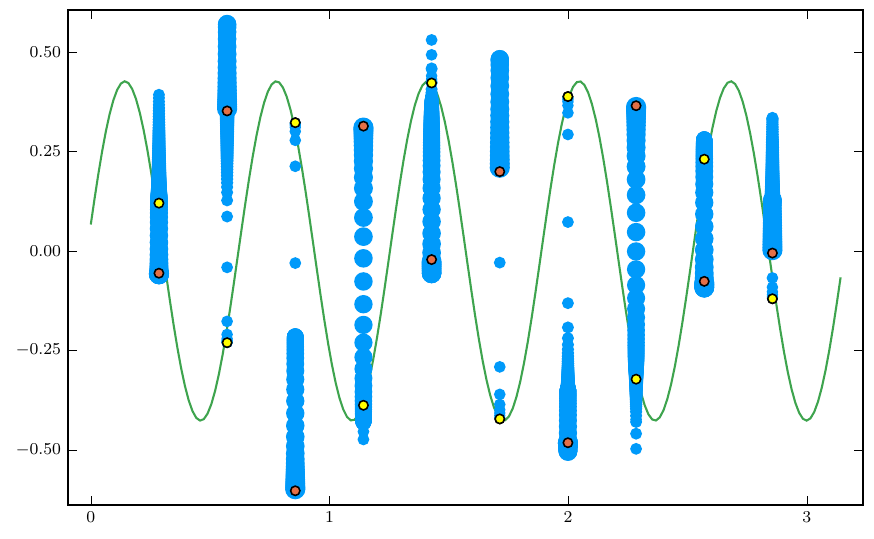}
\caption{[Ordering of singular values (non-monotone non-symmetric symbol, $f(\theta)=-\E^{\mathbf{i}\theta}+1+\E^{-\mathbf{i}\theta}+\E^{-2\mathbf{i}\theta}+\E^{-3\mathbf{i}\theta}$)] All ten eigenvector sequences, except the fifth and and tenth, seem to be continuous and correct.}
\label{fig:eigv_nonsym2}
\end{figure}
In Figure~\ref{fig:nonsym_n10_allThat} we show the sequences of $B_n^{(\gamma_k)}$ for all combinations of $\hat{T}_{n}(f)=T_{n,\varepsilon,\varphi}(f)$, $\varepsilon,\varphi=\{-1,0,-1\}$. In Table~\ref{tab:nonsym_n10_allThat} are shown the actual yielded orderings. As it can be seen, they all exhibit errors, and the different $\hat{T}_n(f)$ have different bias to the initial ordering, and subsequent result. In Table~\ref{tab:nonsym_n10_allThat2} the corrected orderings are displayed and, given the used Algorithm~\ref{algo:ordering}, four acceptable versions are given.
\begin{figure}[H]
\centering
\includegraphics[width=0.32\textwidth]{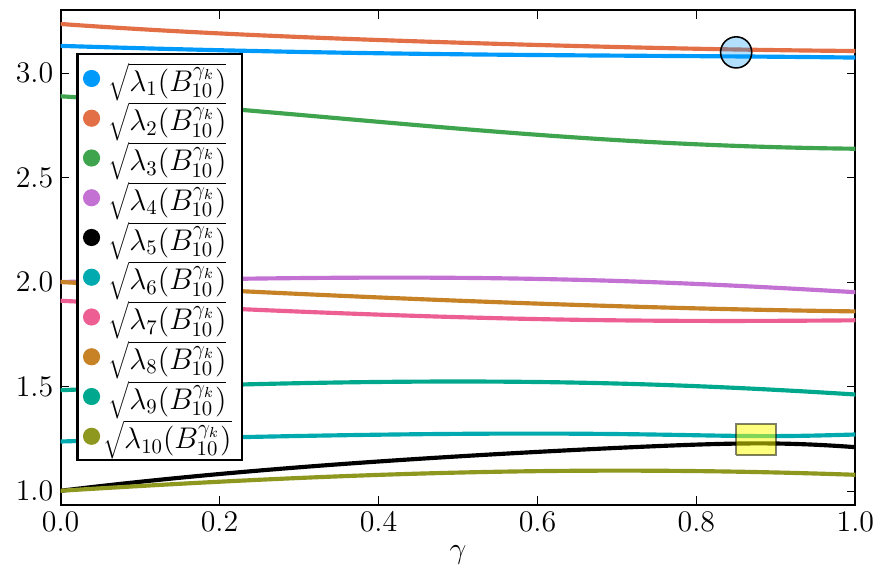}
\includegraphics[width=0.32\textwidth]{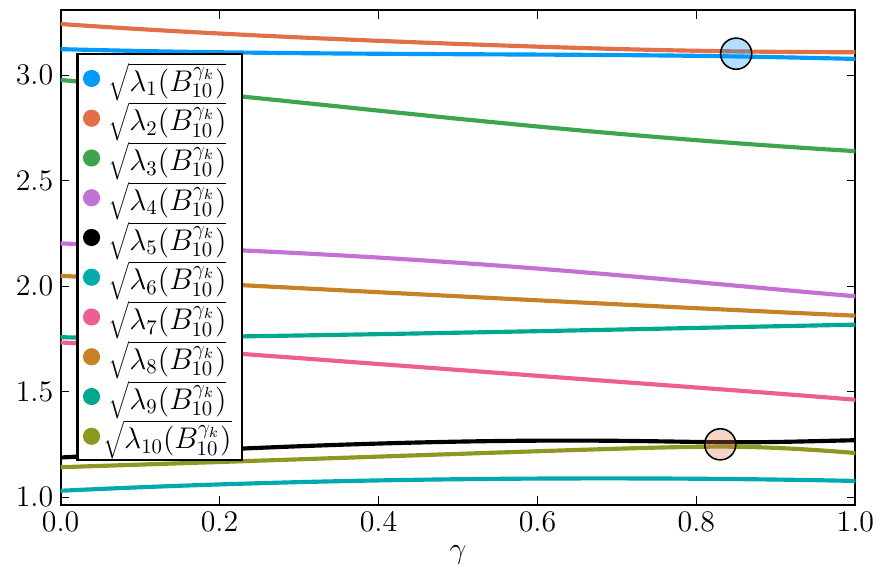}
\includegraphics[width=0.32\textwidth]{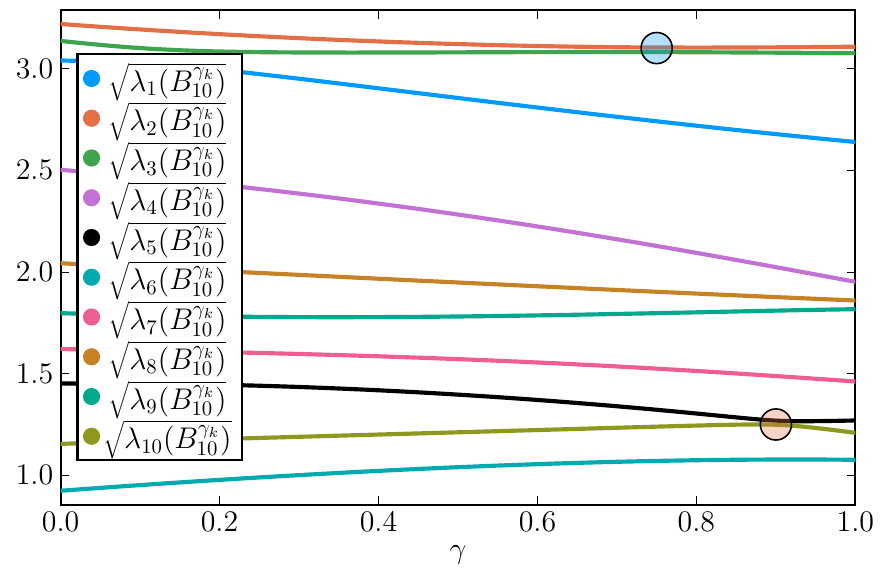}

\includegraphics[width=0.32\textwidth]{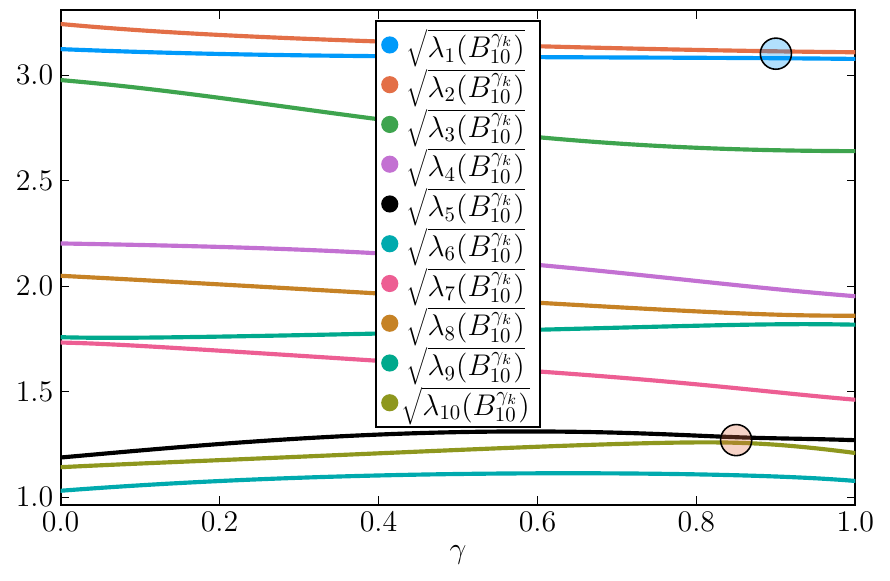}
\includegraphics[width=0.32\textwidth]{exmp0_e0_p0_n10_100.pdf}
\includegraphics[width=0.32\textwidth]{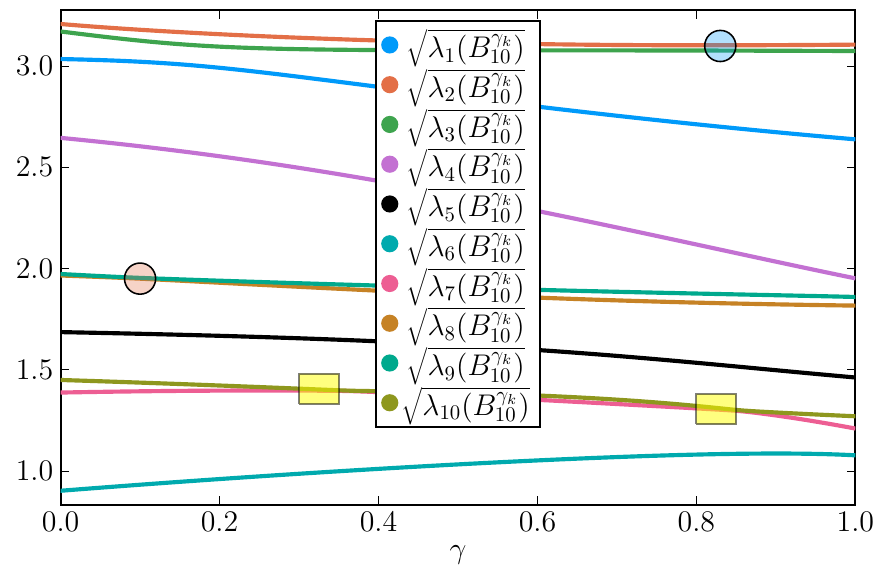}

\includegraphics[width=0.32\textwidth]{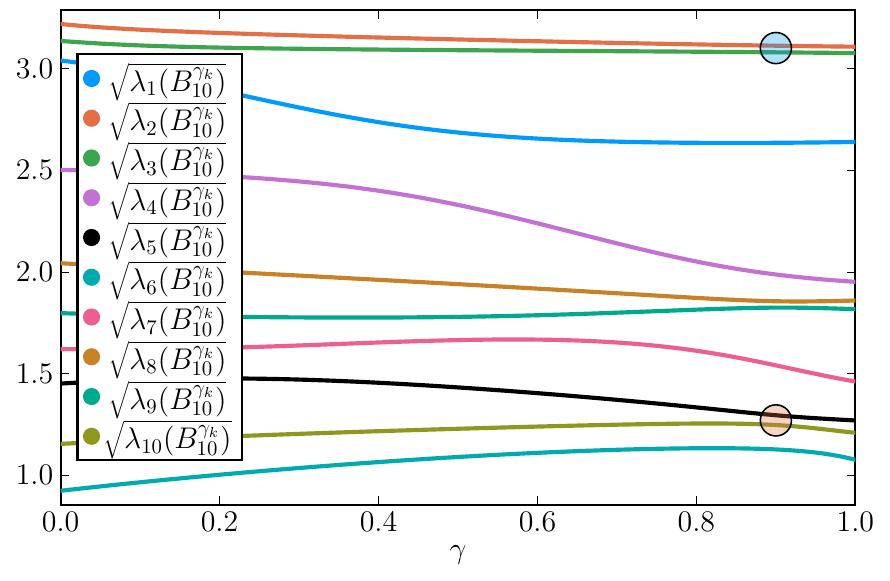}
\includegraphics[width=0.32\textwidth]{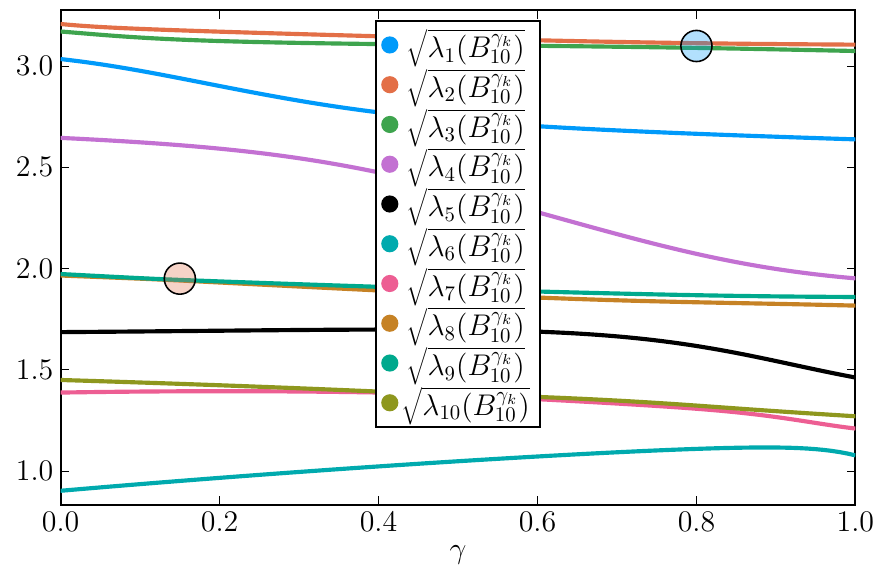}
\includegraphics[width=0.32\textwidth]{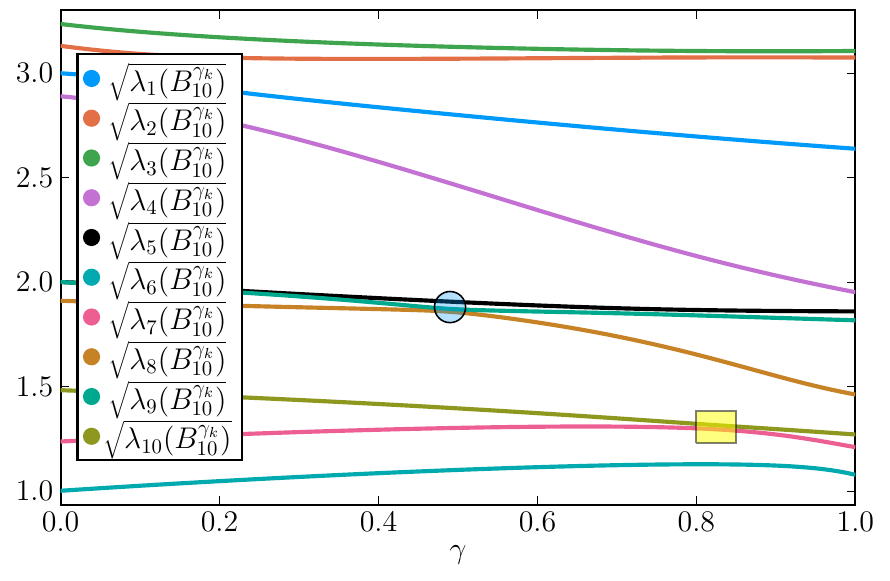}
\caption{[Ordering of eigenvalues (non-monotone non-symmetric symbol, $f(\theta)=-\E^{\mathbf{i}\theta}+1+\E^{-\mathbf{i}\theta}+\E^{-2\mathbf{i}\theta}+\E^{-3\mathbf{i}\theta}$)] $\hat{T}_{n}(f)=T_{n,\varepsilon,\varphi}(f)$ for all combinations $\varepsilon,\varphi=\{-1,0,-1\}$.}
\label{fig:nonsym_n10_allThat}
\end{figure}

\begin{table}[H]
\caption{[Ordering of eigenvalues (non-monotone non-symmetric symbol, $f(\theta)=-\E^{\mathbf{i}\theta}+1+\E^{-\mathbf{i}\theta}+\E^{-2\mathbf{i}\theta}+\E^{-3\mathbf{i}\theta}$)] $\hat{T}_{n}(f)=T_{n,\varepsilon,\varphi}(f)$ for all combinations $\varepsilon,\varphi=\{-1,0,-1\}$. Coloring indicate erroneous ordering, with same color as corresponding panel in Figure~\ref{fig:nonsym_n10_allThat}.}
\label{tab:nonsym_n10_allThat}
\begin{equation*}
\begin{array}{r|rrrrrrrrr}
\toprule
\lambda_j(H_n(f))&(-1,-1)&(-1,0)&(-1,1)&(0,-1)&(0,0)&(0,1)&(1,-1)&(1,0)&(1,1)\\
\midrule
-3.0752&{\color{myblue}3.0752} & {\color{myblue}3.0752} & 2.6384 & {\color{myblue}3.0752} & {\color{myblue}3.0752} & 2.6384 & 2.6384 & 2.6384 & 2.6384 \\
-1.9512&{\color{myblue}-3.1066} & {\color{myblue}-3.1066} & {\color{myblue}-3.1066} & {\color{myblue}-3.1066} & {\color{myblue}-3.1066} & {\color{myblue}-3.1066} & {\color{myblue}-3.1066} & {\color{myblue}-3.1066} & -3.0752 \\
-1.8592&2.6384 & 2.6384 & {\color{myblue}3.0752} & 2.6384 & 2.6384 & {\color{myblue}3.0752} & {\color{myblue}3.0752} & {\color{myblue}3.0752} & 3.1066 \\
-1.2696&-1.9512 & -1.9512 & -1.9512 & -1.9512 & -1.9512 & -1.9512 & -1.9512 & -1.9512 & -1.9512 \\
-1.0765&1.2089 & {\color{myred}1.2696} & {\color{myred}1.2696} & {\color{myred}1.2696} & 1.2089 & 1.4612 & {\color{myred}1.2696} & 1.4612 & {\color{myblue}1.8592} \\
1.2089&-1.2696 & -1.0765 & -1.0765 & -1.0765 & -1.0765 & -1.0765 & -1.0765 & -1.0765 & -1.0765 \\
1.4612&1.8166 & 1.4612 & 1.4612 & 1.4612 & 1.4612 & 1.2089 & 1.4612 & 1.2089 & 1.2089 \\
1.8166&-1.8592 & -1.8592 & -1.8592 & -1.8592 & -1.8592 & {\color{myred}-1.8166} & -1.8592 & {\color{myred}-1.8166} & {\color{myblue}-1.4612} \\
2.6384&1.4612 & 1.8166 & 1.8166 & 1.8166 & 1.8166 & {\color{myred}1.8592} & 1.8166 & {\color{myred}1.8592} & 1.8166 \\
3.1066&-1.0765 & {\color{myred}-1.2089} & {\color{myred}-1.2089} & {\color{myred}-1.2089} & -1.2696 & -1.2696 & {\color{myred}-1.2089} & -1.2696 & -1.2696 \\
\bottomrule
\end{array}
\end{equation*}
\end{table}

\begin{table}[H]
\caption{[Ordering of eigenvalues (non-monotone non-symmetric symbol, $f(\theta)=-\E^{\mathbf{i}\theta}+1+\E^{-\mathbf{i}\theta}+\E^{-2\mathbf{i}\theta}+\E^{-3\mathbf{i}\theta}$)] $\hat{T}_{n}(f)=T_{n,\varepsilon,\varphi}(f)$ for all combinations $\varepsilon,\varphi=\{-1,0,-1\}$. Reordering to correct orderings in Table~\ref{tab:nonsym_n10_allThat}.}
\label{tab:nonsym_n10_allThat2}
\begin{equation*}
\begin{array}{r|rrrrrrrrr}
\toprule
\lambda_j(H_n(f))&(-1,-1)&(-1,0)&(-1,1)&(0,1)\\
&&(0,-1)&(1,-1)&(1,0)\\
&&(0,0)&&(1,1)\\
\midrule
-3.0752& 3.1066& 3.1066 & 2.6384  & 2.6384 \\
-1.9512&-3.0752 & -3.0752 & -3.0752  & -3.0752 \\
-1.8592&2.6384 & 2.6384 & 3.1066 &  3.1066  \\
-1.2696&-1.9512 & -1.9512 & -1.9512  & -1.9512  \\
-1.0765&1.2089 & 1.2089 & 1.2089  & 1.4612  \\
1.2089&-1.2696 & -1.0765 & -1.0765 &  -1.0765 \\
1.4612&1.8166 & 1.4612 & 1.4612 & 1.2089  \\
1.8166&-1.8592 & -1.8592 & -1.8592 & -1.8592 \\
2.6384&1.4612 & 1.8166 & 1.8166 &  1.8166  \\
3.1066&-1.0765 & -1.2696 & -1.2696  & -1.2696 \\
\bottomrule
\end{array}
\end{equation*}
\end{table}

A potential reason for the failure of Algorithm~\ref{algo:ordering} in this example, is that $R_n$ has too many non-zero entries. A possible remedy it that type of situation could be to split up $R_n$ to multiple matrices $R_n^{(i)}$, and have multiple $\gamma_k^{(i)}$, such that,
\begin{align}
T_n(f)=\hat{T}_n(f)-\gamma_k^{(1)}R_n^{(1)}-\gamma_k^{(2)}R_n^{(2)}-\ldots -\gamma_k^{(s)}R_n^{(s)},
\end{align}
where $T_n(f)$ is the target matrix and $\hat{T}_n(f)$ is the matrix with known eigendecomposition.
First let all $\gamma_k^{(i)}$ be zero, increase $\gamma_k^{(1)}$ to one, then $\gamma_k^{(2)}$, and so on.
Another approach to find the true ordering in a case like this to generate a sequence of grids, and then use a matrix-less method, and if the result is non-erratic it can be assumed to be the correct grid.

\section{Conclusions}\label{sec:end}

In a series of recent papers the spectral behavior of the matrix sequence $\{Y_nT_n(f)\}$ has been studied in the sense of the spectral distribution, with  the generating function $f$ being Lebesgue integrable and with real Fourier coefficients. This kind of study was also motivated by computational purposes for the solution of the related large linear systems using the (preconditioned) MINRES algorithm and for the extension of applicability of eigenvalue matrix-less algorithms. Here we have developed further tools, by exploiting also algebraic results such as the Cantoni-Butler Theorem. Indeed when $f$ is real-valued we have proved that
\begin{itemize}
\item $\lambda_j(H_n(f))=(-1)^{j+1}\lambda_j(T_n(f))$, $j=1,\ldots,n$;
\item $\lambda_j(T_n(f)) = f(\xi_{j,n})$, $j=1,\ldots,n$ and $\{\xi_{j,n}\}$ a.u. on $[0,\pi]$, if $f$ is Riemann integrable, with connected range  and with a finite number of local minima, maxima, and discontinuity points;
\item $\lambda_j(T_n(f)) = f(\xi_{j,n})+\psi_{j,n}$, $j=1,\ldots,n$, $\psi_{j,n}$ infinitesimal in $n$, and $\{\xi_{j,n}\}$ a.u. on $[0,\pi]$, if $f$ is Riemann integrable, with connected range and when dropping the restriction on the finite number of local minima, maxima, and discontinuity points.
\end{itemize}
We have also reported further distribution and localization results which are consequences of the above items.

When $f$ is complex-valued, but still with real Fourier coefficients, the same type of localization and distributional findings are obtained, but in connection with $|f|$ and with the singular values of $T_n(f)$.

Several numerical experiments have been reported for giving a visual evidence of the numerical results and for showing better approximations of the spectra of $T_n(f)$ and $H_n(f)$, with the idea of extending the matrix-less procedures to the more challenging setting in which the generating function $f$ is non-monotone.

\bibliographystyle{siam}
\bibliography{hankel}
\end{document}